\documentclass[1 [leqno, 11pt]{amsart}
\usepackage{amssymb,  amsmath, latexsym, amssymb, amsfonts, amsbsy, amsthm,mathtools, graphicx, CJKutf8, CJKnumb, CJKulem, extarrows, color, dsfont,}
\usepackage{subcaption}
\usepackage{graphics}
\usepackage{epsfig}
\usepackage{graphicx}

\graphicspath{{./EPS/}}

\numberwithin{equation}{section}

\allowdisplaybreaks

\let\al=\alpha
\let\b=\beta

\let\f=\frac

\let\ep=\epsilon

\let\na=\nabla

\let\pa=\partial

\newcommand{\beq}{\begin{equation}}
\newcommand{\eeq}{\end{equation}}
\newcommand{\ben}{\begin{eqnarray}}
\newcommand{\een}{\end{eqnarray}}
\newcommand{\beno}{\begin{eqnarray*}}
\newcommand{\eeno}{\end{eqnarray*}}
\newcommand{\bagn}{\begin{align}}
\newcommand{\eagn}{\end{align}}

\newtheorem{theorem}{Theorem}[section]
\newtheorem{definition}[theorem]{Definition}
\newtheorem{lemma}[theorem]{Lemma}
\newtheorem{proposition}[theorem]{Proposition}
\newtheorem{Corollary}[theorem]{Corollary}
\newtheorem{remark}[theorem]{Remark}
\newtheorem{Theorem}{Theorem}[section]

\newtheorem{Proposition}[Theorem]{Proposition}

\begin{document}

\title[Linear inviscid damping]{Linear inviscid damping in the presence of an embedding eigenvalue}

\author[S. Ren]{Siqi Ren}

\date{\today}

\address{Department of Applied Mathematics, Zhejiang University of Technology,  310032, Hangzhou, P. R. China}
\email{sirrenmath@zjut.edu.cn}

\author[Z. Zhang]{Zhifei Zhang}
\address{School of Mathematical Sciences, Peking University, 100871, Beijing, P. R. China}
\email{zfzhang@math.pku.edu.cn}

\maketitle

\begin{abstract}
In this paper, we investigate the long-time dynamics of the linearized 2-D Euler equations  around a hyperbolic tangent  flow $(\tanh y,0)$.  A key difference compared to previous results is that the linearized operator has an embedding eigenvalue, which has a significant impact on the dynamics of the linearized system.  For the first mode, the dynamics consists of there parts: non-decay part related to the eigenspace associated with the embedding eigenvalue,  slow decay part due to the resolvent singularity,
and fast decay part related to the inviscid damping. For higher modes, the dynamics is similar to the inviscid damping phenomena in the case without embedding eigenvalues.
\end{abstract}

\section{Introduction}
In this paper, we consider the 2-D Euler equations in a domain $\mathbb{T}\times  \mathbb{R}$:
\begin{align}\label{eq:Euler}
{\pa_t\textbf{U}}+{\textbf{U}}\cdot\nabla {\textbf{U}}+\nabla P=0,\quad \nabla\cdot {\textbf{U}}=0.
\end{align}
Here $\textbf{U}(t,x,y)=(U^1, U^2)$ is the velocity and $P(t,x,y)$ is the pressure.
In terms of the vorticity $W=\pa_xU^2-\pa_yU^1$, the 2-D Euler equations take as follows
\beno
\pa_tW+\textbf{U}\cdot \na W=0.
\eeno

The shear flow $(u(y),0)$ is a steady solution of \eqref{eq:Euler}. In this paper, we are concerned with the long-time dynamics of the linearized Euler equations around this flow, which takes in terms of the vorticity as follows
\begin{align}
\begin{split}
\label{eq:Euler-linearize}
&\pa_t\omega+u\pa_x\omega+u''\pa_x\psi=0,\quad -\Delta \psi=\omega,\\
&\omega|_{t=0}=\omega_0(x,y),
\end{split}
\end{align}
where $\psi$ denotes the stream function, and the velocity $\textbf{V}=(V^1,V^2)=(\pa_y\psi,-\pa_x\psi)$.
Taking the Fourier transform to \eqref{eq:Euler-linearize} in $x$($x\in\mathbb{T}\to\al\in\mathbb{Z})$, we obtain
\begin{align}\label{eq:Euler-linearize-omega}
\pa_t\hat{{\omega}}+i\al \mathcal{R}_{\al}' \hat{\omega}=0,
\end{align}
where $\mathcal{R}_{\al}'\hat{\omega}=u\hat{\omega}-u''(\pa_y^2-\al^2)^{-1}\hat{\omega}$,\;
$\left((\pa_y^2-\al^2)^{-1}f\right)(y)=-(2\al)^{-1}\int_{\mathbb{R}}e^{-\al|y-z|}f(z)dz.$  In terms of the stream function $\hat{\psi}=-(\pa_y^2-\al^2)^{-1}\hat{\omega}$,  \eqref{eq:Euler-linearize-omega} is equivalent to
\begin{align}\label{eq:Euler-linearize-psi}
\pa_t\hat{\psi}+i\al\mathcal{R}_{\al}\hat{\psi}=0,\quad
\hat{\psi}|_{t=0}=\hat{\psi}_0(\al,y)=-(\pa_y^2-\al^2)^{-1}
\hat{\omega}_0,
\end{align}
where the Rayleigh operator $\mathcal{R}_\al$ is defined by
\begin{align}\label{Op:Ray}
\mathcal{R}_{\al}\hat{\psi}
=(\pa_y^2-\al^2)^{-1}\left(u(\pa_y^2-\al^2)-u''\right)\hat{\psi}.
\end{align}
 Without loss of generality, we will focus on $\al\in \mathbb{Z}^+$ due to the following fact
\begin{align*}
\hat{{\omega}}(t,0,y)&\equiv\hat{\omega}_0(0,y), \quad \text{and}\quad\hat{{\omega}}(t,\al,y)=\bar{\hat{{\omega}}}(t,-\al,y)\quad\text{for}\,\,\al\in\mathbb{Z}^-.
\end{align*}

For the Couette flow $u(y)=y$,  Orr \cite{Orr}  observed that the velocity could decay to zero as $t\to\infty$, although the vorticity is conserved. This phenomenon is the so-called inviscid damping,  an analogue of the Landau damping in Plasma physics. Lin and Zeng \cite{LZ} proved the optimal decay rate of the velocity for the linearized Euler equation, i.e., $\|\textbf{V}(t)\|_{L^2}\lesssim t^{-1}$, $\|V^2(t)\|_{L^2}\lesssim t^{-2}$, and they also
proved that nonlinear inviscid damping does not hold for the perturbation of vorticity in $H^{<\f32}$. In a breakthrough work \cite{BM}, Bedrossian and
Masmoudi proved nonlinear inviscid damping for the perturbation in the Gevrey class $2-$. Later on, Ionescu and Jia \cite{IJ-cmp} relaxed  the regularity of the perturbation to the Gevrey class 2, which should be optimal according to the instability result in \cite{DM}.

In a series of works, Wei, the second author and Zhao proved the linear inviscid damping for general shear flows in a finite channel. The work \cite{WZZ-cpam}  proved  the linear inviscid damping for a class of monotone shear flows, for which the linearized operator has no embedding eigenvalues. The work \cite{WZZ-apde}  considered  non-monotone shear flows  including the Poiseuille flow $(1-y^2,0)$, still under the spectral assumption of no embedding eigenvalues. The work \cite{WZZ-am} dealt with the Kolmogorov flow $(\sin y, 0)$.
All these works are based on solving the Rayleigh equation in a direct way
\begin{align}\nonumber
(u-\textbf{c})(\varphi''-\al^2\varphi)-u''\varphi=f.
\end{align}
One of key ingredients is to establish the limiting absorption principle(LAP). For monotone flows, this is relatively easy. For non-monotone flows, \cite{WZZ-apde} developed a compactness method relying on the blow-up analysis near the critical point to prove the LAP. In the sprit of wave equation, Wei, Zhang and Zhu \cite{WZZ-cmp} also developed the vector field method for monotone shear flows. In two breakthrough works \cite{IJ-acta, MZ-ann},  Ionescu-Jia and Masmoudi-Zhao proved independently nonlinear inviscid damping for stable monotone shear flows in a finite channel when the initial vorticity is supported in the interior of the channel and has the Gevrey class regularity. Let us mention many important works \cite{BCV, CZ, IIJ, GNRS, RWWZ, Z1, Z2} and references therein on the linear inviscid damping.

In this paper, we consider a hyperbolic tangent flow
\begin{align*}
u(y)=\tanh y.
\end{align*}
This flow is of physical interest \cite{DH, schade-phyfluid, Stuart-jfm, TG, TGA}. The works \cite{schade-phyfluid, Stuart-jfm} studied the nonlinear stability of 2-D inviscid fluid around $(\tanh y,0)$. In the works \cite{TG, TGA}, the hyperbolic tangent function was taken as a 
laminar flow profile that provides the closest approximation to viscous flow in unbounded domains, especially at high Reynolds numbers.

From the perspective of mathematic, our main motivation stems from the fact that the linearized operator around $u(y)=\tanh y$ admits an embedding eigenvalue. The eigenvalue gives rise to novel dynamic phenomena and brings new mathematical challenges. To the best of our knowledge, there are few results in fluid mechanics considering the case when the linearized operator has embedding eigenvalues. It is worth mentioning  that \cite{Stepin} 
provided an abstract expansion of the Rayleigh operator in terms of eigenfunctions.

\subsection{Main results}

Let us first introduce the weighted norms
\begin{align*}
&\|f\|_{L^2_{w}(I)}:=\Big(\int_{I}|f(y)|^2w(y)dy\Big)^{\f12},\\
&\|f\|_{H^k_{w}(I)}:=\Big(\int_{I}\big(
\sum_{i=0}^k|f^{(i)}(y)|^2\big)w(y)dy\Big)^{\f12},
\end{align*}
 where $I\subset \mathbb{R}$, $k\in\mathbb{N}$ and $w(y)$ is a weight function.\smallskip

 Our main results are stated as follows.

 \begin{Theorem}\label{Thm:main}
 We have the following long-time dynamics for the solution of \eqref{eq:Euler-linearize-psi}:

 \begin{itemize}
\item[(1)] {\bf Mode $\al=1$}. $\mathcal{R}_{1}$ has a single embedding eigenvalue with one dimensional eigenspace
$\{\mathrm{sech}\, y\}$. Define
\begin{align}\label{def:a0} \quad a_{0}=p.v.\int_{\mathbb{R}}\f{\hat{\omega}_0(1,y)}{\sinh y}dy,\quad
b_{0}=\hat{\omega}_0(1,0).
\end{align}
If $\hat{\omega}_0(1,\cdot)\in H^3_{1/(u')^4}(\mathbb{R})$, then there exists a function $f_1(t,y)$, which is independent of $\omega_0$, satisfying
 \begin{align} \label{lim:lambda1}\lim_{t\to\infty}\|f_1(t,\cdot)\|_{H^1}=0,\;
\lim_{t\to\infty}t\|f_1(t,\cdot)\|_{L^2}=0,
\end{align}
such that for any $t\in\mathbb{R}$,
\begin{align}
\label{Thm:main-al-1}&\|\hat{\psi}(t,1,\cdot)
-\f{a_{0}+i\pi b_0}{2\pi i}\mathrm{sech}\,y -a_{0}f_1(t,\cdot)\|_{H^1}\leq C \langle t\rangle^{-1}\|\hat{\omega}_0(1,\cdot)\|_{H^2_{1/(u')^2}},\\
\label{Thm:main-al-12}&\|\hat{\psi}(t,1,\cdot)-\f{a_{0}+i\pi b_0}{2\pi i}\mathrm{sech}\,y
-a_{0}f_1(t,\cdot)\|_{L^2}\leq C \langle t\rangle^{-2}\|\hat{\omega}_0(1,\cdot)\|_{H^3_{1/(u')^4}},
\end{align}
where $C$ is an universal constant.

\item[(2)] {\bf Modes $\al\ge 2$}. $\mathcal{R}_{\al}$ 
has no embedding eigenvalues. If  $\hat{\omega}_0(\al,\cdot)\in H^2_{1/(u')^4}(\mathbb{R})$, then  it holds that for any $t\in\mathbb{R}$,
 \begin{align}
\label{Thm:main-al-geq2}&\|\pa_y\hat{\psi}(t,\al,\cdot)\|_{L^2}
+\al\|\hat{\psi}(t,\al,\cdot)\|_{L^2}\leq C\al^{-1} \langle t\rangle^{-1}\|\hat{\omega}_0(\al,\cdot)\|_{H^1_{1/(u')^2}},\\
\label{Thm:main-al-geq22}&\|\hat{\psi}(t,\al,\cdot)\|_{L^2}\leq C\al^{-2} \langle t\rangle^{-2}\|\hat{\omega}_0(\al,\cdot)\|_{H^2_{1/(u')^4}},
\end{align}
where $C$ is an universal constant.

\end{itemize}
\end{Theorem}

The following theorem is a direct consequence of Theorem \ref{Thm:main} under more assumptions on the initial vorticity.

 \begin{Theorem}\label{Thm:main0}
 If the initial vorticity satisfies the following conditions
  \begin{itemize}
 \item[(1)]  $\hat{\omega}_0(0,y)=0$ for any $y\in\mathbb{R}$,

\item[(2)] $\omega_0\in H^{-1}_xH^3_{y,1/(u')^4}(\mathbb{T}\times \mathbb{R})$,

\item[(3)]  $a_{0}=b_0=0$,
 \end{itemize}
 then it holds that for any $t\in\mathbb{R}$,
 \begin{align}\label{Thm:main-gernal}
 \begin{split}
&\|\textbf{V}(t)\|_{L^2}\leq C \langle t\rangle^{-1}\|\omega_0\|_{H^{-1}_xH^2_{y,1/(u')^2}},\\
&\;\|V^2(t)\|_{L^2}\leq C \langle t\rangle^{-2}\|\omega_0\|_{H^{-1}_xH^3_{y,1/(u')^4}},
\end{split}
\end{align}
where $C$ is an universal constant.
\end{Theorem}

Let us give some remarks on our results.

\begin{itemize}

\item For the first mode, compared with previous results, we find that in the presence of an embedding eigenvalue, the optimal decay rates hold when certain appropriate functions are removed from the solution.

   We rewrite the new dynamic phenomenon \eqref{Thm:main-al-1}-\eqref{Thm:main-al-12} in the form of operator as
    \begin{align}\label{decay:operator}
    \|e^{-i\mathcal{R}_1t}\left(I-\mathcal{P}_{0}\right)-F_t\|_{H^{5-k}_{1/(u')^{4-2k}}\to H^{1-k}}\leq C\langle t\rangle ^{-2+k},\quad k=0,1,
    \end{align}
    where $\mathcal{P}_{0}(\psi_0):=\f{a_0\left(-(\pa_y^2-1)\psi_0\right)+i\pi b_0\left(-(\pa_y^2-1)\psi_0\right)}{2\pi i}\mathrm{sech}\;y$ is the projection of $e^{-i\mathcal{R}_1t}$ onto the eigenspace, $F_t(\psi_0):=a_0\left(-(\pa_y^2-1)\psi_0\right)f_{1}(t,y)$ is a rank one operator.
    A natural expectation is that removing the projection $\mathcal{P}_{0}$ from $e^{-i\mathcal{R}_{1}t}$ would result in the optimal inviscid damping.  However, unexpectedly, an additional rank one operator $F_t$ emerges, which leads to a loss of decay compared to the optimal decay rate. 
     From the perspective of technical level, the emergence of rank one operator $F_t$ can be attributed to  the $\f{1}{c}$ singularity in the resolvent resulting from the embedding eigenvalue, see the decomposition \eqref{def:Psi1} in  Lemma \ref{lem:dec-psi} and Remark \ref{rmk:K}. 

    On the other hand, the  phenomenon  is similar to the  dispersive estimates for Schr\"odinger group $e^{itH}$ with $H=-\Delta+V$, in the presence of a resonance and/or an eigenvalue at zero energy, see  \cite{ES-DPDE, ES-JAM, Y-CMP}. For a more detailed history, see the survey paper \cite{schlag2007}.

\item From  \eqref{eq:Euler-linearize-omega}, it is easy to  see that the following quantities are  conserved
\begin{align*}
a(t):=p.v.\int_{\mathbb{R}}\f{\hat{\omega}_0(t,1,y)}{\sinh y}dy\equiv a_0\quad\text{and}\quad
b(t):=\hat{\omega}_0(t,1,0)\equiv b_0.
\end{align*}
Moreover, if $a_0=b_0=0$ for the initial vorticity,  then linear inviscid damping holds, i.e., Theorem \ref{Thm:main0}.

\item Consider a family of hyperbolic tangent flows $\tanh \f{y}{L}$,\;$L>0$. It was known from \cite{Lin03} or \cite{LLZ} that
    \begin{itemize}
    \item
    Spectral instability holds for $\mathcal{R}_{\al}\left(1\leq \al<\f{1}{L}\right)$,  if $L<1$.
     \item
     Spectral stability holds for $\mathcal{R}_{\al}(\al\geq1)$,  if $L\geq 1$.
    \end{itemize}
The first author has proved the linear inviscid  damping for $L>1$ and $\al\geq 1$ in \cite{R}. This paper deals with  more difficult critical case $L=1$.
\item Taylor \cite[(2.22)\;\text{with}\;$\sigma=0$]{Taylor}
    provides anothor family of shear flows of hyperbolic tangent type with the form
    \begin{align}\label{def:u-gamma}
    u_{\gamma}(y)=\f{\f{\gamma}{2}e^y-\f{1}{2\gamma}e^{-y}}
    {\f{\gamma}{2}e^y+\f{1}{2\gamma}e^{-y}},\;\gamma>0.
    \end{align}
Similar result as in Theorem \ref{Thm:main} holds for \eqref{def:u-gamma}, which follows from the $y-$translation invariance of equation \eqref{eq:Euler-linearize-psi} and the observation $u_{\gamma}(y)=\tanh (y+\ln \gamma)$.

\item We believe that the approach developed here could be applied to general $\tanh$-type flows.

\end{itemize}

\subsection{Framework, difficulties, and  ideas}
 We adopt the framework established in previous works  \cite{WZZ-cpam, WZZ-apde, WZZ-am}. However, there will be two types of essential difficulties inherent in this problem.
\begin{itemize}
\item[(1)] The singularity of the resolvent/spectal density function at embedding eigenvalue $c=0$.

\item[(2)]$u'(y)$ is positive everywhere but decays exponentially at infinity. Moreover, the physical domain $\mathbb{T}\times \mathbb{R}$ lacks compactness.

    We note that  the previous works \cite{WZZ-cpam, WZZ-apde, WZZ-am} have primarily focused on bounded domains, where the strict monotonicity of the flow has played a crucial role. Regarding the aspect of monotonicity, we believe that the difficulty (2) of our problem falls between that of \cite{WZZ-cpam} and \cite{WZZ-apde}/\cite{IIJ}($u'(y)$ has a true zero point).
\end{itemize}

  The framework itself is straightforward, aiming to deduce the explicit expression of the stream function through the following approach:
\begin{align}\label{eq:stream formula'}
\notag\hat{\psi}(t,\al,y)&=\f{1}{2\pi i}\int_{\pa\Omega_{\ep}}e^{-i\al \textbf{c}t}(\textbf{c}-\mathcal{R}_{\al})^{-1}\hat{\psi}_0(\al,y)d\textbf{c}\\
&\xlongequal{\textbf{Step 1}}\lim_{\ep\to0}\f{1}{2\pi i}\int_{\pa\Omega_{\ep}}e^{-i\al \textbf{c}t}\Phi(\al,y,\textbf{c})d\textbf{c}\notag\\
&\xlongequal{\textbf{Step 2}}\f{1}{2\pi i}\int_{\mathrm{Ran}\;u}e^{-i\al ct}\tilde{\Phi}(\al,y,c)dc,
\end{align}
where $\Omega_{\ep}\supset \mathrm{Ran}\;u$ denotes a complex domain with a thickness of $\ep$,  and the spectral density function $\Phi(\al,y,\textbf{c})$ solves the Rayleigh equation
\begin{align}\label{eq:inhomou}
(u-\textbf{c})(\Phi''-\al^2\Phi)-u''\Phi=\hat{\omega}_0.
\end{align}
Roughly speaking, the framework can be divided into three key steps.\smallskip

\begin{itemize}

\item[{\bf Step 1}.] Solving the Rayleigh equation quantitatively.

We first construct a smooth solution  $\phi_1(y,c)$ in an {open region} $(y,\textbf{c})\in\mathbb{R}\times \tilde{D}_{\ep_0}$ satisfying the following equation
    \begin{align*}
   (u-\textbf{c})(\phi_1''-\al^2\phi_1)+2u'\phi_1'=0,\;
   \phi_1(y_c,\textbf{c})=1,\;\phi_1'(y_c,\textbf{c})=0,
    \end{align*}
    where $\phi=(u-\textbf{c})\phi_1$ satisfies the homogenous Rayleigh equation, see Proposition \ref{prop:Rayleigh-Hom}. Then we use $\phi_1$ to  construct the unique solution for inhomogeneous Rayleigh equation \eqref{eq:inhomou}, see Proposition \ref{pro:solve-Phi}.
   The key of this step is to estimate $\phi_1$ uniformly in $\textbf{c}$, $y$, $\al$, where  the main  difficulty lies in (2).

   Indeed, the previous works \cite{WZZ-cpam, WZZ-apde, WZZ-am} constructed $\phi_1(y,\textbf{c})$  in a {compact region} such as $[u(0),u(1)]\times \Omega_{\ep_0}$, due to the bounded physical domain. Consequently, the uniform continuity on a compact region allows for inheriting the estimates for real variable $c$ to complex variable $\textbf{c}$.  To overcome the lacking of compactness in our problem, we introduce a new region
\begin{align*}
O_{\ep_0}&:=\big\{c+i\ep|c\in(-1,1),\; 0<|\ep|< \min\{C_o^{-1}(1-c^2), \ep_0\}\big\},
\end{align*}
where the real properties can be inherited, see Proposition \ref{cor:phi_1-infty}. Here, the explicit vanishing rate of $u'(y)$ at infinity plays an important role.
The uniform real estimates in $c$ can be deduced from integral equations, see Lemmas \ref{prop:phi1} and \ref{prop:goodphi1}.

\item[{\bf Step 2}.] Establishing the limiting absorption principle(LAP).

For $c\in(-1,1)/\{0\}$, the difficulty (2) is a crucial obstacle in establishing LAP. On one hand, following a similar approach as in \cite[Proposition 6.7]{WZZ-cpam}, we prove the spectral density function has the limits $\Phi_{\pm}$:
\begin{align*}
\lim_{\ep\to 0\pm}\Phi(\al,y,c_{\ep})=\Phi_{\pm}(\al,y,c).
\end{align*}
On the other hand, in order to apply the dominated convergence theorem in \eqref{eq:stream formula'},  we establish the {resolvent bound which is uniform in $\ep$} and {weighted in $c$}:
 \begin{align}\label{bd:yn}
|\Phi(\al,y,c_{\ep})|\leq C(1-c^2)^{-\f12}e^{-C^{-1}|y-y_c|}.
\end{align}
In contrast to the resolvent bounds in previous works, i.e., \cite[Proposition 6.2]{WZZ-apde} and \cite[Lemma 2.4]{WZZ-cmp},  which were proved using  the contradiction argument and blow-up analysis,  our bound \eqref{bd:yn} is derived directly from the {explicit expression of the inhomogeneous solution} constructed in the step 1, cf. \eqref{solution-Phi-1}. The weight $(1-c^2)^{-\f12}$ comes from the exponential degeneracy of $u'$ at infinity, see Proposition \ref{lem:limitfunction}.

For the embedding eigenvalue $c=0$, the difficulty (1) plays a crucial role in establishing LAP. The singularity of spectral density function $\Phi(\al,y,\textbf{c})$ poses an obstacle at $c=0$ while using conventional approaches. We establish \textbf{a new type of LAP}:
\begin{align*}
\lim_{\pm\mathrm{Im}(\textbf{c})>0,\;|\textbf{c}|\to 0}\textbf{c}
\Phi(1,y,\textbf{c})=\f{\mathcal{T}(\hat{\omega}_0)(0)\pm i\pi \hat{\omega}_0(0)}{\pm2\pi i}\mathrm{sech}\,y ,
\end{align*}
where $\mathcal{T}$ is some integral operator. The new resolvent bound takes
 \begin{align*}
\Big|\textbf{c}\Phi(1,y,\textbf{c})\Big|\leq C,\quad \text{for}\;\textbf{c}\in B_{\delta}(0)/(-1,1)\subset O_{\ep_0},
\end{align*}
see Proposition \ref{lem:B-delta1}. The key is to conduct a {refined analysis of the Wronskian} of the Rayleigh equation near $c=0$, which gives
 \begin{align*}\lim_{\textbf{c}\to 0}\f{W(\textbf{c},1)}{\textbf{c}}=\mathrm{sgn}(\mathrm{Im}(\textbf{c}))2\pi i,
\end{align*}
see Section 5. As soon as the LAP is established, the expression of stream function for $\al=1$ can be obtained by {a carefully chosen  contour with two parameters $\ep$ and $\delta$},  which separates the singularity from the other continuous spectrums. See Proposition \ref{Pro:yn}, Figure 2 and Figure 3.

\item[{\bf Step 3}.] Estimating the integral operators.

We give the boundedness of  linear/bilnear integral operators, which are emerged in the expression of stream function.
The main difficulty in this step is (2). To overcome (2), we develop {weighted estimates} for both spatial variable $y\in\mathbb{R}$ with exponential weight $1/u'(y)$ and spectral variable $c\in(-1,1)$ with degenerate weight $1-c^2$, see Lemma \ref{lem:pa_cA} and Lemma \ref{lem:key-T}. We note that $u'(y_c)\sim 1-c^2$ and  bound \eqref{u-key} plays a crucial role in establishing the desired estimates. 
\end{itemize}

\subsection{Notations}

\begin{itemize}
\item To distinguish a complex number from a real number, we sometimes use the notation $\textbf{c}$ or $c_{\ep}=c+i\ep$.
\item 
We define the following complex domains:
\begin{align}\label{domain:D}
D_{\ep_0}:=\{c+i\ep| c\in(-1,1),\; 0<|\ep|<\ep_0\},\quad
\widetilde{D}_{\ep_0}:=D_{\ep_0}\cup(-1,1).
\end{align}
 \begin{align}
\label{domain:O-ep}O_{\ep_0}&:=\big\{c+i\ep|c\in(-1,1),\; 0<|\ep|< \min\{C_o^{-1}(1-c^2), \ep_0\}\big\},\quad \widetilde{O}_{\ep_0}:=O_{\ep_0}\cup(-1,1),
\end{align}
which satisfies
$O_{\ep_0}\subset D_{\ep_0},\quad
\widetilde{O}_{\ep_0}\subset \widetilde{D}_{\ep_0}.$ Here, $\ep_0\in(0,\f12)$, $C_o>1$ are fixed constants determined later.
The open ball are denoted as
\begin{align*}
B_{\delta}(0):=\{\textbf{c}||\textbf{c}|< \delta\}.
\end{align*}


\item For $c\in \widetilde{D}_{\ep_0}$,  we denote $y_c:=u^{-1}(\mathrm{Re}(c))$. In particular, for $c\in(-1,1)$, we denote $y_c:=u^{-1}(c)$.

\item Good derivative is defined as \begin{align}\label{def:goodG}
    \pa_G:=\f{\pa_y}{u'(y_c)}+\pa_c.
    \end{align}
\item $f(y)\sim g(y)$ denotes
$$C^{-1}g(y)\leq f(y)\leq C g(y),$$
where the constant $C$ is univeral.
\item The parameter $\al$ is always a positive integer. Occasionally,  we may omit the explicit dependence on $\al$ for brevity. 
\item  For simplicity,  sometimes we use the notation $'$ to represent the derivative with respect to $y$, and $''$ to represent the second order  derivative with respect to $y$, etc.
\end{itemize}

\section{Homogenous Rayleigh equation}

 In this section, we take $\ep_0\in(0,\f12]$ in \eqref{domain:D}-\eqref{domain:O-ep}. The goal of this section is to solve the homogenous Rayleigh equation for $c\in \widetilde{D}_{\ep_0}$:
\begin{align}\label{eq:phi}
&(u-c)(\phi''-\al^2\phi)-u''\phi=0.
\end{align}
We follow the framework introduced in \cite{WZZ-cpam,WZZ-apde}. The degeneracy  of $u'(y)$ at infinity leads to some essential difficulties.

To proceed, let's present  some basic properties of $u(y)=\tanh y$, which will be used frequently. The following lemma can be proved by direct calculations.
\begin{lemma}\label{lem:simple-useful-equ}
It holds that
\begin{itemize}
\item
$u'(y)=1-u(y)^2,\; u''(y)=-2u(y)u'(y).$
\item
$u(y)$, $u''(y)$ are odd in $y$.  $u'(y)$, $u'''(y)$ are even in $y$.
\item
For $c\in \mathrm{Ran}u=(-1,1)$, it holds that
\begin{align}
\notag u'(y_c)&=1-c^2,\; u''(y_c)=-2c(1-c^2),\;u'(y_c)-u'(y)=(u+c)(u-c).
\end{align}
\item It holds that
\begin{align}
\notag\pa_G\left(u(y)-c\right)&=\f{-(u+c)(u-c)}{u'(y_c)},\;
\pa_G\left(\f{u'(y)}{u'(y_c)}\right)=\f{-2(u-c)u'(y)}{u'(y_c)^2}.\\
\notag\pa_G\left(\f{1-c^2}{u(y)-c}\right)&=1,\;
\pa_G\left(\f{u'(y)}{u(y)-c}\right)
=-\f{u'(y)}{u'(y_c)},\;\pa_G\left(\f{(1-c^2)u'(y)}{(u(y)-c)^2}\right)=0.
\end{align}

\end{itemize}
\end{lemma}

\begin{lemma}\label{lem:simple-useful-inequ}
It holds that
\begin{align}
\label{sinhy} \cosh y\sim e^{|y|},\; |\sinh y|&\sim e^{|y|}\min\{|y|,1\},\;
 |\tanh y|\sim \min\{|y|,1\}.\\
\label{coshy1}\cosh y&\leq C\cosh z\cosh (y-z).\\
\label{coshy} \cosh z&\sim\cosh y,\;\text{if}\;\;\;|y-z|\leq C.
\end{align}
For $(y,c)\in\mathbb{R}\times (-1,1)$, it holds that
\begin{align}
\label{fm:u(y)-c}|u(y)-c|
&=\f{\sinh |y-y_c|}{\cosh y\cosh y_c},
\end{align}
and
\begin{align}
\label{u-key}\f{u'(y)+u'(y_c)+(1-c^2)}{|u(y)-c|}&\leq \f{C}{\tanh |y-y_c|}\leq \f{C}{\min\{|y-y_c|,1\}},
\end{align}
where the constant $C$ is  independent of $y,c,\al$. For  $c\in \widetilde{D}_{\ep_0}$, $y_c\leq z\leq y$ or $y\leq z\leq y_c$, it holds that
\begin{align}\label{u(z)-c<u(y)-c}
|u(z)-c|\leq |u(y)-c|.
\end{align}
\end{lemma}
\begin{proof}
The estimate \eqref{sinhy} is obvious. The inequality \eqref{coshy1} follows from
$$\cosh y\sim e^{|y|}\leq e^{|z|}e^{|y-z|}\leq C \cosh z\cosh |y-z|,$$
 which also implies \eqref{coshy}. The identity \eqref{fm:u(y)-c} is obvious. The inequality \eqref{u-key} follows from \eqref{coshy1} and \eqref{sinhy} that
\begin{align}
\notag|u(y)-c|&=\f{\sinh |y-y_c|}{\cosh y\cosh y_c}\geq C^{-1}\f{\sinh |y-y_c|}{\cosh|y-y_c|}\max\{\cosh^{-2} y_c,\cosh^{-2} y\}\\
&= C^{-1}\tanh |y-y_c|\Big(u'(y_c)+u'(y)\Big).\notag
\end{align}
Thanks  to the monotonicity of $u(y)$, together with $1-c^2=u'(y_c)$, we have \eqref{u(z)-c<u(y)-c}.
\end{proof}

\subsection{Existence of the solution}

 Let    $\phi_1(y,c)=\f{\phi(y,c)}{u(y)-c}$ solve
\begin{align}
\label{eq:Rayleigh-hom-S1}& \phi_1''-\al^2\phi_1+\f{2u'}{u-c}\phi_1'=0,\;\;\phi_1(y_c,c)=1,\,\;\phi_1'(y_c,c)=0.
\end{align}

\begin{definition}\label{def:T_kj}
We define the following integral operators
\begin{align*}
Tf(y,c)&:= T_0\circ T_{2,2}f=\int_{y_c}^y\f{\int_{y_c}^{z}f(w,c)(u(w)-c)^2dw} {(u(z)-c)^2}dz,
\end{align*}
with
\begin{align*}
T_0f(y,c)&:=\int_{y_c}^yf(z,c)dz,\\
T_{k,j}f(y,c)&:=\frac{\int_{y_c}^yf(z,c)(u(z)-c)^kdz}{(u(y)-c)^j}(y\neq y_c),\;\;j\leq k+1, \;\;T_{k,k}f(y_c,c)=0.
\end{align*}
\end{definition}

Equation \eqref{eq:Rayleigh-hom-S1} is equivalent to the following integral formulation
\begin{align*}
\phi_1&=1+\al^2T\phi_1.
\end{align*}

\begin{proposition}\label{prop:Rayleigh-Hom}
There exists a unique solution $\phi_1(y,c)\in  C\big(\mathbb{R}\times \widetilde{D}_{\ep_0}\big)\cap
C^2\big(\mathbb{R}\times (-1,1)\big)$ to \eqref{eq:Rayleigh-hom-S1}, which holds
\begin{align}
\label{phi_1-1}\phi_1(y,c)&=1+\al^2\int_{y_c}^y\f{\int_{y_c}^z(u(w)-c)^2\phi_1(w,c)dw}{(u(z)-c)^2}dz.
 \end{align}
Consequently, $\phi(y,c)=(u(y)-c)\phi_1(y,c)\in  C\big(\mathbb{R}\times \widetilde{D}_{\ep_0}\big)\cap
C^2\big(\mathbb{R}\times (-1,1)\big)$)  is a solution to the homogenous Rayleigh equation \eqref{eq:phi}. Moreover, it holds that

(1) $\pa_y\phi_1(y,c),\f{\phi_1(y,c)-1}{(u(y)-c)^2}\in  C\big(\mathbb{R}\times \widetilde{D}_{\ep_0}\big)$.

  (2)  for any $y\in\mathbb{R}$, $\phi_1(y,c)$, $\f{\phi_1(y,c)-1}{(u(y)-c)^2}$ is analytic in $\widetilde{D}_{\ep_0}$.

  (3) $\phi_1(y,0)$ is even in $y$, $\pa_y\phi_1(y,0)$ is odd in $y$.

  (4) for $(y,c)\in\mathbb{R}\times (-1,1)$, it holds that
 \begin{align}
 &\phi_1(y,c)\geq \phi_1(z,c)\geq 1,\quad \text{for}\;\;y\leq z\leq y_c\; \text{or}\;\; y_c\leq z\leq y\label{phi_1(z)leq phi_1(y)},\\
  &0\leq (y-y_c)\pa_y\phi_1(y,c)\leq \al^2(y-y_c)^2\phi_1(y,c),\label{phi_1-1leq''}\\
 &0\leq \phi_1(y,c)-1\leq \f12\al^2(y-y_c)^2\phi_1(y,c).\label{phi_1-1leq}
 \end{align}

\end{proposition}
To prove the existence in Proposition \ref{prop:Rayleigh-Hom}, we will introduce some norms.
\begin{definition}\label{def:X-space}
Let $A> \al\geq 1$, $\gamma\in\mathbb{N}$. We define
\beno
&&\|f\|_{{X}}:=\sup_{(y,c)\in \mathbb{R}\times \tilde{D}_{\epsilon_0}}\bigg|\frac{f(y,c)}{\cosh(A(y-y_c))}\bigg|.\\
&& \|f\|_{Y^{\gamma}}:=\sup_{(y,c)\in \mathbb{R}\times (-1,1)}\bigg|\frac{\left((u'(y_c)\right)^{\gamma}f(y,c)}{\cosh(A(y-y_c))}\bigg|.\\
&&
\|f\|_{Y}:= \sum_{k=0}^2\sum_{\b+\gamma=k}A^{-k}\|\partial_y^{\b}\partial_c^{\gamma}f\|_{Y^{\gamma}}.
\eeno

\end{definition}
We have the following lemma of contraction mapping.

\begin{lemma}\label{lem:T-bound}

 There exists a constant $C$ independent of $A$ such that
\beno
&&\quad\|Tf\|_{Y^0}\leq \frac{C}{A^2}\|f\|_{Y^0},\quad\|Tf\|_{X}\leq \frac{C}{A^2}\|f\|_X,\quad\|Tf\|_{Y}\leq \frac{C}{A^2}\|f\|_{Y}.
\eeno
 Moreover, it holds that
 \begin{align}
 \label{cont1}\text{if}\quad  &f\in C\big(\mathbb{R}\times \widetilde{D}_{\ep_0}\big),\quad \text{then}\quad  T_0f,\;T_{2,2}f,\;Tf,\; \f{Tf}{(u-c)^2}\in C\big(\mathbb{R}\times\tilde{ D}_{\ep_0}\big),\\
  \text{if}\quad  &f\in C^k\big(\mathbb{R}\times (-1,1)\big),\quad  \text{then}\quad Tf\in C^k\big(\mathbb{R}\times (-1,1)\big),\quad k=1,2. \label{cont2}
\end{align}

\end{lemma}

\begin{proof}
 First, we show for $\gamma\in \mathbb{N}$,
\begin{align}\label{1YT}
\|Tf\|_{Y^{\gamma}}\leq A^{-2}\|f\|_{Y^{\gamma}}.
\end{align}
In particular, $\gamma=0$ gives the first estimate of the lemma.

Direct calculation gives
\begin{align}
\label{T0}&\|T_0f\|_{Y^{\gamma}}=\sup_{(y,c)\in \mathbb{R}\times (-1,1)}\bigg|\f{(u'(y_c))^{\gamma}\int_{y_c}^yf(z,c)dz}{\cosh A(y-y_c)}\bigg|\\
&\leq
\notag\sup_{(y,c)\in \mathbb{R}\times (-1,1)}\bigg|\f{\int_{y_c}^y\cosh A(z-y_c)dz}{\cosh A(y-y_c)}\bigg|\sup_{(z,c)\in \mathbb{R}\times (-1,1)}\bigg|\f{(u'(y_c))^{\gamma}f(z,c)}{\cosh A(z-y_c)}\bigg|\leq A^{-1}\|f\|_{Y^{\gamma}}.
\end{align}
Thanks to \eqref{u(z)-c<u(y)-c},   we have
\begin{align}
\label{T22}\|T_{2,2}f\|_{Y^{\gamma}}=&\sup_{(y,c)\in \mathbb{R}\times (-1,1)}\bigg|\f{(u'(y_c))^{\gamma}\int_{y_c}^y(u(z)-c)^{2}f(z,c)}{(u(y)-c)^{2}\cosh A(y-y_c)}\bigg|\\
\leq& \|T_{0}f\|_{Y^{\gamma}}\leq A^{-1}\|f\|_{Y^{\gamma}}.\nonumber
\end{align}
Therefore, \eqref{1YT} follows from \eqref{T0} and \eqref{T22}.
In the same way, we can obtain
\begin{align*}
\|T_0f\|_{X}\leq A^{-1}\|f\|_{X},\quad \|T_{2,2}f\|_{X}\leq A^{-1}\|f\|_{X},
 \end{align*}
which show the second estimate of the lemma.

Next we prove the third estimate. Thanks to \eqref{u-key}, 
we have
\begin{align*}
\Bigg|\f{\Big(u'(y_c)+u'(y)\Big)\int_{y_c}^y\cosh A(z-y_c)dz}{(u(y)-c)\cosh A(y-y_c)}\Bigg|
&\leq CA^{-1}\f{\tanh A|y-y_c|}{\tanh |y-y_c|}\leq C,
\end{align*}
which along with \eqref{u(z)-c<u(y)-c} shows that for $k,\gamma\in \mathbb{N}$,
\begin{align}
\|T_{k,k+1}f\|_{Y^{\gamma+1}}=&\sup_{(y,c)\in \mathbb{R}\times (-1,1)}\bigg|\f{(u'(y_c))^{\gamma+1}\int_{y_c}^y(u(z)-c)^{k}f(z,c)}{(u(y)-c)^{k+1}\cosh A(y-y_c)}\bigg|\notag\\
&\leq
\sup_{(y,c)\in \mathbb{R}\times (-1,1)}\bigg|\f{u'(y_c)\int_{y_c}^y\cosh A(z-y_c)dz}{(u(y)-c)\cosh A(y-y_c)}\bigg|\nonumber\\
&\qquad\times\sup_{(z,c)\in \mathbb{R}\times (-1,1)}\bigg|\f{(u'(y_c))^{\gamma}f(z,c)}{\cosh A(z-y_c)}\bigg|\leq C\|f\|_{Y^{\gamma}}\label{T_k,k+1}.
\end{align}
In the same way as in \eqref{T_k,k+1},  for $k,\gamma\in \mathbb{N}$, we have
\begin{align}
&\quad\|u'(y)T_{k,k+1}f\|_{Y^{\gamma}}\leq C\|f\|_{Y^{\gamma}}\label{T_k,k+1'}.
\end{align}
For $T_{k,j}$, we rewrite it as
\begin{align}\label{Tkj}
&T_{k,j}f(y,c)=\int_0^1f(y_c+t(y-y_c),c)(y-y_c)^{k+1-j}\f{(\int_0^1u'(y_c+st(y-y_c)ds)^k}
{(\int_0^1u'(y_c+s(y-y_c)ds)^{j}}t^kdt.
\end{align}
Therefore, we have
\begin{equation}
\label{cTk,k+1yc}
\begin{split}
\pa_cT_{k,k+1}f(y,c)&=
\f{\pa_{y_c}\bigg(\int_0^1f(y_c+t(y-y_c),c)\f{t^k(\int_0^1u'(y_c+st(y-y_c)ds)^k}
{(\int_0^1u'(y_c+s(y-y_c)ds)^{k+1}}dt\bigg)}{u'(y_c)},\\
\pa_cT_{k,k+1}f(y_c,c)
&=\f{\pa_{y}f(y_c,c)}{(k+1)(k+2)u'(y_c)^3}
+\f{\pa_cf(y_c,c)}{(k+1)u'(y_c)^2}-\f{u''(y_c)f(y_c,c)}{(k+1)(k+2)u'(y_c)^3}.
\end{split}
\end{equation}
We can deduce from \eqref{cTk,k+1yc} that for $k,\gamma\in \mathbb{N}$
\begin{align}\label{pacT_k,k+1}
&\quad\|\pa_cT_{k,k+1}f\|_{Y^{\gamma+2}}\leq C\|f\|_{Y^{\gamma}}+CA^{-1}\Big(\|\pa_yf\|_{Y^{\gamma}}
+\|\pa_cf\|_{Y^{\gamma+1}}\Big).
\end{align}
Using  the facts  $T_0f(y_c,c)=0$, $T_{k,k}f(y_c,c)=0$, $\partial_cT_0f(y,c)=\f{f(y_c,c)}{u'(y_c)}$, $T_{k,k+1}f(y_c,c)=\f{f(y_c,c)}{(k+1)u'(y_c)}$ and \eqref{cTk,k+1yc}, a tedious  calculation gives
\begin{align}
\notag&\partial_yTf(y,c)=T_{2,2}f(y,c),\\
\notag&\partial_cTf(y,c)=2T_0\circ T_{2,3}f(y,c)-2T_0\circ T_{1,2}f(y,c)+T\pa_cf(y,c),\\
\label{DerT1}&\partial_y^2Tf(y,c)=-2u'(y)T_{2,3}f(y,c)+f(y,c),\\
\notag&\partial_y\partial_cTf(y,c)
=2T_{2,3}f(y,c)-2T_{1,2}f(y,c)+T_{2,2}\partial_cf(y,c),\\
&\partial_c^2Tf(y,c)
\notag=2T_0\partial_cT_{2,3}f(y,c)-2T_0\partial_cT_{1,2}f(y,c)+\frac{f(y_c,c)}{3u'(y_c)^2}\\
\notag&\quad\quad\quad\quad\quad+2T_0T_{2,3}\partial_cf(y,c)-2T_0T_{1,2}\partial_cf(y,c)+T\partial_c^2f(y,c).
\end{align}
Summing up \eqref{DerT1}, \eqref{T0}, \eqref{T22}, \eqref{T_k,k+1}, \eqref{T_k,k+1'}, \eqref{pacT_k,k+1} and using Lemma \ref{lem:simple-useful-equ}, we can conclude $\|Tf\|_{Y}\leq CA^{-2}\|f\|_{Y}$.

Finally, we prove the continuity of the solution. For $(y,c)\in \mathbb{R}\times \widetilde{D}_{\ep_0}$, using \eqref{Tkj}, we rewrite the integral operators in Definitions \ref{def:T_kj} as
\begin{align}
\notag &T_0f(y,c)=(y-y_c)\int_0^1f(y_c+t(y-y_c),c)dt,\\
\notag &T_{2,2}f(y,c)=\int_0^1(y-y_c)K_2(t,y,c)f(y_c+t(y-y_c),c)dt,\\
\label{fm:Tf}&Tf(y,c)=\int_0^1\int_0^1(y-y_c)^2K(t,s,y,c)f(y_c+ts(y-y_c),c)dtds,
\end{align}
where
\begin{align*}
K_2(t,y,c)&=\f{(t\int_0^1u'(y_c+st(y-y_c)ds)^2}
{(\int_0^1u'(y_c+s(y-y_c)ds)^{2}},\quad
K(t,s,y,c)&=\f{s(ts\int_0^1u'(y_c+st(y-y_c)ds)^2}
{(\int_0^1u'(y_c+s(y-y_c)ds)^{2}}.
\end{align*}
 It is obvious that
 \begin{align*}
 &K_2\in C\big([0,1]\times\mathbb{R}\times \widetilde{D}_{\ep_0}\big),\;\; K\in C\big([0,1]^2\times\mathbb{R}\times \widetilde{D}_{\ep_0}\big),\\ &\f{(y-y_c)^2K(t,s,y,c)}{(u(y)-c)^2}
 =\f{K(t,s,y,c)}{\big(\int_{0}^1u'(y_c+t(y-y_c)dt\big)^2}\in C\big([0,1]^2\times\mathbb{R}\times \widetilde{D}_{\ep_0}\big),
 \end{align*}
 Therefore, \eqref{cont1} follows. Moreover, if  $f\in C\big(\mathbb{R}\times (-1,1)\big)$, using \eqref{Tkj},  we can extent $T_{k,k+1}f$ as a continuous function on $\mathbb{R}\times(-1,1)$.

 Thanks to \eqref{cont1}  and the formulas of 1st order derivative in \eqref{DerT1}, we obtain \eqref{cont2} for $k=1$. Similarly, thanks to \eqref{cont1}  and the formulas of 2nd order derivative in \eqref{DerT1}, we can obtain \eqref{cont2} for $k=2$.
 \end{proof}

Now we prove Proposition \ref{prop:Rayleigh-Hom}.

\begin{proof}

By Lemma \ref{lem:T-bound}, we know that  $1-\al^2T$ is invertible in  the spaces
\begin{align*}
X\cap C(\mathbb{R}\times \Tilde{D}_{\ep_0})\quad and \quad Y.
\end{align*}
 More precisely,   we write \eqref{phi_1-1} as follows
\begin{align*}
\phi_1(y,c)=(1-\al^2T)^{-1}1=\sum_{k=0}^{+\infty}\al^{2k}T^k1,
\end{align*}
which converges in $X\cap C(\mathbb{R}\times \Tilde{D}_{\ep_0})$ and  $Y$. Together with the $C^2$ continuity of $T$,
 we conclude that
 \begin{align*}
 \phi_1(y,c)\in C(\mathbb{R}\times \Tilde{D}_{\ep_0})\cap C^2(\mathbb{R}\times (-1,1))\;\; \text{is}\; \text{a}\; \text{solution}\;  \text{to}\;
  \eqref{phi_1-1}.
 \end{align*}
  To verify  \eqref{eq:Rayleigh-hom-S1}, it suffices to take derivatives on  \eqref{phi_1-1}.  The uniqueness is obvious.

Now we are in a position to prove (1)-(4). To prove (1), we have  by  Lemma \ref{lem:T-bound} that  $\pa_y\phi_1(y,c)=\al^2T_{2,2}\phi_1 \in C\big(\mathbb{R}\times \widetilde{D}_{\ep_0}\big)$,  and $\f{\phi_1(y,c)-1}{(u(y)-c)^2}=\f{\al^2T\phi_1(y,c)}{(u(y)-c)^2}\in C\big(\mathbb{R}\times \widetilde{D}_{\ep_0}\big)$.

 For (2), since $\phi_1(y,c)$ satisfies \eqref{eq:Rayleigh-hom-S1}, we have that $\phi_1(y,c)$ is analytic in $\widetilde{D}_{\ep_0}/(-1,1)$ with fixed $y\in\mathbb{R}$. Then along  with $\phi_1(y,c)$ continuous in $\widetilde{D}_{\ep_0}$ and symmetry principle in complex analysis,  we deduce that  $\phi_1(y,c)$ is analytic in $\widetilde{D}_{\ep_0}$. In a similar way, we obtain that  $\f{\phi_1(y,c)}{(u(y)-c)^2}$ is analytic in $\widetilde{D}_{\ep_0}$.

  To address (3), we take $c=0$ in \eqref{eq:Rayleigh-hom-S1} to get $\phi_1''-\al^2\phi_1+\f{2u'}{u}\phi_1'=0,\;\phi_1(0,0)=1,\,\phi_1'(0,0)=0$. Since $\f{2u'(y)}{u(y)}$ is odd, we can deduce that if $\phi_1(y,0)$ is a solution of \eqref{eq:Rayleigh-hom-S1} in $C\big(\mathbb{R}\times \widetilde{D}_{\ep_0}\big)\cap
C^2\big(\mathbb{R}\times (-1,1)\big),$ then $\phi_1(-y,0)$ is also a solution with the same regularity. Due to the uniqueness, we conclude that $\phi_1(y,0)$ is even in $y$. Moreover, $\pa_y\phi_1(y,0)=\f{\int_0^yu(z)^2\phi_1(z,0)dz}{u(y)^2}$ is odd in $y$.

For (4), noticing that $T$ is a positive operator in $Y$,  we have
\begin{align*}
 \left(T^k1\right)(y,c)\geq 0\quad \text{for}\; c\in(-1,1),
\end{align*}
 which implies
 \begin{align*}
 \phi_1(y,c)=1+\sum_{k=1}^{+\infty} \left(T^k1\right)(y,c)\geq 1.
 \end{align*}
Moreover, we have
 \begin{align} \label{phi_1'}
\pa_y\phi_1(y,c)=\f{\al^2\int_{y_c}^y(u(z)-c)^2\phi_1(z,c)dz}{(u(y)-c)^2}>0(<0),\;\;\;y>y_c(y<y_c).
\end{align}
Therefore, we can conclude \eqref{phi_1(z)leq phi_1(y)} and \eqref{phi_1-1leq''}. To prove \eqref{phi_1-1leq}, by \eqref{u(z)-c<u(y)-c} and \eqref{phi_1(z)leq phi_1(y)}, we have
  \begin{align*}
0\leq\phi_1(y,c)-1 \leq \al^2\phi_1(y,c)\Big|\int_{y_c}^y\int_{y_c}^zdwdz\Big|=\f12\al^2(y-y_c)^2\phi_1(y,c).
 \end{align*}
\end{proof}

\subsection{Uniform estimates in real variable}
\begin{lemma}\label{prop:phi1}
Let $\phi_1(y,c)$ be the solution constructed in Proposition \ref{prop:Rayleigh-Hom}.
For $(y,c)\in\mathbb{R}\times (-1,1)$, it holds that
\begin{align}
\label{est:calF'} e^{-\al|z-y|}&\leq \f{\phi_1(y,c)}{\phi_1(z,c)}\leq  e^{\al|z-y|}\quad \text{for}\,\, y,z\in\mathbb{R},\\
\label{est:calF}C^{-1} \al\min\{\al|y-y_c|,1\}&\leq \Big|\f{\pa_y\phi_1(y,c)}{\phi_1(y,c)}\Big|\leq \al\min\{\al|y-y_c|,1\},\\
\label{est:phi_1-ul}C^{-1}e^{C^{-1}\al|y-y_c|}&\leq \phi_1(y,c)\leq e^{\al|y-y_c|},\\
\label{est:phi_1(y)/phi_1(z)}e^{-\al|z-y|}&\leq \f{\phi_1(y,c)}{\phi_1(z,c)}\leq Ce^{-C^{-1}\al|z-y|}\quad \text{for}\;\; y_c\leq y\leq z\;\text{or}\;\;z\leq y\leq y_c,\\
\label{est:phi_1-1}C^{-1}\min\{\al^2|y-y_c|^2,1\}&\leq
\f{\phi_1(y,c)-1}{\phi_1(y,c)}\leq C\min\{\al^2|y-y_c|^2,1\},
\end{align}
where the constants $C$ is independent of $c$, $\al$, $y$, $z$.
\end{lemma}

\begin{proof}
  We define
$\mathcal{F}(y,c)=\f{\pa_y\phi_1(y,c)}{\phi_1(y,c)}$, which satisfies
\begin{align}
\label{eq:mathcalF}&\mathcal{F}'+\mathcal{F}^2+\f{2u'}{u-c}\mathcal{F}=\al^2,\;\mathcal{F}(y_c,c)=0.
\end{align}
We claim that
\begin{align}\label{bd:calF}
\mathcal{F}(y,c)\; \text{has}\; \text{the}\; \text{same}\; \text{sign}\; \text{with}\; y-y_c,\quad\text{and}\quad |\mathcal{F}(y,c)| \leq \al\quad \text{for}\,\, (y,c)\in\mathbb{R}\times(-1,1).
\end{align}
Indeed,  the first assertion of \eqref{bd:calF} follows from \eqref{phi_1-1leq''}, which 
along with \eqref{eq:mathcalF} gives \begin{align*}
\mathcal{F}'+\mathcal{F}^2\leq \al^2.
 \end{align*}
 To address the second assertion  in \eqref{bd:calF}, it suffices to consider
  \begin{align*}
 (y,c)\in [-M,M]\times [u(-M),u(M)]\; \;\text{with}\; \text{arbitrary}\; M>0.
 \end{align*}
 In fact, for $c\in[u(-M),u(M)]$, $y\in[y_c,M]$, we take $y_0$ such that $\mathcal{F}(y_0,c)=\sup_{y\in[y_c,M]}\mathcal{F}(y,c)$. Therefore, it holds that
 \begin{align*}
 \mathcal{F}'(y_0,c)\geq 0,\; \;\text{and}\;\; \al^2\geq \mathcal{F}'(y_0,c)+\mathcal{F}^2(y_0,c)\geq \mathcal{F}^2(y_0,c),
 \end{align*}
 which gives  $0\leq \mathcal{F}(y,c)\leq \al$ for $y\in[y_c,M]$. Similarly, we take the minimum point on $[-M,y_c]$ to get   $-\al\leq \mathcal{F}(y,c)\leq 0$ for $y\in[-M,y_c]$.
 This proves  the claim.\smallskip

Now we are in a position to prove the estimates \eqref{est:calF'}-\eqref{est:phi_1-1}. The estimate \eqref{est:calF'} follows by  integrating $\mathcal{F}(w,c)=(\ln \phi_1(w,c))'$
over $[z,y]$ or $[y,z]$ and \eqref{bd:calF}.

 To prove \eqref{est:calF},  we have by  \eqref{phi_1-1leq''} that $|\mathcal{F}(y,c)|\leq \al^2|y-y_c|$.  Together with \eqref{bd:calF}, we obtain the upper bound. It will take us more effort to  prove the lower bound. It suffices to prove \begin{align*}
 |\mathcal{F}(y,c)|\geq C^{-1}\al^2y_*\quad \text{with}\quad y_*=\min\{\f{|y-y_c|}{2},\f{1}{\al}\}.
  \end{align*}
  Indeed, for $y\geq y_c$, 
 we obtain by \eqref{phi_1-1leq''}, \eqref{est:calF'} that
\begin{align*}
|\mathcal{F}(y,c)|=\mathcal{F}(y,c)&= \f{\al^2}{(u(y)-c)^2}\int_{y_c}^y(u(z)-c)^2\f{\phi_1(z,c)}{\phi_1(y,c)}dz\\
&\geq \f{C^{-1}\al^2}{(u(y)-c)^2}\int_{y-y_*}^y(u(z)-c)^2e^{-\al(y-z)}dz\\
&\geq C^{-1}\al^2y_*e^{-\al y_*}\big(\f{u(y-y_*)-c}{u(y)-c}\big)^2\geq C^{-1}\al^2y_*,
\end{align*}
here we have used
\begin{align*}
\f{u(y-y_*)-c}{u(y)-c}=\f{\sinh(y-y_c-y_*)\cosh y}{\sinh(y-y_c)\cosh(y-y_*)}\geq C^{-1}.
\end{align*}
For $y\leq y_c$, we can similarly prove  $|\mathcal{F}(y,c)|=-\mathcal{F}(y,c)\geq C^{-1}\al^2y_*$.

  For \eqref{est:phi_1-ul}, the upper bound  can be obtained by taking $z=y_c$ in \eqref{est:calF'}. Now we prove the lower bound. Notice that for $y\geq y_c$,
\begin{align*}
\ln\phi_1(y,c)&=\int_{y_c}^y\mathcal{F}(z,c)dz\geq  C^{-1}\min\{\al^2|y-y_c|^2,\al|y-y_c|\}.
\end{align*}
This shows that for $y\geq y_c$,
\begin{align*}
\phi_1(y,c)\geq e^{C^{-1}\min\{\al^2|y-y_c|^2,\al|y-y_c|\}}\geq C^{-1}e^{C^{-1}\al|y-y_c|}.
\end{align*}
The same argument holds for $y\leq y_c$. Thus, we conclude the lower bound in \eqref{est:phi_1-ul}.

To prove \eqref{est:phi_1(y)/phi_1(z)}, it suffices to show the upper bound. Indeed, for $y_c\leq y\leq z$, we have by the lower bound of \eqref{est:calF} that
\begin{align*}
\ln\f{\phi_1(y,c)}{\phi_1(z,c)}=-\int_{y}^z\mathcal{F}(w,c)dw\leq
-\int_{y}^zC^{-1}\al dw=-C^{-1}\al|z-y|.
 \end{align*}
The same estimate holds  for $y_c\leq y\leq z$.

For \eqref{est:phi_1-1},  the upper bound follows from \eqref{phi_1-1leq}. To obtain the lower bound, thanks to \eqref{est:phi_1-ul} and  $1-e^{-x}\geq C^{-1}\min\{x,1\}(x>0)$, we infer that
\begin{align*}
1-\f{1}{\phi_1(y,c)}&\geq 1-e^{-C^{-1}\min\{\al^2|y-y_c|^2,\al|y-y_c|\}}\\
&\geq C^{-1}\min\Big\{C^{-1}\min\{\al^2|y-y_c|^2,\al|y-y_c|\},1\Big\} \\
&\geq C^{-1}\min\{\al^2|y-y_c|^2,1\}.
\end{align*}

\end{proof}

Let's investigate the estimates for good derivatives, see \eqref{def:goodG}. Direct calculation gives
\begin{align}\label{good:phi1}
&\pa_G^k(fg)=\sum_{i=0}^kC_k^i\cdot\pa_G^ig\cdot\pa_G^{k-i}f,\quad\pa_{G}\int_{y_c}^yf(z,c)dz=\int_{y_c}^y\pa_{G}f(z,c)dz.
\end{align}

\begin{lemma}\label{prop:goodphi1}
 Let $(y,c)\in\mathbb{R}\times (-1,1)$, $k=1,2,3,4$. Then it holds that
 \begin{itemize}
 \item
 $\pa_G^k\phi_1(y,c)\in C\big(\mathbb{R}\times (-1,1)\big).$
 \item  There exists a constant $C$ independent of $c$, $y$, $\al$ such that
\begin{align}
\label{est:goodk}&\Big|\f{\pa_G^k\phi_1(y,c)}{\phi_1(y,c)}\Big|
\leq C\left(\f{|u(y)-c|}{1-c^2}\right)^{k}
\cdot\al|y-y_c|\cdot\langle\al|y-y_c|\rangle^{k-1}
\end{align}

\item $\pa_G\phi_1(y,0)$ is an odd function.
\end{itemize}

\end{lemma}
\begin{proof}

 Recall   $\mathcal{F}(y,c)=\f{\pa_y\phi_1(y,c)}{\phi_1(y,c)}$ satisfies  \eqref{eq:mathcalF}. Then it follows from  Proposition \ref{prop:Rayleigh-Hom} (1), (3) that
\begin{align}\label{ODD:F}
\mathcal{F}(y,c)\in C(\mathbb{R}\times (-1,1)),\quad \text{and}\quad \mathcal{F}(-y,0)=-\mathcal{F}(y,0).
\end{align}
Define $\widetilde{\mathcal{F}}(y,c):=\min\{\al^2|y-y_c|,\al\}$.  For $y_c\leq z\leq y$ or $y\leq z\leq y_c$, we have by  \eqref{est:calF} that
\begin{align}\label{est:min{al|z-y_c|,1}}
\mathcal{F}(z,c)\sim \widetilde{\mathcal{F}}(z,c)\leq \widetilde{\mathcal{F}}(y,c),
\end{align}
and
\begin{align}\label{est:min{al|z-y_c|,1'}}
\Big|\int_{y_c}^y\mathcal{F}(z,c)dz
\Big|\sim\Big|\int_{y_c}^y\widetilde{\mathcal{F}}(z,c)dz
\Big|\sim\min\{\al^2|y-y_c|^2,\al|y-y_c|\}\leq C\al|y-y_c|.
\end{align}
 We also define
\begin{align}\label{eq:mathcalG} \mathcal{G}(y,c):=\f{\pa_c\phi_1(y,c)}{\phi_1(y,c)},\;
\mathcal{G}_1(y,c):=\f{\mathcal{F}}{u'(y_c)}+\mathcal{G}
=\f{\pa_G\phi_1(y,c)}{\phi_1(y,c)}.
\end{align}
 Let $k\in\mathbb{Z}^+$. Taking $\pa_G^k$ on  \eqref{phi_1-1},  we get by \eqref{good:phi1} that
\begin{align}\label{paGkphi1}
\pa_G^k\phi_1(y,c)&=\al^2T_0\pa_G^kT_{2,2}\phi_1
=\al^2\int_{y_c}^y\pa_{G}^k\Big(\f{u'(y_c)^2}{(u(z)-c)^2}
\cdot\f{\int_{y_c}^z(u(w)-c)^2\phi_1(w,c)dw}{u'(y_c)^2}\Big)dz,
\end{align}
which together with \eqref{good:phi1} and Lemma \ref{lem:simple-useful-equ} gives
\begin{align}\label{paGphi1(yc,c)}
\pa_{G}^k\phi_1(y_c,c)=0,\quad \pa_y\pa_{G}^k\phi_1(y_c,c)=0.
\end{align}
Thanks to the definitions and \eqref{paGphi1(yc,c)}, we have
\begin{align}
\label{eq:mathcalG-F}
&\pa_c\mathcal{F}=\pa_y\mathcal{G},\; \pa_G^{k-1}\mathcal{G}_1=\int_{y_c}^y\pa_G^{k}\mathcal{F}dz,\;\mathcal{F}(y_c,c)=\mathcal{G}(y_c,c)=0,\;\pa_G^{k}\mathcal{G}_1(y_c,c)=0.
\end{align}
Using \eqref{eq:mathcalG}-\eqref{eq:mathcalG-F}, we have
\begin{align}
\label{1phi1}&\f{\pa_G\phi_1}{\phi_1}
=\mathcal{G}_1=\int_{y_c}^y\pa_G\mathcal{F}(z,c)dz,\\
\label{2phi1}&\f{\pa_G^2\phi_1}{\phi_1}
=\int_{y_c}^y\pa_G^2\mathcal{F}(z,c)dz+\f{(\pa_G\phi_1)^2}{\phi_1^2},\\
\label{3phi1}&\f{\pa_G^3\phi_1}{\phi_1}
=\int_{y_c}^y\pa_G^3\mathcal{F}(z,c)dz+\f{3\pa_G^2\phi_1\pa_G\phi_1}{\phi_1^2}
-\f{2(\pa_G\phi_1)^3}{\phi_1^3},\\
\label{4phi1}&\f{\pa_G^4\phi_1}{\phi_1}
=\int_{y_c}^y\pa_G^4\mathcal{F}(z,c)dz+\f{4\pa_G^2\phi_1\pa_G\phi_1}{\phi_1^2}
+\f{3(\pa_G^2\phi_1)^2}{\phi_1^2}-\f{12\pa_G^2\phi_1(\pa_G\phi_1)^2}{\phi_1^3}
+\f{6(\pa_G\phi_1)^4}{\phi_1^4}.
\end{align}
By \eqref{eq:mathcalF}, we have
\begin{align}\label{calF:deri}
&(\phi^2\pa_c\mathcal{F})'=-2u'\phi_1^2\mathcal{F},\;(\phi^2\mathcal{F}')'=2\big((u')^2-u''(u-c)\big)\phi_1^2\mathcal{F}
\end{align}
with  $\phi=(u-c)\phi_1$. We get by Lemma \ref{lem:simple-useful-equ} that
\begin{align*}
u''(u-c)+u'(u'(y_c)-u')=-u'(u-c)^2.
\end{align*}
Therefore, we have by \eqref{calF:deri} that
\begin{align}\label{CAlF}
\pa_G\mathcal{F}(z,c)
&=\f{2\int_{y_c}^zu'(w)(u(w)-c)^2\phi_1(w,c)^2\mathcal{F}(w,c)dw}{u'(y_c)(u(z)-c)^2\phi_1(z,c)^2}.
\end{align}
Taking $\pa_G^{k-1}$ on \eqref{CAlF}, we  get
\begin{align}\label{eq:paGF}
\notag\quad\pa_G^{k}\mathcal{F}(z,c)&=2\pa_{G}^{k-1}\Big(\f{(1-c^2)^2}{(u-c)^2}\cdot\phi^{-2}_1\Big)\cdot\int_{y_c}^z\f{u'(u-c)^2}{(1-c^2)^3}\cdot \phi_1^2\cdot\mathcal{F}dw\\
&\quad+(\f{1-c^2}{u-c})^2\cdot \phi^{-2}_1\cdot\int_{y_c}^z\pa_{G}^{k-1}\Big(\f{u'(u-c)^2}{(1-c^2)^3}\cdot \phi_1^2\cdot\mathcal{F}\Big)dw.
\end{align}

\vspace{\baselineskip}

Now we are in a position to prove the lemma. Using \eqref{ODD:F} and \eqref{est:min{al|z-y_c|,1}}-\eqref{est:min{al|z-y_c|,1'}},  \eqref{eq:paGF}, we can prove by the induction that
\begin{align*}
\pa_G^k\mathcal{F}(y,c)\in C\big(\mathbb{R}\times (-1,1)\big),\quad \text{and}\quad \pa_G^k\phi_1(y,c)\in C\big(\mathbb{R}\times (-1,1)\big),\quad \text{for}\; k\in \mathbb{Z}^+,
\end{align*}
which gives the first assertion of the lemma.

To prove \eqref{est:goodk}, we observe by Lemma \ref{lem:simple-useful-equ} for
the terms in \eqref{eq:paGF}  that
\begin{align*} &\pa_G^i\left(\f{u'(u-c)^2}{(1-c^2)^3}\right)
=\f{(-1)^i(i+3)!}{3!}\cdot\f{u'}{(1-c^2)}\cdot
\left(\f{u-c}{1-c^2}\right)^{i+2},\quad i\in \mathbb{Z}^+,
\end{align*}
which gives
\begin{align}
\label{est:u'/u-c1}&\bigg|\f{\pa_G^i\left(\f{(1-c^2)^2}{(u-c)^2}\right)}{\f{(1-c^2)^2}{(u-c)^2}}\bigg|\leq C\Big(\f{|u-c|}{1-c^2}\Big)^{i},\;\bigg|\f{\pa_G^i\left(\f{u'(u-c)^2}
{(1-c^2)^3}\right)}{\f{u'(u-c)^2}{(1-c^2)^3}}\bigg|\leq C\left(\f{|u-c|}{1-c^2}\right)^{i},\;i\in \mathbb{Z}^+.
\end{align}
We claim that
\begin{align}\label{est:paGkF}
|\pa_G^k\mathcal{F}(z,c)|
\leq C\left(\f{|u(z)-c|}{1-c^2}\right)^{k}\cdot\tilde{\mathcal{F}}(z,c)\cdot\langle \al|z-y_c|\rangle ^{k-1},\;\; k=1,2,3,4.
\end{align}
We observe in the estimates \eqref{est:goodk}, \eqref{est:u'/u-c1}, \eqref{est:paGkF} that \textbf{each  good derivative $\pa_G$ gains a factor $\f{|u-c|}{1-c^2}$}. To establish the bounds \eqref{est:paGkF} and  \eqref{est:goodk}  through the induction on $k$, we will use the estimates \eqref{est:min{al|z-y_c|,1}}-\eqref{est:min{al|z-y_c|,1'}}, \eqref{est:u'/u-c1} and equations \eqref{1phi1}-\eqref{eq:paGF}. The order of the induction is as follows:
\eqref{est:paGkF}($k=1$), \eqref{est:goodk}($k=1$), \eqref{est:paGkF}($k=2$), \eqref{est:goodk}($k=2$), \eqref{est:paGkF}($k=3$), \eqref{est:goodk}($k=3$), \eqref{est:paGkF}($k=4$), \eqref{est:goodk}($k=4$). Let us just show the first several steps. The proof of the rest steps are similar and left to readers.

 Thanks to \eqref{eq:paGF}($k=1$), \eqref{u(z)-c<u(y)-c}, \eqref{phi_1(z)leq phi_1(y)} and \eqref{est:min{al|z-y_c|,1}}, we have
\begin{align*}
|\pa_G\mathcal{F}(z,c)|
&\leq \f{C\Big|\int_{y_c}^zu'(w)dw\Big|\widetilde{\mathcal{F}}(z,c)}{1-c^2}
\leq \f{C|u(z)-c|\widetilde{\mathcal{F}}(z,c)}{1-c^2}.
\end{align*}
This shows \eqref{est:paGkF} for $k=1$.
 Thanks to \eqref{1phi1}, \eqref{u(z)-c<u(y)-c}, \eqref{est:paGkF}($k=1$), \eqref{est:min{al|z-y_c|,1'}}, we have
\begin{align*}
&\quad\Big|\f{\pa_G\phi_1(y,c)}{\phi_1(y,c)}\Big|\leq\int_{y_c}^y|\pa_G\mathcal{F}(z,c)|dz\leq \f{|u(y)-c|\cdot|\int_{y_c}^y\tilde{\mathcal{F}}(z,c)dz|}{1-c^2}\leq \f{C|u(y)-c|\cdot\al|y-y_c|}{1-c^2}.
\end{align*}
This shows  \eqref{est:goodk} for $k=1$.
 Thanks to \eqref{eq:paGF}($k=2$), \eqref{good:phi1}, \eqref{est:u'/u-c1}($i=1$), \eqref{u(z)-c<u(y)-c}, \eqref{phi_1(z)leq phi_1(y)}, \eqref{est:min{al|z-y_c|,1}}, \eqref{est:goodk}($k=1$) and  \eqref{est:paGkF}($k=1$),  we have
\begin{align*}
|\pa_G^2\mathcal{F}(z,c)|
&\leq
C\Big(\f{|u(y)-c|}{1-c^2}\Big)^{2}\tilde{\mathcal{F}}(z,c)\cdot\langle\al|z-y_c|\rangle.
\end{align*}
This shows \eqref{est:paGkF} for $k=2$.
Thanks to \eqref{eq:mathcalG-F}, we have
\begin{align}
\pa_G\mathcal{G}_1=\f{\pa_G^2\phi_1}{\phi_1}
-\f{(\pa_G\phi_1)^2}{\phi_1^2},\; \pa_G^2\mathcal{F}=\pa_z\pa_G\mathcal{G},\; \pa_G\mathcal{G}_1(y_c,c)=0.
\end{align}
Therefore,  we have  by  \eqref{est:goodk}($k=1$) and \eqref{est:paGkF}($k=2$) that
\begin{align*}
\notag\Big|\f{\pa_G^2\phi_1}{\phi_1}\Big|
&\leq (\f{\pa_G\phi_1}{\phi_1})^2+\pa_G\mathcal{G}_1
\leq (\f{\pa_G\phi_1}{\phi_1})^2+\int_{y_c}^y|\pa_G^2\mathcal{F}(z,c)|dz\\
&\leq C(1-c^2)^{-2}|u(y)-c|^2\al|y-y_c|\cdot\langle\al|y-y_c|\rangle.
\end{align*}
This shows \eqref{est:goodk} for $k=2$.

To prove the last assertion in the lemma, we take $c=0$ in \eqref{CAlF}. Since  $\mathcal{F}(w,0)$ is odd (by \eqref{ODD:F}), it holds that $\phi_1(y,0)$, $u'(y)$ and $u(y)^2$ are even. The oddness of $\pa_G\mathcal{F}(y,0)$ follows.  Combining with \eqref{1phi1}, we obtain the oddness of  $\pa_G\phi_1(y,0)$.
\end{proof}

\subsection{Uniform estimates in complex variable}
 Let $c_{\ep}=c+i\ep$. Recall the complex domains  defined in \eqref{domain:O-ep}
 \begin{align}
\label{def:O-ep}
O_{\ep_0}&:=\big\{c_{\ep}|\;c\in(-1,1),\; 0<|\ep|< \min\{C_o^{-1}(1-c^2), \ep_0\}\big\},\;\widetilde{O}_{\ep_0}:=O_{\ep_0}\cup(-1,1),
\end{align}
see Figure 2.
\begin{figure}[!h]
\centering
\includegraphics[width=0.9\textwidth]{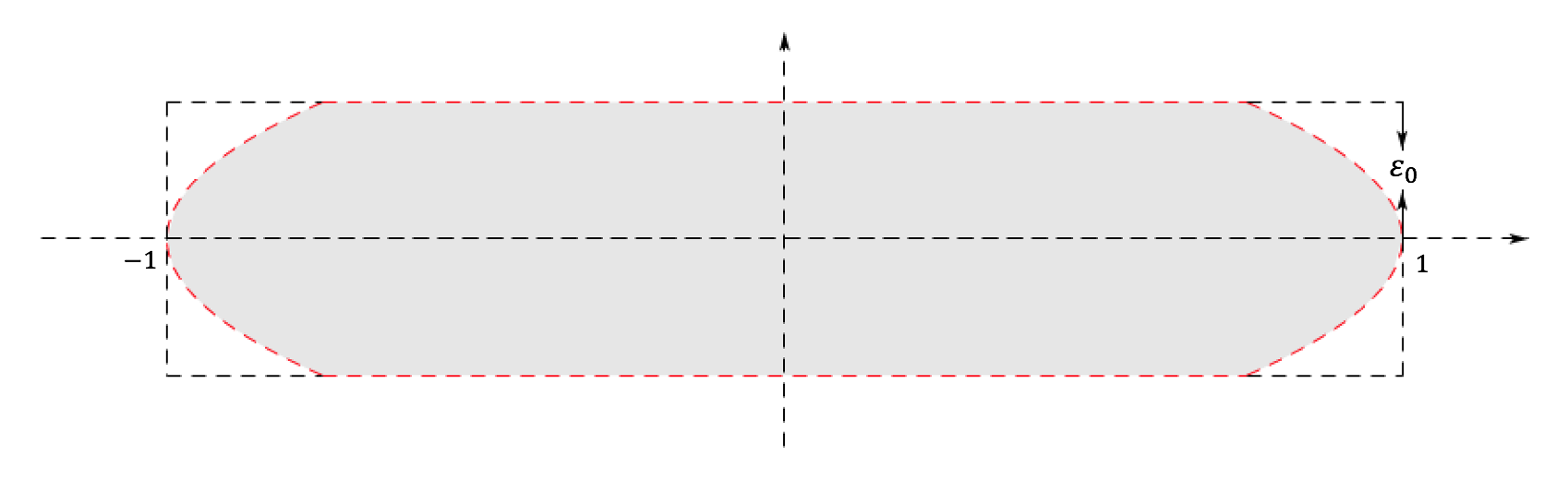}
\caption{ $O_{\varepsilon_0}$\;\text{and}\; $ {\tilde{O}}_{\varepsilon_0}$.}
\label{fig2}
\end{figure}
The constant $C_o$ is chosen in the following proposition.
 \begin{Proposition}\label{cor:phi_1-infty}
Let $\phi_1$ be the solution constructed  in Proposition \ref{prop:Rayleigh-Hom}. There exists universal constants $C_o$ and $\ep_1$ such that for  any $\ep_0\in(0,\ep_1]$ and $(y,c_{\ep})\in\mathbb{R}\times O_{\ep_0}$, it holds that
  \begin{align}
\label{est:up-lower-phi-1}&\f12\leq \Big|\f{\phi_1(y,c_{\ep})}{\phi_1(y,c)} \Big|\leq \f32,\\
\label{est:phi_1-ul-c}&C^{-1}e^{C^{-1}\al|y-y_c|}\leq \phi_1(y,c_{\ep})\leq e^{\al|y-y_c|},\\
&\Big|\f{\phi_1(y,c_{\ep})-1}{\phi_1(y,c_{\ep})}\Big|
+\Big|\f{\phi_1(y,c_{\ep})^2-1}{\phi_1(y,c_{\ep})^2}\Big|\leq C\min\{\al^2|y-y_c|^2,1\},\label{est:up-phi-1}
\end{align}
 where the  constants $C$ is independent of $\ep$, $c$, $y$.
\end{Proposition}

\begin{proof}
Fix $c\in(-1,1)$. Let $\ep\in(0,\ep_0)$.  Denote
\begin{align*}
f_{\ep}(y,c):=\phi_1(y,c_{\ep})-\phi_1(y,c),\quad
\text{and}\quad
\|f\|_{Y_A}:=\|\f{f(y,c)}{\cosh A(y-y_c)}\|_{L^{\infty}_y(\mathbb{R})}\;\text{with}\;\;A>\al.
\end{align*}
Let us first prove that
\begin{align}\label{est:fepX}
\|f_{\ep}\|_{Y_A}+\|f_{\ep}'\|_{Y_A}\leq C(1-c^2)^{-1}|\ep|.
\end{align}
Here the constant $C$ is independent of  $\ep$, $c$, but it may depend on $\al$. It is easy to verify  that
\begin{align}\label{u-u}
|(u-c_{\ep})^2-(u-c)^2|\leq 2|\ep|\cdot|u-c_{\ep}|,\;
\Big|\f{1}{(u-c_{\ep})^2}-\f{1}{(u-c)^2}\Big|\leq \f{2|\ep|}{|u-c|^3}.
\end{align}
Using \eqref{phi_1-1},  \eqref{u-u} and  \eqref{u(z)-c<u(y)-c}, we have
\begin{align}
\notag|f_{\ep}(y,c)|&\leq \al^2\Bigg|\int_{y_c}^y\f{\int_{y_c}^z|u(w)-c_{\ep}|^2
f_{\ep}(w,c)dw}{|u(z)-c_{\ep}|^2}dz\Bigg|
+2|\ep|\al^2\Bigg|\int_{y_c}^y\f{\int_{y_c}^z|u(w)-c_{\ep}
|\phi_1(w,c)dw}{|u(z)-c_{\ep}|^2}dz\Bigg|\\
\notag&\quad+2|\ep|\al^2\Bigg|\int_{y_c}^y
\f{\int_{y_c}^z|u(w)-c|^2\phi_1(w,c)dw}{|u(z)-c|^3}dz\Bigg|\\
\label{fep}&\leq \al^2\Bigg|\int_{y_c}^y\int_{y_c}^zf_{\ep}(w,c)dwdz\Bigg|
+4|\ep|\al^2\bigg|\int_{y_c}^y\f{\int_{y_c}^z\phi_1(w,c)dw}
{|u(z)-c|}dz\bigg|.\\
\notag|f_{\ep}'(y,c)|&\leq \al^2\Bigg|\f{\int_{y_c}^y|u(w)-c_{\ep}|^2f_{\ep}(w)dw}
{|u(y)-c_{\ep}|^2}\Bigg|
+2|\ep|\al^2\Bigg|\f{\int_{y_c}^y|u(w)-c_{\ep}
|\phi_1(w,c)dw}{|u(y)-c_{\ep}|^2}\Bigg|\\
\notag&\quad+2|\ep|\al^2\Bigg|\f{\int_{y_c}^y
|u(w)-c|^2\phi_1(w,c)dw}{|u(y)-c|^3}\Bigg|\\
\label{f'ep}&\leq \al^2\bigg|\int_{y_c}^yf_{\ep}(w)dw\bigg|
+\f{4|\ep|\al^2\Big|\int_{y_c}^y\phi_1(w,c)dw\Big|}{|u(y)-c|}.
\end{align}
Therefore, thanks to \eqref{fep}, \eqref{est:phi_1-ul}($\phi_1(w,c)\lesssim\cosh \al(w-y_c)$) and \eqref{u-key}, we infer that
\begin{align*}
\|f_{\ep}\|_{Y_A}&\leq \f{\al^2}{A^2}\|f_{\ep}\|_{Y_A}
+C|\ep|\al^2\Big\|\f{\int_{y_c}^y\f{\al\int_{y_c}^z\cosh \al(w-y_c)dw}{(1-c^2)\tanh \al|z-y_c|}
dz}{\cosh A(y-y_c)}\Big\|_{L^{\infty}_y}\leq
\f{\al^2}{A^2}\|f_{\ep}\|_{Y_A}+C\al(1-c^2)^{-1}|\ep|,
\end{align*}
  which gives $\|f_{\ep}\|_{Y_A}\leq C(1-c^2)^{-1}|\ep|.$ This along with  \eqref{f'ep} shows
\begin{align*}
\|f_{\ep}'\|_{Y_A}&\leq \f{\al^2}{A}\|f_{\ep}\|_{Y_A}
+C|\ep|\al^2\Big\|\f{\al\int_{y_c}^z\cosh \al(w-y_c)dw}{(1-c^2)\tanh \al|y-y_c|\cosh A(y-y_c)}\Big\|_{L^{\infty}_y}\\
&\leq C(1-c^2)^{-1}|\ep|.
\end{align*}
Thus, we finish the proof of \eqref{est:fepX}.\smallskip

Now we are in a position to prove \eqref{est:up-lower-phi-1}-\eqref{est:up-phi-1}. For $|y-y_c|\leq 2$, we use \eqref{est:fepX}, \eqref{phi_1(z)leq phi_1(y)} and $\cosh A|y-y_c|\leq C$ to obtain
\begin{align}\label{est:phi1-phi12}
 &\quad\Big|\f{\phi_1(y,c+i\ep)-\phi_1(y,c)}{\phi_1(y,c)} \Big|\leq \Big\|f_{\ep}\Big\|_{Y_A}\cdot
 \Big|\f{\cosh A(y-y_c)}{\phi_1(y,c)} \Big|\leq C(1-c^2)^{-1}|\ep|.
 \end{align}
We denote
\begin{align*}
J_1=[y_c-1,y_c+1],\; J_2=[y_c-2,y_c+2],
\end{align*}
from which, it follows that
\begin{align*}
J_{1}^c=(-\infty,y_c-1)\cup(y_c+1,+\infty),\;\;\text{and}\; \; J_2/J_1=[y_c-2,y_c-1)\cup(y_c+1,y_c+2].
\end{align*}
Moreover, we introduce a smooth cut-off function $\chi(y)$ which takes $0$ on $J_1$ and $1$ on $J_2^c$, with $|\chi''(y)|+|\chi''(y)|\leq C$ on $J_2/J_1$. For $|y-y_c|\geq 1$, we claim that
\begin{align}\label{est:ferJ}
\notag&\Big\|\f{f_{\ep}(y,c)\chi(y)}{\phi_1(y,c)}\Big\|_{L^{\infty}_y(J_1^c)}\\
&\leq C(1-c^2)^{-1}|\ep| \Big\|\f{f_{\ep}(y,c)\chi(y)}{\phi_1(y,c)}\Big\|_{L^{\infty}_y(J_1^c)}
+\f12C(1-c^2)^{-1}|\ep|+C^2(1-c^2)^{-2}|\ep|^2,
\end{align}
with constant $C$ independent of $\ep$, $c$, $y$, $\ep_0$.

Therefore, the bound \eqref{est:up-lower-phi-1} can be derived from \eqref{est:phi1-phi12} and  \eqref{est:ferJ} by choosing $C_o=4C$. The bound \eqref{est:phi_1-ul-c} follows directly from \eqref{est:up-lower-phi-1} and \eqref{est:phi_1-ul}. To obtain the bound \eqref{est:up-phi-1}, we use \eqref{phi_1-1}, \eqref{u(z)-c<u(y)-c}, \eqref{est:phi_1(y)/phi_1(z)} and \eqref{est:up-lower-phi-1} to get for   $(y,c_{\ep})\in \mathbb{R}\times O_{\ep_0}$,
\begin{align*}
\Big|\f{\phi_1(y,c_{\ep})-1}{\phi_1(y,c_{\ep})}\Big|
&\leq  \bigg|\f{\al^2\int_{y_c}^y\f{\int_{y_c}^z(u(w)-c_{\ep})^2
\phi_1(w,c_{\ep})dw}{(u(z)-c_{\ep})^2}dz}{\phi_1(y,c_{\ep})}\bigg|\leq 3\al^2\Big|
\int_{y_c}^y\int_{y_c}^z\f{\phi_1(w,c)}{\phi_1(y,c)}dwdz
\Big|\\
&\leq C\al^2\Big|
\int_{y_c}^y\int_{y_c}^ze^{-C^{-1}\al|y-w|}dwdz
\Big|= C\al^2\Big|
\int_{y_c}^y\int^{y}_we^{-C^{-1}\al|y-w|}dzdw\Big|\\
&\leq C
\int_0^{\al|y-y_c|}xe^{-C^{-1}x}dx\leq C\min\{\al^2|y-y_c|^2,1\}.
\end{align*}
\vspace{\baselineskip}

It remains to prove the claim \eqref{est:ferJ}. Thanks to \eqref{fm:u(y)-c}, we have
\begin{align}\label{est:u'/u-c}
&\f{\big(1-c^2+u'(z)\big)\textbf{1}_{J_1^c}(z)}{|u(z)-c|}\leq \f{C\textbf{1}_{J_1^c}(z)}{\tanh|z-y_c|}\leq C\textbf{1}_{J_1^c}(z).
\end{align}
For $y_c\leq w\leq z$ or
$z\leq w\leq y_c$, we get by  \eqref{sinhy} that
\begin{align}\label{est:u'/|u-c|^2}
&\f{u'(w)\textbf{1}_{J_1^c}(z)}{|u(z)-c|^2}=
\f{\cosh^2z\cosh^2y_c\textbf{1}_{J_1^c}(z)}{\cosh^2 w\sinh^2|z-y_c|}\leq
\f{e^{2|z-w|}\textbf{1}_{J_1^c}(z)}{(1-c^2)e^{2|z-y_c|}}
\leq (1-c^2)^{-1}e^{-2|w-y_c|}.
\end{align}
Using \eqref{eq:Rayleigh-hom-S1}, we have
\begin{align*}
f_{\ep}''-\al^2 f_{\ep}+\f{2u'f'_{\ep}}{u-c_{\ep}}=-\f{2i\ep u'\phi_1'(y,c)}{(u-c_{\ep})(u-c)},\;f_{\ep}(y_c,c)=f'_{\ep}(y_c,c)=0,
\end{align*}
which gives
\begin{align}\label{eq:fepchi}
(f_{\ep}\chi)''-\al^2 (f_{\ep}\chi)+\f{2u'(f_{\ep}\chi)'}{u-c_{\ep}}=-\f{2i\ep u'\phi_1'(y,c)\chi(y)}{(u-c_{\ep})(u-c)}+g_{\chi,\ep}(y,c),
\end{align}
with
\begin{align*}
g_{\chi,\ep}:=2f_{\ep}'\chi'+f_{\ep}\chi''+\f{2u'}{u-c_{\ep}}f_{\ep}\chi'.
\end{align*}
Thanks to \eqref{est:u'/u-c} and  $J_2/J_1\subset J_1^c$, we have
\begin{align}\label{est:gchi}
|g_{\chi,\ep}(w,c)|\leq C(|f_{\ep}(w,c)|+|f_{\ep}'(w,c)|)\textbf{1}_{J_2/J_1}(w).
\end{align}
By multiplying $(u-c_{\ep})(u-c)\phi_1$ to equation \eqref{eq:fepchi} and
 $(u-c_{\ep})(u-c)f_{\ep}\chi$ to \eqref{eq:Rayleigh-hom-S1}, and subsequently taking the difference, we obtain the following equation
\begin{align*}
\left((u-c)^2\phi_1^2\left(\f{(u-c_{\ep})f_{\ep}\chi}{(u-c)\phi_1}\right)'\right)'
=i\ep u''\phi_1 f_{\ep}\chi-2i\ep u'\phi_1'\phi_1\chi
+(u-c_{\ep})(u-c)\phi_1 g_{\chi,\ep}.
\end{align*}
The integral form takes
\begin{align}\label{eq:integal-fep}
\notag\f{f_{\ep}(y,c)\chi(y)}{\phi_1(y,c)}&=
i\ep\f{u(y)-c}{u(y)-c_{\ep}}\int_{y_c}^y\f{\int_{y_c}^zu''(w)\phi_1(w,c)f_{\ep}(w,c)\chi(w)dw}{(u(z)-c)^2\phi_1(z,c)^2}dz\\
\notag&\quad-2i\ep \f{u(y)-c}{u(y)-c_{\ep}}\int_{y_c}^y\f{\int_{y_c}^zu'(w)\phi_1'(w,c)\phi_1(w,c)\chi(w)dw}{(u(z)-c)^2\phi_1(z,c)^2}dz\\
&\quad+\f{u(y)-c}{u(y)-c_{\ep}}\int_{y_c}^y\f{\int_{y_c}^z(u(w)-c_{\ep})(u(w)-c)\phi_1(w,c)g_{\chi,\ep}(w,c)dw}{(u(z)-c)^2\phi_1(z,c)^2}dz
=I_1+I_2+I_3.
\end{align}

Thanks to \eqref{u(z)-c<u(y)-c}, \eqref{est:u'/|u-c|^2} and  \eqref{est:phi_1(y)/phi_1(z)}, 
there exists constant $C$ independent of $c$, $y$, $\ep$ such that
\begin{align}
\notag F(y,c)&:=\Big|\int_{y_c}^y\int_{y_c}^z\f{u'(w)}{|u(z)-c|^2}\cdot \f{\phi_1(w,c)^2}{\phi_1(z,c)^2}\cdot \textbf{1}_{J_1^c}(z)dwdz\Big|\\
\notag&\lesssim (1-c^2)^{-1}\Big|\int_{y_c}^y\int_{y_c}^ze^{-2|w-y_c|}\cdot e^{-2C^{-1}\al|z-w|}dwdz\Big|\\
\label{est:F(y,c)}&\lesssim (1-c^2)^{-1}\Big|\int_{y_c}^ye^{-2|w-y_c|}\int_{w}^y e^{-2C^{-1}\al|z-w|}dzdw\Big|\leq C(1-c^2)^{-1},
\end{align}
which along with
$\Big|\f{u(y)-c}{u(y)-c_{\ep}}\Big|\leq 1$, $|\f{u''}{u'}|\leq 2$ and \eqref{est:F(y,c)} implies
\begin{align*}
|I_1|&\leq |\ep|\cdot F(y,c)\cdot \Big\|\f{f_{\ep}\chi}{\phi_1}\Big\|_{L^{\infty}_y(\mathbb{R})}\leq C(1-c^2)^{-1}|\ep| \Big\|\f{f_{\ep}\chi}{\phi_1}\Big\|_{L^{\infty}_y(\mathbb{R})},\\
|I_2|&\leq 2|\ep|\cdot F(y,c)\cdot \Big\|\f{\phi_1'}{\phi_1}\Big\|_{L^{\infty}_y(\mathbb{R})}\leq C(1-c^2)^{-1}|\ep|\al.\notag
\end{align*}
Thanks to $|u(w)-c_{\ep}|\leq |u(w)-c|+|\ep|$, \eqref{u(z)-c<u(y)-c}, \eqref{est:gchi}, $J_2/J_1\subset J_1^c$, \eqref{est:fepX} and \eqref{u-key}, we have
\begin{align*}
|I_3|&\lesssim \Big|\int_{y_c}^y\int_{y_c}^z\big(1+\f{|\ep|}{|u(z)-c|}\big)
\textbf{1}_{J_1^c}(z)\cdot\f{(|f_{\ep}(w,c)|+|f_{\ep}'(w,c)|)
\textbf{1}_{J_2/J_1}(w)}{\phi_1(z,c)}dwdz\Big|\\
&\lesssim \Big|\int_{y_c}^y \int_{y_c}^z (1+\f{(1-c^2)^{-1}|\ep|}{\tanh|z-y_c|}\big)\textbf{1}_{J_1^c}(z)\cdot\f{\cosh A|w-y_c|
\cdot\textbf{1}_{J_2/J_1}(w)}{e^{C^{-1}\al|z-y_c|}}dwdz\Big|\cdot\Big(\|f_{\ep}\|_{Y_A}
+\|f_{\ep}'\|_{Y_A}\Big)\\
&\lesssim(1+(1-c^2)^{-1}|\ep|) \cdot\Big|\int_{y_c}^ye^{-C^{-1}\al|z-y_c|} dz\Big|\cdot(1-c^2)^{-1}|\ep|\\
&\leq  \f12 C(1-c^2)^{-1}|\ep|+ C(1-c^2)^{-2}|\ep|^2.
\end{align*}
 Summing up the estimates of $I_1$, $I_2$, $I_3$, we  obtain \eqref{est:ferJ}.\end{proof}

The following corollaries will be useful in both constructing the solution to inhomogeneous Rayleigh equation and establishing the limiting absorption principle, in Sections 3,4 and 6.

\begin{Corollary}\label{Cor:basic-limit}
Let $\al\geq 1$, $c_{\ep}\in O_{\ep_0}$ and $z\leq y\leq y_c$ or $y_c\leq y\leq  z$. It holds that
\begin{align}
 \label{Oep-1}&\Big|\f{\phi(y,c_{\ep})}{\phi(z,c_{\ep})^2}\Big|
 \leq \f{|u(y)-c_{\ep}|}{|u(z)-c_{\ep}|^2\cdot \phi_1(z,c)}\leq \f{Ce^{-C^{-1}\al|y-y_c|}}{|u(y)-c_{\ep}|},\\
 \label{Oep-2}&\bigg|\f{\phi(y,c_{\ep})\int_{y_c}^z\phi_1(w,c_{\ep}
 )f(w)dw}{\phi(z,c_{\ep})^2}\bigg|\leq \f{C(1-c^2)^{-\f12}e^{-C^{-1}|z-y_c|}}{|u(z)-c_{\ep}|^{\f12}},\quad \text{if}\;f\in L^2_{1/(u')^2}(\mathbb{R}),\\
 \label{Oep-3}&\bigg|\f{\phi(y,c_{\ep})\int_{y_c}^z\phi_1(w,c_{\ep}
 )f(w)dw}{\phi(z,c_{\ep})^2}\bigg|\leq Ce^{-C^{-1}\al|z-y|},\quad \text{if}\;f\in H^1_{1/(u')^2}(\mathbb{R}),
 \end{align}
 where the constant $C$ is independent of $\ep$, $c$, $y$, $z$, $\al$. 
\end{Corollary}
\begin{proof}
We will only prove the estimates for $z\leq y\leq y_c$, since the proof for $y_c\leq y\leq z$ is similar.

 The estimate \eqref{Oep-1} follows by
 Corollary \ref{cor:phi_1-infty}, \eqref{phi_1(z)leq phi_1(y)} and \eqref{est:phi_1-ul}.

To prove \eqref{Oep-2} and \eqref{Oep-3}, we consider $c\in(-1,1)$ and deduce by $u'(w)\sim e^{-2|w|}$ that
\begin{align}
 u'(w)\leq \big(u'(w)\big)^{\f12}e^{-|w|}\leq C\big(u'(w)\big)^{\f12}e^{|y_c|} e^{-|w-y_c|}\leq C(1-c^2)^{-\f12}e^{-|w-y_c|}\big(u'(w)\big)^{\f12}\label{est:u'-c},
 \end{align}
 which together with \eqref{est:phi_1-ul} gives
 \begin{align}
 \phi_1^{\f{1}{\al}}(w,c)u'(w)\leq  Ce^{|y-y_c|}\cdot(1-c^2)^{-\f12}e^{-|w-y_c|}\big(u'(w)\big)^{\f12}
 \leq C(1-c^2)^{-\f12}\big(u'(w)\big)^{\f12}\label{est:phi_1-u'}.
 \end{align}
Therefore, thanks to  \eqref{est:phi_1-ul}, we  get by Cauchy-Schwarz inequality that
 \begin{align}
\notag\Big|\int_{y_c}^z\phi_1(w,c)^{\f{1}{\al}}f(w)dw\Big|
&=\Big|\int_{y_c}^z\big(\phi_1(w,c)^{\f{1}{\al}}u'(w)\big)\cdot (f/u')(w)dw\Big|
 \\
 &\leq C(1-c^2)^{-\f12}|u(z)-c|^{\f12}\|f/u'\|_{L^2}. \label{est:yc-z-C}
 \end{align}
 Thus, the estimate \eqref{Oep-2} follows by \eqref{est:up-lower-phi-1},
\eqref{phi_1(z)leq phi_1(y)}, \eqref{est:phi_1-ul} and \eqref{est:phi_1-u'}
 \begin{align}
 \notag\bigg|\f{\phi(y,c_{\ep})\cdot\big|\int_{y_c}^z\phi_1(w,c
 _{\ep})f(w)dw\big|}{\phi(z,c_{\ep})^2}\bigg|
 &\leq \f{
 \Big|\int_{y_c}^z\phi_1(w,c)^{\f{1}{\al}}f(w)dw\Big|\cdot \phi_1(y,c)}{|u(z)-c_{\ep}|\cdot\phi_1(z,c)^{1+\f{1}{\al}}}\\ \notag
 &\leq \f{C(1-c^2)^{-\f12}}{|u(z)-c_{\ep}|^{\f12}
 \cdot\phi_1(z,c)^{\f{1}{\al}}}\leq \f{C(1-c^2)^{-\f12}e^{-C^{-1}|z-y_c|}}{|u(y)-c_{\ep}|^{\f12}}.
 \end{align}
The estimate \eqref{Oep-3} follows by
\eqref{est:phi_1(y)/phi_1(z)}, \eqref{phi_1(z)leq phi_1(y)} and \eqref{est:phi_1-ul}
 \begin{align}
 \notag\bigg|\f{\phi(y,c_{\ep})\cdot\big|\int_{y_c}^z\phi_1(w,c
 _{\ep})f(w)dw\big|}{\phi(z,c_{\ep})^2}\bigg|
 &\leq \f{
 \Big|\int_{y_c}^zf(w)dw\Big|}{|u(z)-c_{\ep}|}
 \cdot \f{\phi_1(y,c)}{ \phi_1(z,c)}\\
& \leq C\|f/u'\|_{L^{\infty}}\cdot e^{-C^{-1}\al|z-y|}\notag\\
&\leq C\|f\|_{H^{1}_{1/(u')^2}}\cdot e^{-C^{-1}\al|z-y|}.\notag
 \end{align}
The above constants  $C$ are independent of $y$, $z$, $\ep$, $c$, $\al$.
  \end{proof}


\section{Inhomogeneous Rayleigh equation }
 The goal of this section is to solve the inhomogenous Rayleigh equation for $\textbf{c}\in O_{\ep_0}$:
\begin{align}\label{eq:phi-inh}
(u-\textbf{c})(\Phi''-\al^2\Phi)-u''\Phi=f
\end{align}
Let  $\phi$, $\phi_1$ be the solution of homogenous solution constructed  in Proposition \ref{prop:Rayleigh-Hom}.
\begin{definition}\label{def:W}
 We define the Wronskian of the Rayleigh equation as
\begin{align}
\label{def:W-1}W(\textbf{c},\al)&:=\int_{\mathbb{R}}\f{1}{\phi(y,\textbf{c})^2}dy,\quad
\textbf{c}
 \in O_{\ep_0}.
\end{align}

\end{definition}
Thanks to \eqref{est:phi_1-ul-c}, we have for $\textbf{c}\in O_{\ep_0}$,
\begin{align}
\Big|\f{1}{\phi(y,\textbf{c})^2}\Big|&\lesssim \f{1}{|\mathrm{Im}(\textbf{c})|^2\phi_1(y,\textbf{c})^2}\leq
C|\mathrm{Im}(\textbf{c})|^{-2}e^{-2C^{-1}\al|y-y_c|}
\in L^1_y(\mathbb{R}).\label{est:W1-f0.5}.
\end{align}
Consequently, \eqref{def:W-1} is well-defined.

\begin{lemma}\label{cor:Wneq0}
Let $\textbf{c}\in O_{\ep_0}$, $\al\geq 1$. Then  it holds that
\begin{align}\label{Wneq0}
W(\textbf{c},\al)\neq 0.\end{align}
\end{lemma}

\begin{proof}

In what follows, we prove it by contradiction. Assuming  $W(\textbf{c},\al)=0$,   we  shall construct a non-zero function with two expressions
\begin{align*}
\Gamma(\al,y,\textbf{c}):=\phi(y,\textbf{c})\int_{\pm\infty}^y\f{1}{\phi(z,\textbf{c})^2}dz,
\end{align*}
which satisfies $\Gamma''-\al^2\Gamma-\f{u''}{u-\textbf{c}}\Gamma=0$.
Let $y_c=u^{-1}(\mathrm{Re}(\textbf{c}))$. For $y\leq y_c$, thanks to \eqref{Oep-1},  we obtain
\begin{align}\label{est:tildephi}
|\Gamma(y,\textbf{c})|=\Big|\int_{-\infty}^y\f{\phi(y,\textbf{c})}{\phi(z,\textbf{c})^2}dz\Big|\lesssim \f{e^{-C^{-1}\al|y-y_c|}}{|u(y)-\textbf{c}|}\lesssim|\textrm{Im}(\textbf{c})|^{-1}e^{-C^{-1}\al|y-y_c|}\in L^2_y(-\infty,y_c).
 \end{align}
Similarly, for  $y_c\leq y<+\infty$, we obtain $\Gamma(y,\textbf{c})\in L^2_y(y_c,+\infty)$. Therefore, $\Gamma(y,\textbf{c})\in L^2_y(\mathbb{R})$.
 Moreover, we can deduce from the equation that $\Gamma\in H^2_y(\mathbb{R})$. Thus, $\Gamma(\al,y,\textbf{c})$ is an eigenfunction of $\mathcal{R}_{\al}$, cf. \eqref{Op:Ray}, which belongs to  $H^2$, with eigenvalue $\textbf{c}\in O_{\ep_0}\subset \mathbb{C}/\mathbb{R}$.
This contradicts with the linear stability of  $\mathcal{R}_{\al}$($u(y)=\tanh y$), see  \cite[Theorem 1.5 (i)]{Lin03}.
\end{proof}

\begin{proposition}\label{pro:solve-Phi}
Let $f\in L^2_{1/(u')^2}(\mathbb{R})$, 
 $\al\geq 1$, $\textbf{c}\in O_{\ep_0}$. The inhomogeneous equation \eqref{eq:phi-inh} admits a unique solution in  $H^2(\mathbb{R})$, which can be expressed as follows
\begin{align}
\Phi(\al,y,\textbf{c})\label{solution-Phi-1}&:=\phi(y,\textbf{c})\int_{-\infty}^y\f{\int_{y_c}^z\phi_1(w,\textbf{c})f(w)dw}{\phi(z,\textbf{c})^2}dz
+\mu(\textbf{c},\al)\phi(y,\textbf{c})\int_{-\infty}^y\f{1}{\phi(z,\textbf{c})^2}dz\\
\label{solution-Phi-2}&=\phi(y,\textbf{c})\int_{+\infty}^y\f{\int_{y_c}^z\phi_1(w,\textbf{c})f(w)dw}{\phi(z,\textbf{c})^2}dz
+\mu(\textbf{c},\al)\phi(y,\textbf{c})\int_{+\infty}^y\f{1}{\phi(z,\textbf{c})^2}dz,
\end{align}
where
\begin{align}\label{def:W-mu}
\mu(\textbf{c},\al)=-\f{T(f)(\textbf{c})}{W(\textbf{c},\al)},
\end{align}
 with $W(\textbf{c},\al)$ defined in \eqref{def:W-1} and
\begin{align}\label{def:T(h)} T(f)(\textbf{c})&:=\int_{\mathbb{R}}\f{\int_{y_c}^z\phi_1(w,\textbf{c})f(w)dw}{(u(z)-\textbf{c})^2\phi_1(z,\textbf{c})^2}dz,\quad \textbf{c}\in O_{\ep_0}.
\end{align}
\end{proposition}
\begin{proof}
We begin by  demonstrating the well-definedness of  \eqref{solution-Phi-1}  for $\textbf{c}\in O_{\ep_0}$. We examine each term individually. Recall $y_c=u^{-1}(\mathrm{Re}(\textbf{c}))$.

If $f\in L^2_{1/(u')^2}$, then we deduce from \eqref{est:up-lower-phi-1}, Cauchy-Schwarz inequality, \eqref{phi_1(z)leq phi_1(y)} and \eqref{est:phi_1-ul} that
\begin{align}
\quad\notag\Big|\f{\int_{y_c}^z
\phi_1(w,\textbf{c})f(w)dw}
{\phi(z,\textbf{c})^2}\Big|&\lesssim \f{|u(z)-\mathrm{Re}(\textbf{c})|^{\f12}\|f\|_{L^2_{1/(u')^2}}}{|u(z)-\textbf{c}|^{2}\phi_1(z,\mathrm{Re}(\textbf{c}))}\\
\label{est:T-int}&\lesssim |\mathrm{Im}(\textbf{c})|^{-\f32}e^{-C^{-1}\al|z-y_c|}\in L^1_{z}(\mathbb{R}).
\end{align}
Consequently, \eqref{def:T(h)} is well-defined. Indeed, considering $W(\textbf{c})$ is well-defined, we can conclude that $\mu(\textbf{c},\al)$, cf. \eqref{def:W-mu} is well-defined. For  $y\leq y_c$,  we have by  \eqref{Oep-1} that
 \begin{align}\label{inhoPhi-1}
 \Big|\int_{-\infty}^y\f{\phi(y,\textbf{c})}{\phi(z,\textbf{c})^2}\Big|
 &\lesssim |\mathrm{Im}(\textbf{c})|^{-1}e^{-C^{-1}\al(y_c-y)}\in L^2_y(-\infty,y_c).
 \end{align}
 For  $y\leq y_c$, $f\in L^2_{1/(u')^2}$, we have by  \eqref{Oep-2} that
 \begin{align}\label{inhoPhi-2}
 &\quad\notag\Big|\phi(y,\textbf{c})\int_{-\infty}^y\f{\int_{y_c}^z\phi_1(w,\textbf{c})f(w)dw}{\phi(z,\textbf{c})^2}dz\Big|\\
 &\lesssim|\mathrm{Im}(\textbf{c})|^{-\f12}(1-\mathrm{Re}(\textbf{c})^2)^{-\f12}e^{-C^{-1}(y_c-y)}\in L^2_y(-\infty,y_c).
 \end{align}
Therefore, each term of $\Phi(\al,\cdot,\textbf{c})$ in \eqref{solution-Phi-1} is well-defined.

Furthermore,  we obtain by \eqref{inhoPhi-1} and \eqref{inhoPhi-2} that $\Phi(\al,\cdot,\textbf{c})$ belongs to $L^2(-\infty,y_c)$. 
Similarly, the expression in \eqref{solution-Phi-2} is well-defined and belongs to $L^2(y_c,+\infty)$.  The definition of $\mu$ in \eqref{def:W-mu} yields  the equality in \eqref{solution-Phi-2}.
Consequently,  we have $\Phi(\al,\cdot,\textbf{c})\in L^2(\mathbb{R})$. It can be directly verified that \eqref{solution-Phi-1}  is the solution to the equation \eqref{eq:phi-inh}. Using  equation \eqref{eq:phi-inh}, we can obtain $\Phi(\al,\cdot,\textbf{c})\in H^2(\mathbb{R})$ for  $\al\geq 1$ and $\textbf{c}\in O_{\ep_0}$.
\end{proof}

\section{Embedding eigenvalue}
The goal of this section is to provide the classical results regarding the embedding eigenvalue of Rayleigh operator $\mathcal{R}_{\al}$, cf. \eqref{Op:Ray}, and to express them in our terminology.

Consider the Rayleigh equation
\begin{align}\label{eq:eigen}
(u-c)(\varphi''-\al^2\varphi)-u''\varphi=0.
\end{align}
\begin{definition}\label{def:eigenfucntion}
Let $\al\in \mathbb{Z}^+$.
\begin{itemize}
\item For $c\in \mathbb{C}/ \mathrm{Ran} u$, if  \eqref{eq:eigen} admits a $H^2(\mathbb{R})$ solution, we refer to  $\textbf{c}$ as eigenvalue of  $\mathcal{R}_{\al}$
    with eigenfunction $\varphi(\cdot,\textbf{c})$.
\item For $c\in \mathrm{Ran} u$, if  \eqref{eq:eigen} admits a $H^2(\mathbb{R})$ solution,  we refer to  $c$ as embedding eigenvalue of  $\mathcal{R}_{\al}$  with embedding eigenfunction $\varphi(\cdot,c)$.
\end{itemize}
\end{definition}

Let us remark that the definition of embedding eigenvalues is equivalent to one introduced in  \cite{WZZ-apde}: {\it $c\in \mathrm{Ran}\, u$ is an embedding eigenvalue of $\mathcal{R}_\al$ if there exists $0\neq \varphi\in H^1(\mathbb{R})$ such that  for all $\chi\in H^1(\mathbb{R})$,
\begin{align}\nonumber
\int_{\mathbb{R}}(\varphi'\chi'+\al^2\varphi\chi)dy+p.v.\int_{\mathbb{R}}\frac{u''\varphi\chi}{u-c}dy+i\pi\sum_{y\in u^{-1}\{c\},u'(y)\neq 0}\frac{(u''\varphi\chi)(y)}{|u'(y)|}=0.
\end{align}
}

\begin{lemma}\label{lem:embedd-u''-psi}
 Let $c\in\mathbb{R}$ and $\varphi\in H^2(\mathbb{R})$ be the non-zero solution to equation  \eqref{eq:eigen}. Then it holds that  $\varphi(y_c,c)\neq 0$ and $u''(y_c)=0$.
\end{lemma}
\begin{proof}
Let $y=y_c$ in \eqref{eq:eigen}. Using  equation \eqref{eq:eigen}, we have $u''(y_c)\cdot \varphi(y_c,c)=0$.
If $\varphi(y_c)=0$, then $\f{\varphi(y,c)}{u(y)-c}$ is well-defined for $\varphi(\cdot,c)\in H^2$. By the classical Rayleigh's trick, we have
 \begin{align}
 \int_{\mathbb{R}}\big|\varphi'-\f{u'\varphi}{u-c}\big|^2+\al^2|\varphi|^2dy=0,
\end{align}
which gives $\varphi(\cdot,c)\equiv 0$ for $\al\neq 0$. This leads to a contradiction. Thus, we must have  $\varphi(y_c)\neq 0$ and $u''(y_c)=0$.
\end{proof}

\begin{lemma}\label{lem:eigen-al=1'}
Let $\lambda\in\mathbb{R}$. There exists non-zero $\varphi\in L^2(\mathbb{R})$ satisfying $\varphi''+2(\cosh^{-2}y)\varphi=\lambda\varphi$ if and only if $\lambda=1$. Moreover, $\varphi(y)=C\mathrm{sech}\,y $.
\end{lemma}
\begin{proof}
See \cite[P.92]{R}.
\end{proof}

For hyperbolic tangent flow, let us summarize the classical results using the terminology in this paper as follows.

\begin{lemma}\label{lem:eigen-al=1}
Let $u(y)=\tanh y$, $\al\geq 1$, $c\in\mathbb{C}$. It holds that

(1)  $\mathcal{R}_{\al}$: $H^2(\mathbb{R})\to H^2(\mathbb{R})$ has no eigenvalues.
Only for $\al=1$, $\mathcal{R}_{\al}$ has a single embedding eigenvalue $c=0$ with 1-D  eigenspace
$E_1=\mathrm{span}\{\mathrm{sech}\,y \}$.

In other words,
the Rayleigh equation \eqref{eq:eigen} has a non-zero $H^2(\mathbb{R})$ solution $\varphi(y,c)$ if and only if $\al=1$, $c=0$. By adjusting for a constant, we have $\varphi(y,0)=\mathrm{sech}\,y $.

(2) The solution constructed in Proposition \ref{prop:Rayleigh-Hom} has the following explicit expression
\begin{align*}
\phi(y,0)=\f{1}{2}\big(\sinh y+\f{y}{\cosh y}\big).
\end{align*}
Moreover, it holds that
\beq\label{cosh-1y}
-\mathrm{sech}\,y =
\left\{
\begin{aligned}
&\phi(y,0)\int_{-\infty}^y\f{1}{\phi(z,0)^2}dz,\quad y<0,\\
&\phi(y,0)\int_{+\infty}^y\f{1}{\phi(z,0)^2}dz,\quad y>0.
\end{aligned}
\right.
\eeq

\end{lemma}
\begin{proof}
 First we prove (1). Indeed, for $\al\geq 1$, $c\in\mathbb{C}/\mathbb{R}$, $\mathcal{R}_{\al}$ has no eigenvalues, see  \cite[Theorem 1.5 (i)]{Lin03}. For $\al\geq 1$, $c\in\mathbb{R}$, $\mathcal{R}_{\al}$ has a single embedding eigenvalue $c=0$ with eigenfunction $\cosh ^{-1} y$. The proof follows from
Lemma \ref{lem:embedd-u''-psi} and Lemma \ref{lem:eigen-al=1'}, noting that $u''(0)=u(0)=0$ and $\f{u''}{u}=-2\cosh^{-2} y$.

Next we prove (2). For $c=0$, we have two linearly independent solutions of \eqref{eq:eigen}: $\varphi(y,0)=\mathrm{sech}\,y $ and
$\tilde{\varphi}(y,0)=\varphi(y,0)\int_0^y\f{1}{\varphi(z,0)^2}dz=
\f{1}{2}\big(\sinh y+\f{y}{\cosh y}\big)$. Since $\phi(y,0)$ is also a solution of \eqref{eq:eigen} with $c=0$, we have $\phi=C_1\varphi+C_2\tilde{\varphi}$. By Proposition \ref{prop:Rayleigh-Hom}, we obtain  $\phi(0,0)=0$ and $\phi'(0,0)=1$. Consequently, $C_1=0$, $C_2=1$ and
\begin{align}\nonumber
\phi(y,0)=\varphi(y,0)\int_0^y\f{1}{\varphi(z,0)^2}dz
=\f{1}{2}\big(\sinh y+\f{y}{\cosh y}\big).
\end{align}
The first equality in the above line gives $(\f{\phi(y,0)} {\varphi(y,0)})'=\f{1}{\varphi(y,0)^2}$, which  implies  $(\f{\varphi(y,0)}{\phi(y,0)})'=\f{-1}{\phi(y,0)^2}$. 
We integrate it to get \eqref{cosh-1y}.
\end{proof}

\section{analysis of Wronskian}
In this section, our aim is  to analyse the quantitative behavior of the Wroskian $W(\textbf{c},\al)$ and the related functions. This analysis is crucial in understanding the limiting absorption principle near the  embedding eigenvalue.

\subsection{Behavior near $\textbf{c}=0$}
The following lemma is the key to establish the limiting absorption principle in the presence of  embedding value $c=0$. 
\begin{lemma}\label{lem:pa_cW_1}
Let $\al=1$,  
 $\textbf{c}\in O_{\ep_0}$. It holds that  \begin{align*}\lim_{\textbf{c}\to 0}\f{W(\textbf{c},1)}{\textbf{c}}=\mathrm{sgn}(\mathrm{Im}(\textbf{c}))2\pi i.
\end{align*}
\end{lemma}
  In this subsection, we use the notation $c_{\ep}=c+i\ep$ instead of $\textbf{c}$ to represent complex variable. We use the notation $c$ for real variable. The symbol $'$ denotes the derivative with respect to $c_{\ep}$, while $\pa_c$ represents for the derivative with respect to $c$. Let $\phi(y,0),\phi_1(y,0)$ be as constructed in Proposition \ref{prop:Rayleigh-Hom}.

 To prove Lemma \ref{lem:pa_cW_1}, we introduce the regular component of the Wronskian.

\begin{definition}\label{def:W1}
We define
\begin{align}
\notag&\label{def:W-3}W_1(c_{\ep},\al)\\
&:=
\int_{\mathbb{R}}\f{(u(y)+c_{\ep})^2}{\phi_1(y,c_{\ep})^2}
+\f{(1+u(y)^2-2c_{\ep}^2)u'(y)}{(u(y)-c_{\ep})^2}\left(\f{1}{\phi_1(y,c_{\ep})^2}-1\right)dy,\quad c_{\ep}\in \widetilde{O}_{\ep_0}.
\end{align}
\end{definition}

For $c_{\ep}\in \widetilde{O}_{\ep_0}$, we have by \eqref{est:phi_1-ul-c} and  \eqref{est:up-phi-1} that
\begin{align}
\Big|\f{(u(y)+c_{\ep})^2}{\phi_1(y,c_{\ep})^2}\Big|&\lesssim
\big|\f{1}{\phi_1(y,c)^2}\big|
\leq Ce^{-2C^{-1}\al|y-y_c|}\in L^1_y(\mathbb{R}),\label{est:W1-f1}\\
 \notag\Bigg|\f{(1+u(y)^2-2c_{\ep}^2)u'(y)}{(u(y)-c_{\ep})^2}\left(\f{1}{\phi_1(y,c_{\ep})^2}-1\right)\Bigg|
 &\lesssim \f{u'(y)\min\{\al^2|y-y_c|^2,1\}}{|u(y)-c|^2}\\
 &\leq C(1-c^2)^{-2}u'(y)\in L^1_y(\mathbb{R}).\label{est:W1-f2}
\end{align}
Therefore, $W_1(\cdot,\al)$ is well-defined in  $\tilde{O}_{\ep_0}$.

\begin{lemma}\label{lem:W1} $W(\cdot,\al)$  is  analytical in $O_{\ep_0}$, and $W_1(\cdot,\al)$  is  analytic in $\widetilde{O}_{\ep_0}$.
In addition, we have an alternative expression for the Wronskian as follows
\begin{align}
\label{def:W-2'}
W(c_{\ep},\al)=(1-c_{\ep}^2)^{-2}\Big(W_1(c_{\ep},\al)+2c_{\ep}\ln\f{c_{\ep}-1}{c_{\ep}+1}\Big).
\end{align}

\end{lemma}
\begin{remark}
The $ln z$ represents the principal branch of logarithm, for $z\in\mathbb{C}/\{(-\infty,0]\}$. Thus $\ln\f{c_{\ep}-1}{c_{\ep}+1}$ is defined  for $\f{c_{\ep}-1}{c_{\ep}+1}\in \mathbb{C}/\{(-\infty,0]\}$, i.e., $c_{\ep}\in \mathbb{C}/[-1,1]$. Moreover, for $c_{\ep}\in\mathbb{C}/[-1,1]$,  $\ln\f{c_{\ep}-1}{c_{\ep}+1}$  is analytic and
\begin{align}\label{ln-limit-1}
\ln\f{c_{\ep}-1}{c_{\ep}+1}&=\ln\f{|1-c_{\ep}|}{|1+c_{\ep}|}
+i\mathrm{sgn}(\ep)\Big(\pi-\arctan \f{2|\ep|}{1-c^2-\ep^2}\Big),
\end{align}
and
\begin{align}\label{ln:limt}
\lim_{c_{\ep\to \pm 1 }}\ln\f{c_{\ep}-1}{c_{\ep}+1}=\infty.
\end{align}
\end{remark}

\begin{proof}[Proof of Lemma \ref{lem:W1}]
Using Proposition \ref{prop:Rayleigh-Hom} (2), we can conclude  that $\f{1}{\phi(\cdot,c_{\ep})^2}$ is analytic in $ O_{\ep_0}$, and
the integrated function in \eqref{def:W-3} is analytic in $ \widetilde{O}_{\ep_0}$.
Taking  $\Omega=B_{\delta}(c_{\ep})\subset O_{\ep_0}$ in Lemma \ref{lem:c-paramete-derivative}, we can use \eqref{est:W1-f0.5} to deduce that $W(c_{\ep})$ is  analytic on $O_{\ep_0}$.
Similarly, thanks to Lemma \ref{lem:c-paramete-derivative} and \eqref{est:W1-f1}-\eqref{est:W1-f2}, we know that $W_1(c_{\ep})$ is analytic  in $\widetilde{O}_{\ep_0}$.

Due to  $u'(y)=1-u(y)^2$, direct calculation   gives
\begin{align}\label{decompose-(1-c^2)^2/(u-c)^2}
\f{(1-c_{\ep}^2)^2}{(u-c_{\ep})^2}=\f{(u^2-c_{\ep}^2)^2
+(u-c_{\ep})^2u'+(1-c_{\ep}^2)u'+2c_{\ep}u'(u-c_{\ep})}{(u-c_{\ep})^2}.
\end{align}
Thanks to \eqref{decompose-(1-c^2)^2/(u-c)^2} and $\phi=(u-c_{\ep})\phi_1$, the equality \eqref{def:W-2'} can be established by a  direct calculation with
Definition \ref{def:W} and Definition \ref{def:W1}.
\end{proof}

\begin{lemma}\label{lem:c-paramete-derivative}
let $\Omega\subset \mathbb{C}$ be an open bounded connected domain. Assume that $f(y,c_{\ep})$ satisfies

(1) for any $y\in\mathbb{R}$, $f(y,c_{\ep})$ is analytical in $c_{\ep}\in \Omega$;

(2) there exists  $F(y)\in L^1(\mathbb{R})$ such that for any $z\in \Omega$, $|f(y,c_{\ep})|\leq CF(y)$, where $C$ independent of $y$.

Then  $\int_{\mathbb{R}}f(y,c_{\ep})dy$ is analytical for $c_{\ep}\in \Omega$.
\end{lemma}
\begin{proof}
We will prove for $c_{\ep}\in \Omega$,
$
\f{d}{dc_{\ep}}\int_{\mathbb{R}}f(y,c_{\ep})dy
=
\int_{\mathbb{R}}\f{d}{dc_{\ep}}f(y,c_{\ep})
dy.
$
 Taking a contour $\Gamma\subset \Omega$ that contains $c_{\ep}$ with $dist(c_{\ep},\Gamma)>0$.  For   $\Delta c_{\ep}$ small enough,  we have by Cauchy formula and condition (2) that
\begin{align*}
&\quad\Big|\f{f(y,c_{\ep}+\Delta c_{\ep})-f(y,c_{\ep})}{\Delta c_{\ep}}\Big|=\Big|\f{1}{2\pi i}\int_{\Gamma}\Big(\f{f(y,z)}{z-c_{\ep}-\Delta c_{\ep}}-\f{f(y,z)}{z-c_{\ep}}\Big)\f{1}{\Delta c_{\ep}}dz\Big|\\
&=\Big|\f{1}{2\pi i}\int_{\Gamma}\f{f(y,z)}{(z-c_{\ep}-\Delta c_{\ep})(z-c_{\ep})}dz\Big|\leq CF(y)\cdot 2dist(c_{\ep},\Gamma)^{-2}\cdot|\Gamma|\leq CF(y),
\end{align*}
where $C$ is independent of $y$ and $\Delta c_{\ep}$. Let $\Delta c_{\ep}\to 0$, we have the equality 
by  the dominated convergence theorem and condition (1).
\end{proof}
The following lemmas focus on   $W_1(\cdot,\al)$ at $c=0$.
\begin{lemma}\label{lem:embedd-critA}
 The following statements are equivalent:

\begin{itemize}

\item[(1)] $W_1(0,\al)=0.$

\item[(2)] $\al=1.$

\item[(3)] \eqref{eq:eigen} has  non-zero $H^2(\mathbb{R})$ solution.

 \end{itemize}

 In particular, it holds that
 \begin{align}
 W_1(0,1)=0\quad\text{and} \quad W_1(0,\al)\neq 0,\quad \al\geq 2.\label{A=0} \end{align}

\end{lemma}
\begin{proof}
 Lemma \ref{lem:eigen-al=1} establishes the equivalence between (2) and (3). It suffices to show $(1) \Rightarrow (3)$ and $(2) \Rightarrow (1)$. Let's consider the following facts (B1)-(B5):

From \eqref{est:phi_1-ul}, we have (B1): $\phi(y,0) \notin L^{2}(\mathbb{R}^{\pm})$.

Notice by $u'=1-u^2$ that
\begin{align}\label{equality:W1}
\f{1}{\phi(z,0)^2}
=\f{u(z)^4}{\phi(z,0)^2}+\f{(1+u(z)^2)u'(z)}{u(z)^2}\Big(\f{1}{\phi_1(z,0)^2}-1\Big)
+\Big(u(z)-\f{1}{u(z)}\Big)'.
\end{align}
Consequently, we define
\begin{align}
\notag\varphi_{\pm}(y)&=\phi(y,0)\int_{\pm\infty}^y\f{u(z)^4}{\phi(z,0)^2}+
\f{(1+u(z)^2)u'(z)}{u(z)^2}\Big(\f{1}{\phi_1(z,0)^2}-1\Big)dz\\
&\quad+\phi(y,0)\Big(u(y)-\f{1}{u(y)}\Big).\label{phipm}
\end{align}
Thanks to \eqref{est:W1-f1}-\eqref{est:W1-f2}($c_{\ep}=0$), $\varphi_{\pm}$ is well-defined on $\mathbb{R}$.

We can verify by \eqref{phi_1(z)leq phi_1(y)}, \eqref{est:phi_1-ul},  \eqref{equality:W1}  and direct calculation that
(B2): $\varphi_{\pm}\in L^2(\mathbb{R}^{\pm})$ satisfies \begin{align}\label{eq:eigen-phi-c=0}
\varphi''-\al^2\varphi-\f{u''}{u}\varphi=0.\end{align}

Using \eqref{equality:W1}, we have (B3): $\varphi(y)=\phi(y,0)\int_{-\infty}^y\f{1}{\phi(z,0)^2}dy$,  $y<0$.

Using \eqref{phipm}, we have $\varphi_-(y)-\varphi_+(y)=\phi(y,0)W_1(0)$, which gives (B4): $\varphi_{+}\equiv\varphi_{-}\Longleftrightarrow W_1(0,\al)=0$.

Lastly, we claim (B5): If $W_1(0,1)\neq 0$, then $\varphi_-\notin L^2(\mathbb{R}^+)$.

Indeed, assume $W_1(0,1)>0$. We have by Definition \ref{def:W1} that for $y$ large enough,
\begin{align*}
\int_{-\infty}^y\f{u(z)^4}{\phi(z,0)^2}+\f{(1+u(z)^2)u'(z)}{u(z)^2}\Big(\f{1}{\phi_1(z,0)^2}-1\Big)dz\geq \f{W_1(0,1)}{2}>0.
\end{align*}
On the other hand, thanks to $u(y)<1$ and $u(y)\to 1$($y\to+\infty$), we have for $y$ large enough, $0>u(y)-\f{1}{u(y)}>-\f{W_1(0,1)}{4}$. Together with $\phi(y,0)>0$($y>0$), we have by \eqref{phipm} that

\begin{align*}
\varphi_-(y)\geq \f{W_1(0,1)}{4}\phi(y,0)\geq \f{W_1(0,1)}{4}\cdot C^{-1}e^{C^{-1}\al|y|}\geq  C^{-1}e^{C^{-1}\al|y|}\notin L^2(\mathbb{R}^+),
\end{align*}
which proves (B5).

Now we are in a position to prove the following.\smallskip

$(1) \Longrightarrow (3)$ . Assuming $W_1(0,\al)=0$, using (B2), (B4), we can construct $\varphi:=\varphi_{\pm}$ as a solution of \eqref{eq:eigen-phi-c=0} in $L^2(\mathbb{R})$. Using  \eqref{eq:eigen-phi-c=0} and $\f{u''}{u'}\in L^{\infty}(\mathbb{R})$, we have $\varphi\in H^2(\mathbb{R})$.\smallskip

$(2) \Longrightarrow (1)$ . We prove by contradiction. Assume $W_1(0,1)\neq 0$. Due to $(2) \Longrightarrow (3)$,  we have  a  non-zero $L^2(\mathbb{R})$ solution $\varphi(y)$ of \eqref{eq:eigen-phi-c=0}. Notice by (B2), (B3) that $\varphi_-(y)$ and $\phi(y,0)$  are two linearly independent solutions of \eqref{eq:eigen-phi-c=0}.  Consequently, $\varphi(y)=C_1\phi(y,0)+C_2\varphi_-(y)$. Noticing $\varphi\in L^2(\mathbb{R}^-)$, together with (B1) and (B2),  we have $C_1=0$. This implies $\varphi(y)=C_2\varphi_-(y)$. However, considering (B5), we conclude $C_2=0$. Thus, we obtain $\varphi(y)\equiv 0$ which leads to a contradiction.
\end{proof}
\begin{lemma}\label{lem:pa_cW1}
Let $\al\geq 1$, $c\in(-1,1)$. It holds that
\begin{align*}
\pa_cW_1(0,\al)=0.
\end{align*}
\end{lemma}

\begin{proof}
Notice  that
\begin{align*}
1+u^2-2c^2=(u^2-c^2)+(1-c^2),\;\; \text{and}\;\;u^2-c^2=(1-c^2)-u'.
\end{align*}
Therefore, using the definition  \eqref{def:W-3},  we have
\begin{align}
\notag W_1(c)&=\int_{\mathbb{R}}\f{(u^2-c^2)^2}{\phi^2}dy+
p.v.\int_{\mathbb{R}}\f{(u^2-c^2)u'}{(u-c)^2\phi_1^2}
-p.v.\int_{\mathbb{R}}\f{(u^2-c^2)u'}{(u-c)^2}
dy\\
\notag &\quad+
\int_{\mathbb{R}}\f{(1-c^2)u'}{(u-c)^2}
\left(\f{1}{\phi_1^2}-1\right)dy\\
&=p.v.\int_{\mathbb{R}}\f{(1-c^2)(u^2-c^2)}{\phi^2}dy
+
\int_{\mathbb{R}}\f{(1-c^2)u'}{(u-c)^2}
\left(\f{1}{\phi_1^2}-1\right)dy
-2-2c\ln\f{1-c}{1+c}\notag\\
&=p.v.\int_{\mathbb{R}}\f{(1-c^2)^2-(1-c^2)u'}{\phi^2}dy
+
\int_{\mathbb{R}}\f{(1-c^2)u'}{(u-c)^2}
\left(\f{1}{\phi_1^2}-1\right)dy
-2-2c\ln\f{1-c}{1+c}.\label{fm:W1}
\end{align}
Recalling  Lemma \ref{lem:simple-useful-equ}, we have for good derivative $\pa_G=\f{\pa_y}{1-c^2}+\pa_c$ that
\begin{align*}
&\pa_G\left(\f{(1-c^2)u'(z)}{(u-c)^2}\right)=0,\quad \pa_G\left(\f{1-c^2}{(u-c)}\right)=1\quad \pa_G\left(\f{(1-c^2)^2}{(u-c)^2}\right)=\f{2(1-c^2)}{u-c}.
\end{align*}
Notice that if $\lim_{y\to\infty}f(c,y)=0$, then
\begin{align}
\notag\pa_c\Big(p.v.\int_{\mathbb{R}}f(c,y)dy\Big)
&=p.v.\int_{\mathbb{R}}(\pa_c+\f{\pa_y}{1-c^2})f(c,y)dy-(1-c^2)^{-1}f(c,y)|_{y=-\infty}^{+\infty}\\
&=p.v.\int_{\mathbb{R}}\pa_Gf(c,y)dy.\label{fm:paGintf}
\end{align}
Consequently,  we have
\begin{align}
\notag\pa_cW_1&=p.v.\int_{\mathbb{R}}\pa_G\Big(\f{(1-c^2)^2-(1-c^2)u'}{(u-c)^2}\cdot\f{1}{\phi_1^2}\Big)dy
+\int_{\mathbb{R}}\pa_G\Big(\f{(1-c^2)u'}{(u-c)^2}\big(\f{1}{\phi_1^2}-1\big)\Big)dy\\
&\quad\notag-\f{d}{dc}\Big(2c\ln\f{1-c}{1+c}\Big)\\
&\label{fm:pa_cA}=p.v.\int_\mathbb{R}\f{2(1-c^2)}{(u-c)\phi_1^2}dy
+\int_\mathbb{R}\f{(1-c^2)^2}{(u-c)^2}\pa_G(\f{1}{\phi_1^2})dy
-\f{d}{dc}\Big(2c\ln\f{1-c}{1+c}\Big).
\end{align}
We have by Lemma \ref{prop:goodphi1} that  $\phi_1(y,0)$ is an even function and $\pa_G\phi_1(y,0)$ is an odd function. Together with the oddness of $u(y)$, we obtain
\begin{align}
\notag\pa_cW_1(0)&
=p.v.\int_\mathbb{R}\f{2dy}{u(y)\phi_1(y,0)^2}
+\int_\mathbb{R}\f{-2\pa_G\phi_1(y,0)dy}{u(y)^2\phi_1(y,0)^3}
-\Big(2\ln\f{1-c}{1+c}-\f{4c}{1-c^2}\Big)\Big|_{c=0}=0.
\end{align}
\end{proof}

\begin{proof}[Proof of Lemma  \ref{lem:pa_cW_1}.]
Using Lemma \ref{lem:W1}, \eqref{ln-limit-1} and Lemma \ref{lem:embedd-critA}, we have
 \begin{align*}
 \lim_{c_{\ep}\to 0}\f{W(c_{\ep},1)}{c_{\ep}}
&= \lim_{c_{\ep}\to 0}\f{(1-c_{\ep}^2)^2W(c_{\ep},1)}{c_{\ep}}= \lim_{c_{\ep}\to 0}\f{2c_{\ep}\ln\f{c_{\ep}-1}{c_{\ep}+1}}{c_{\ep}}
+\lim_{c_{\ep}\to 0}\f{W_1(c_{\ep},1)-W_1(0,1)}{c_{\ep}-0}\\
&=2\lim_{c_{\ep}\to 0}\f{\ln\f{c_{\ep}-1}{c_{\ep}+1}}{c_{\ep}}
+W'_1(0,1)
=\mathrm{sgn}(\ep)2\pi i+\pa_cW_1(0,1)=\mathrm{sgn}(\ep)2\pi i.
 \end{align*}
\end{proof}

\subsection{Estimates of  $A(c,\al)$}
\begin{definition}\label{def:A}
For $\al\geq 1$, $c\in (-1,1)$, we define
\begin{align}\label{def:W-4}
A(c,\al):=W_1(c,\al)+2c\ln\f{1-c}{1+c}.
\end{align}
\end{definition}
\begin{remark}
 We note by Lemma \ref{lem:W1} that $W(c_{\ep})$ is analytic in $O_{\ep_0}\subset \mathbb{C}/\mathbb{R}$.
The definition of $A$ is related to the continuous extension of $W$. Indeed, due to \eqref{def:W-2'}, the analyticity of $W_1(c_{\ep})$ and \eqref{ln-limit-1}, we can continuously  extent the Wronskian separately from  the upper/lower half complex plane to the real axis. That is to say, for $c\in(-1,1)$,
\begin{align}\label{lim:W}
W^{\pm}(c):=\lim_{\ep\to 0\pm}W(c_{\ep})=(1-c^2)^{-2}\Big(A(c)\pm2\pi ci\Big).
\end{align}
Consequently,
\begin{align*}
(1-c^2)^4|W^{\pm}(c)|^2=A(c)^2+4\pi^2 c^2.
\end{align*}
\end{remark}

\begin{lemma}
\label{lem:A}

Let $\al\geq 1$, $c\in(-1,1)$. It holds that
\begin{itemize}
 \item $A(\cdot,\al)$ is continuous on $(-1,1)$.

 \item
 For $\al\geq 2$, there exists constant $C$ independent of $\al,c$ such that
\begin{align}
\label{est:A}C^{-1}\al^2\leq A(c,\al)^2+4\pi^2c^2\leq C\al^2.
\end{align}

\item For $\al=1$, $c\in(-1,1)/(-\delta,\delta)$ with $\delta\in(0,\f12)$, there exists constant $C$ independent of $c$ such that
\begin{align}
\label{est:A-1}C^{-1}\leq A(c,1)^2+4\pi^2c^2\leq C.
\end{align}
\item
 \begin{align}\label{pacA=0}
 A(0,1)=\pa_cA(0,1)=0.
 \end{align}
\end{itemize}
\end{lemma}

\begin{proof}
The continuity of $A(\cdot,\al)$  directly follows from Lemma \ref{lem:W1}.
Notice that
\begin{align*}
1+u^2-2c^2=(1-c^2)+(u-c)^2+2c(u-c).
\end{align*}
Thanks to Definition \ref{def:W1} and Definition \ref{def:A}, for $c\in(-1,1)$,  we compute
\begin{align}
\notag A(c)&=\int_{\mathbb{R}}
\f{(u(y)+c)^2}{\phi_1(y,c)^2}dy
+\int_{\mathbb{R}}
\f{(1+u(y)^2-2c^2)u'(y)}{(u(y)-c)^2}
\big(\f{1}{\phi_1(y,c)^2}-1\big)dy+2c\ln\f{1-c}{1+c}\\
\notag &=\int_{\mathbb{R}}
\f{(u^2-c^2)^2}{\phi^2}dy+\bigg(p.v.\int_{\mathbb{R}}
\f{2c(u-c)u'}{\phi^2}dy-p.v.\int_{\mathbb{R}}\f{2c(u-c)u'}{(u-c)^2}dy
\\&\notag \quad+\int_{\mathbb{R}}\f{\big(1-c^2+(u-c)^2\big)u'(y)}{(u(y)-c)^2}
\left(\f{1}{\phi_1^2(y,c)}-1\right)dy\bigg)+2c\ln\f{1-c}{1+c}\\
\notag &=p.v.\int_{\mathbb{R}}\f{(u^2-c^2)^2+2cu'(u-c)}{(u(y)-c)^2\phi_1^2(y,c)}dy
+\int_{\mathbb{R}}\f{\big(1-c^2+(u-c)^2\big)u'(y)}{(u(y)-c)^2}
\left(\f{1}{\phi_1^2(y,c)}-1\right)dy\\
&=A_1+A_2.\label{A=A_1+A_2}
\end{align}
 To prove  \eqref{est:A} and \eqref{est:A-1}, we claim that  there exist $C$ independent of $\al,c$ such that
\begin{align}
\label{A1}|A_1(c,\al)|&\leq C, \; |A_2(c,\al)|\leq C\al,\\
\label{A2}-A_2(c,\al)&\geq C^{-1}\al>0.
\end{align}
Thus,  the upper bounds follow by \eqref{A1}. Notice that for any $\al\geq 1$ and $\delta\in(0,\f12)$, the following holds
\begin{align}\label{A^2+c^2,csim 1}
A(c,\al)^2+4\pi^2c^2\geq 4\pi^2\delta^2,\quad \text{for}\quad c\in(-1,1)/(-\delta,\delta).
 \end{align}
 It  suffices to prove the lower bound in \eqref{est:A} for $c\in(-\delta,\delta)$ with some $\delta\in(0,\f12)$.

For $\al$ sufficient large (say $\al\geq N$), the lower bound follows from \eqref{A1} and \eqref{A2} that 
\begin{align*}
A(c,\al)^2+4\pi^2c^2\geq (A_1+A_2)^2
\geq \f12A_2^2-CA_1^2\geq C^{-1}\al^2,\quad c\in(-1,1).
\end{align*}
On the other hand, for any $\al\in\{2,3,...,N-1\}$, it follows from \eqref{A=0} and \eqref{def:W-4} that
\begin{align*}
A(0,\al)=W(0,\al)\neq 0,
\end{align*}
which together with the continuity of $A(\cdot,\al)$ shows \begin{align*}
A^2+4\pi^2c^2\geq A^2\geq C^{-1}\quad \text{for}\quad
 c\in[-\delta,\delta],
 \end{align*}
with some fixed $\delta$. Consequently, the lower bound of \eqref{est:A} follows.


To prove the claims \eqref{A1}-\eqref{A2},   we use the identity \eqref{A=A_1+A_2}. For $A_1$, we decompose that
\begin{align}
\notag A_1&=\int_{\{|y-y_c|\geq \f{1}{\al}\}}\f{2cu'(u-c)}{(u(y)-c)^2\phi_1(y,c)^2}dy+\int_{\{|y-y_c|\leq \f{1}{\al}\}}\f{2cu'(u-c)}{(u(y)-c)^2}(\f{1}{\phi_1(y,c)^2}-1)dy
\\
\label{label:A_1}&\quad+p.v.\int_{\{|y-y_c|\leq \f{1}{\al}\}}\f{2cu'(u-c)}{(u(y)-c)^2}dy+\int_{\mathbb{R}}\f{(u+c)^2}{\phi_1(y,c)^2}dy=A_{11}+A_{12}+A_{13}+A_{14}.
\end{align}
By \eqref{u-key}, \eqref{est:phi_1-ul} and  \eqref{est:phi_1-1}, we have
\begin{align*}
|A_{11}|&\lesssim \int_{\{|y-y_c|\geq \f{1}{\al}\}}\f{\al} {e^{C^{-1}\al|y-y_c|}}dy\leq C,
\end{align*}
and
\begin{align*}
|A_{12}|\lesssim \int_{\{|y-y_c|\leq \f{1}{\al}\}}\f{\al^2|y-y_c|^2\cdot\al}{\tanh\al|y-y_c|}dy\leq C.
\end{align*}
By \eqref{coshy} and  \eqref{fm:u(y)-c},  we have \begin{align}\label{est:u-c/u-c} 0<\f{u(y_c+\f{1}{\al})-c}{c-u(y_c-\f{1}{\al})}=\f{\cosh(y_c-\f{1}{\al})}{\cosh (y_c+\f{1}{\al})}\sim 1,
\end{align}
which together with \eqref{u-key} and \eqref{est:u-c/u-c} gives
\begin{align*}
|A_{13}|&= \Big|2c\ln\f{u(y_c+\f{1}{\al})-c}{c-u(y_c-\f{1}{\al})}\Big|\leq C.
\end{align*}
By \eqref{est:phi_1-ul}, we have
\begin{align*}
|A_{14}|&\lesssim \int_{\mathbb{R}}e^{-2C^{-1}\al|y-y_c|}dy\leq C\al^{-1}.
\end{align*}
Combining the estimates of $A_{11}-A_{14}$, we obtain $|A_1(c,\al)|\leq C$. For $A_2$, we decompose that
\begin{align*}
-A_2&=\int_{\{|y-y_c|\leq \f{1}{\al}\}}\f{(1-c^2+(u-c)^2)u'}{(u(y)-c)^2}\big(1-\f{1}{\phi_1(y,c)^2}\big)dy\\
&\quad-\int_{\{|y-y_c|\geq \f{1}{\al}\}}\f{(1-c^2+(u-c)^2)u'}{(u(y)-c)^2}\f{1}{\phi_1(y,c)^2}dy\\
&\quad+\int_{\{|y-y_c|\geq \f{1}{\al}\}}\f{(1-c^2+(u-c)^2)u'}{(u(y)-c)^2}dy
=A_{21}-A_{22}+A_{23}.
\end{align*}
By \eqref{u-key}, \eqref{est:phi_1-1} and \eqref{sinhy}, we have
\begin{align}\label{A22}
\notag0<A_{21}&\lesssim \int_{\{|y-y_c|\leq \f{1}{\al}\}}\big( \tanh^{-2}|y-y_c|+1\big)\al^2|y-y_c|^2dy\\
&\lesssim \al^{-1}\int_{\{|z|\leq 1\}}\big( \al^2|z|^{-2}+1\big)|z|^2dz\leq C\al.
\end{align}
By \eqref{u-key} and  \eqref{est:phi_1-ul}, we have
\begin{align*}
0<A_{22}&\lesssim \int_{\{|y-y_c|\geq \f{1}{\al}\}}\big(1+\al^2\tanh^{-2}\al|y-y_c|\big)e^{-2C^{-1}\al|y-y_c|}dy\leq C\al.
\end{align*}
By  \eqref{u-key} and \eqref{sinhy}, we have
\begin{align*}
0<A_{23}&\lesssim (1-c^2)\bigg|\f{1}{u(y)-c}\big|_{y_c+\f{1}{\al}}^{y_c-\f{1}{\al}}+\f{1}{u(y)-c}\big|_{-\infty}^{+\infty}\bigg|
+\bigg|u(y)\big|_{y_c+\f{1}{\al}}^{y_c-\f{1}{\al}}+u(y)\big|_{-\infty}^{+\infty}\bigg|\\
&\lesssim\f{1-c^2}{u(y_c+\f{1}{\al})-c}+\f{1-c^2}{c-u(y_c-\f{1}{\al})}+1
\lesssim \f{1}{\tanh \f{1}{\al}}+1\leq C\al.
\end{align*}
Combining the estimates of $A_{21}-A_{23}$, we have $|A_2(c,\al)|\leq C\al$.

To prove \eqref{A2}, we notice by $1-\f{1}{\phi_1^2}\geq 0$ that $-A_{22}+A_{23}\geq 0$. Thus,  thanks to \eqref{fm:u(y)-c},  \eqref{est:phi_1-1} and \eqref{sinhy},
 we have
 \begin{align*}
-A_2(c,\al)&\geq A_{21}\geq \int_{\{|y-y_c|\leq \f{1}{\al}\}}\f{(1-c^2)u'(y)}{(u(y)-c)^2}\big(1-\f{1}{\phi_1(y,c)^2}\big)dy\\
&\geq C^{-1}\int_{\{|y-y_c|\leq \f{1}{\al}\}}\f{\al^2|y-y_c|^2}{\sinh^2|y-y_c|}dy\geq C^{-1}\al>0.
\end{align*}
We finish the proof of \eqref{A1}-\eqref{A2}.

The identity \eqref{pacA=0} follows from Definition \ref{def:A},  Lemma \ref{lem:embedd-critA} and Lemma  \ref{lem:pa_cW1}.
\end{proof}

\begin{lemma}\label{lem:pa_cA}
Let $\al\geq 1$, $c\in(-1,1)$. There exists $C$ independent of $\al,c$ such that
\begin{align}
\label{est:pacA}|\pa_cA(c,\al)|\leq C,
\end{align}
and
\begin{align}
\label{est:packA}|\pa_c^kA(c,\al)|\leq C\al^{-1}(1-c^2)^{-k+2},\quad k=2,3,4.
\end{align}
\end{lemma}

\begin{proof}
We recall by Lemma \ref{lem:simple-useful-equ}, Lemma \ref{prop:goodphi1}  and \eqref{fm:paGintf} that
\begin{align}
&\pa_G\left(\f{(1-c^2)u'(z)}{(u-c)^2}\right)=0,\quad \pa_G\left(\f{(1-c^2)}{(u-c)}\right)=1\quad \pa_G\left(\f{(1-c^2)^2}{(u-c)^2}\right)=\f{2(1-c^2)}{u-c}\label{goodG-2A}.\\
&\Big|\f{\pa_G^k\phi_1(y,c)}{\phi_1(y,c)}\Big|
\leq C\left(\f{|u(y)-c|}{1-c^2}\right)^{k}(\al|y-y_c|)\langle\al|y-y_c|\rangle^{k-1}\label{goodG-A2},\quad k=1,2,3,4.\\
&\pa_c\Big(p.v.\int_{\mathbb{R}}f(c,y)dy\Big)
=p.v.\int_{\mathbb{R}}\pa_Gf(c,y)dy,\quad \text{if}\quad \lim_{y\to\infty}f(c,y)=0.\label{fm:paGintf-2}
\end{align}

By Definition \ref{def:A} and \eqref{fm:pa_cA}, we have
\begin{align}
\pa_cA
&\label{fm:pa_cA}=p.v.\int_\mathbb{R}\f{2(1-c^2)}{(u-c)\phi_1^2}dy
+\int_\mathbb{R}\f{(1-c^2)^2}{(u-c)^2}\pa_G(\f{1}{\phi_1^2})dy.
\end{align}
Thanks to $1-c^2=u'+(u-c)(u+c)$, we  decompose $\pa_cA$ as
\begin{align}
\notag \pa_cA&=p.v.\int_{\mathbb{R}}\f{2u'}{(u-c)\phi_1^2}dy+\int_{\mathbb{R}}\f{2(u+c)}{\phi_1^2}dy
+\int_{\mathbb{R}}\f{(1-c^2)^2}{(u-c)^2}\pa_G(\f{1}{\phi_1^2})dy\\
&=
A'_1+A'_2+A'_3.\label{fm:pa_cA''}
\end{align}
Here the superscript does not represent a derivative.
 We decompose $A'_1$ that
 \begin{align*}
 A'_1
 &=p.v.\int_{|y-y_c|\geq \f{1}{\al}}\f{2u'}{(u-c)\phi_1^2}dy
 +\int_{|y-y_c|\leq \f{1}{\al}}\f{2u'}{u-c}(\f{1}{\phi_1^2}-1)dy
 +p.v.\int_{|y-y_c|\leq \f{1}{\al}}\f{2u'}{u-c}dy\\
 &= A'_{11}+A'_{12}+A'_{13}.
 \end{align*}
  Similar as the estimates for $A_{11},...A_{14}$ in Lemma \ref{lem:A}, we can obtain $|A'_{1i}|\leq C$($1\leq i\leq 3$). 
 By \eqref{est:phi_1-ul}, we have
 \begin{align*}
 |A_2'|\lesssim \int_\mathbb{R}e^{-C^{-1}\al|y-y_c|}dy\leq C\al^{-1}.
 \end{align*}
 By \eqref{goodG-A2}($k=1$), \eqref{u-key}, \eqref{est:phi_1-ul} and  \eqref{sinhy}, we have
 \begin{align*}
 |A_3'|\lesssim \int_\mathbb{R}\f{ (1-c^2)\cdot\al|y-y_c|}{|u(y)-c|\phi_1 (y,c)^2}dy\lesssim \int_\mathbb{R}\f{\al^2|y-y_c|}{e^{2C^{-1}\al|y-y_c|}\tanh\al|y-y_c|}dy\leq C.
 \end{align*}
 Combining the estimates of $A_{1}'-A_{3}'$, we have \eqref{est:pacA}.

Taking $\pa_c$ on \eqref{fm:pa_cA}, thanks to  \eqref{fm:paGintf-2}, Leibniz law and \eqref{goodG-2A}, we have  
\begin{align}\label{fm:pa_c^2A}
\pa_c^2A=\int_\mathbb{R}\f{2}{\phi_1^2}dy+4\int_\mathbb{R}\f{1-c^2}{u-c}\pa_G(\f{1}{\phi_1^2})dy
+\int_\mathbb{R}\f{(1-c^2)^2}{(u-c)^2}\pa_G^2(\f{1}{\phi_1^2})dy=A_1''+A_2''+A_3''.
\end{align}
 By \eqref{goodG-A2}($k=1,2$), \eqref{u-key} and  \eqref{est:phi_1-ul}, we have
 \begin{align*}
 |A_1''|&\lesssim \int_\mathbb{R}\f{1}{e^{2C^{-1}\al|y-y_c|}}dy\leq C\al^{-1},\quad
 |A_2''|\lesssim \int_\mathbb{R}\f{\al|y-y_c|}{ e^{2C^{-1}\al|y-y_c|}}dy\leq C\al^{-1},\\
 |A_3''|&\lesssim \int_\mathbb{R}\f{\langle \al|y-y_c|\rangle^2}{\phi_1(y,c)^2}dy\lesssim \int_\mathbb{R}\f{\langle \al|y-y_c|\rangle^2}{ e^{2C^{-1}\al|y-y_c|}}dy\leq C\al^{-1}.
 \end{align*}
 Thus,  \eqref{est:packA}($k=2$) follows.

Taking $\pa_c$ on  \eqref{fm:pa_c^2A}, thanks to  \eqref{fm:paGintf-2}, Leibniz law and \eqref{goodG-2A}, we have  
\begin{align}\label{fm:pa_c^3A}
\pa_c^3A&=6\int_\mathbb{R}\pa_G(\f{1}{\phi_1^2})dy+6\int_\mathbb{R}\f{1-c^2}{u-c}\pa_G^2(\f{1}{\phi_1^2})dy
+\int_\mathbb{R}\f{(1-c^2)^2}{(u-c)^2}\pa_G^3(\f{1}{\phi_1^2})dy\\
&\notag=A_1'''+A_2'''+A_3'''.
\end{align}
By \eqref{goodG-A2}($k=1,2,3$), \eqref{u-key} and  \eqref{est:phi_1-ul},  we have
 \begin{align*}
 |A_1'''|&\lesssim (1-c^2)^{-1}\int_\mathbb{R}\f{|u(y)-c|\al|y-y_c|}{e^{2C^{-1}\al|y-y_c|}}dy\leq C\al^{-1} (1-c^2)^{-1},\\
  |A_2'''|&\lesssim \int_\mathbb{R}\f{\langle \al|y-y_c|\rangle^2|u(y)-c|}{(1-c^2)\phi_1(y,c)^2}dy\lesssim \int_\mathbb{R}\f{(1-c^2)^{-1} \langle\al|y-y_c|\rangle^2}{e^{2C^{-1}\al|y-y_c|}}dy\leq C\al^{-1} (1-c^2)^{-1},\\
 |A_3'''|&\lesssim \int_\mathbb{R}\f{\langle \al|y-y_c|\rangle^3|u(y)-c|}{(1-c^2)\phi_1(y,c)^2}dy\lesssim \int_\mathbb{R}\f{(1-c^2)^{-1}\langle \al|y-y_c|\rangle^3}{e^{2C^{-1}\al|y-y_c|}}dy\leq C\al^{-1}(1-c^2)^{-1}.
 \end{align*}
Thus, \eqref{est:packA}($k=3$) follows.

Taking $\pa_c$ on  \eqref{fm:pa_c^3A}, thanks to  \eqref{fm:paGintf-2}, Leibniz law and \eqref{goodG-2A}, we have 
\begin{align*}
\pa_c^4A=12\int_\mathbb{R}\pa_G^2(\f{1}{\phi_1^2})dy+8\int_\mathbb{R}\f{1-c^2}{u-c}\pa_G^3(\f{1}{\phi_1^2})dy
+\int_\mathbb{R}\f{(1-c^2)^2}{(u-c)^2}\pa_G^4(\f{1}{\phi_1^2})dy=A_1''''+A_2''''+A_3''''.
\end{align*}
 By \eqref{goodG-A2}($k=1,2,3,4$), \eqref{u-key} and  \eqref{est:phi_1-ul},  we have
 \begin{align*}
 |A_1''''|&\lesssim (1-c^2)^{-2}\int_\mathbb{R}\f{\langle\al|y-y_c|\rangle^2}{e^{2C^{-1}\al|y-y_c|}}dy\leq C\al^{-1} (1-c^2)^{-2},\\
  |A_2''''|&\lesssim (1-c^2)^{-2} \int_\mathbb{R}\f{\langle\al|y-y_c|\rangle^3}{e^{2C^{-1}\al|y-y_c|}}dy\leq C\al^{-1} (1-c^2)^{-2},\\
 |A_3''''|&\lesssim \int_\mathbb{R}\f{\langle \al|y-y_c|\rangle^4|u(y)-c|^2}{(1-c^2)^2\phi_1(y,c)^2}dy\lesssim \int_\mathbb{R}\f{(1-c^2)^{-2}\langle \al|y-y_c|\rangle^4}{e^{2C^{-1}\al|y-y_c|}}dy\leq C\al^{-1}(1-c^2)^{-2}.
 \end{align*}
Thus, \eqref{est:packA}($k=4$) follows.
\end{proof}

\section{Limiting absorption principles}
In this section, we establish the limiting absorption principle(LAP) for the inhomogeneous solution constructed in  Section 3.

\subsection{LAP in presence of embedding eigenvalue   }

\begin{proposition}\label{lem:B-delta1}
 Let $f\in H^1_{1/(u')^2}(\mathbb{R})$,  $\Phi(1,y,\textbf{c})$ be the inhomogeneous solution constructed in Proposition \ref{pro:solve-Phi}. Then for $y\neq 0$, it holds that
\begin{align}
\label{lim:T-theta+}
\lim_{\pm\mathrm{Im}(\textbf{c})>0,\;|\textbf{c}|\to 0}\textbf{c}
\Phi(1,y,\textbf{c})=\f{\mathcal{T}(f)(0)\pm i\pi f(0)}{\pm2\pi i}\mathrm{sech}\,y ,
\end{align}
 with the operator
$\mathcal{T}$ defined as
\begin{align}\label{def:T(h)-(-1,1)} \mathcal{T}(f)(c)&:=p.v.\int_{\mathbb{R}}\f{\int_{y_c}^y\phi_1(z,c)f(z)dz}{(u(y)-c)^2\phi_1(y,c)^2}dy,\quad
c\in(-1,1).
\end{align}

For $y\neq 0$, there exists  $\delta_1(y)>0$, such that for $\delta\in (0,\delta_1)$ it holds that \begin{align}
\label{bd:Phi}\Big|\textbf{c}\Phi(1,y,\textbf{c})\Big|\leq C,\quad \text{for}\;\textbf{c}\in B_{\delta}(0)/(-1,1)\subset O_{\ep_0},\end{align}
where the constant $C$ is independent of $\textbf{c}$.
\end{proposition}

\begin{proof}
Thanks to \eqref{solution-Phi-2}-\eqref{def:T(h)}, we reformulate $\textbf{c}\Phi$ as 
\begin{align}\label{solution-Phi'-l}
\textbf{c}\Phi(1,y,\textbf{c})&=\textbf{c}\phi(y,\textbf{c})
\int_{-\infty}^y\f{\int_{y_c}^z\phi_1(w,\textbf{c})f(w)dw}
{\phi(z,\textbf{c})^2}dz
-\f{T(f)(\textbf{c})}{\f{ W(\textbf{c},1)}{\textbf{c}}}\int_{-\infty}^y\f{\phi(y,\textbf{c})}
{\phi(z,\textbf{c})^2}dz\\
&=\textbf{c}\phi(y,\textbf{c})\int_{+\infty}^y\f{\int_{y_c}^z
\phi_1(w,\textbf{c})f(w)dw}{\phi(z,\textbf{c})^2}dz
-\f{T(f)(\textbf{c})}{\f{ W(\textbf{c},1)}{\textbf{c}}}\int_{+\infty}^y
\f{\phi(y,\textbf{c})}{\phi(z,\textbf{c})^2}dz,
\label{solution-Phi'-2-l}
\end{align}
where
\begin{align}\label{def:T(h)-2}
T(f)(\textbf{c})&:=\int_{\mathbb{R}}\f{\int_{y_c}^z\phi_1(w,\textbf{c})f(w)dw}{(u(z)-\textbf{c})^2\phi_1(z,\textbf{c})^2}dz.
\end{align}
It suffices to prove the lemma for $y<0$ using \eqref{solution-Phi'-l}, since the process for $y>0$ is similar  using \eqref{solution-Phi'-2-l}.
Fix $y<0$, we choose $\delta_0(y)=\min\{\f12|u(y)|,\ep_0\}$ such that for $\delta\in(0,\delta_0)$,
$\textbf{c}\in B_{\delta}(0)/(-1,1)$,
\begin{align}\label{est:u-texbfc}
|u(y)-\textbf{c}|\geq |u(y)|-|\textbf{c}|\geq \f12|u(y)|,
\end{align}
and
\begin{align}\label{est:u-texbfc-1}
 y<y_c,\quad \text{with}\quad y_c=u^{-1}(\mathrm{Re}(\textbf{c})).
\end{align}
 Then we can consider $y<y_c$, $\textbf{c}\in B_{\delta}(0)/(-1,1)$ and $f\in H^1_{1/(u')^2}(\mathbb{R})$. Thanks to \eqref{Oep-1}, \eqref{Oep-3} and \eqref{est:u-texbfc}, we have
\begin{align}\label{absorption-0-1}
\Big|\textbf{c}\phi(y,\textbf{c})
\int_{-\infty}^y\f{\int_{y_c}^z\phi_1(w,\textbf{c})f(w)dw}
{\phi(z,\textbf{c})^2}dz\Big|\lesssim |\textbf{c}|\cdot
e^{-C^{-1}\al|z-y|}
\leq C|\textbf{c}|,
\end{align}
and
\begin{align}\label{absorption-0-2}
\Big|\f{\phi(y,\textbf{c})}{\phi(z,\textbf{c})^2}
\Big|\lesssim \f{e^{-C^{-1}\al|z-y_c|}}{|u(y)-\textbf{c}|}
\leq C,
\end{align}
where $C$ is in dependent of $\textbf{c}$.
Using \eqref{absorption-0-2} and the dominated convergence theorem, we have by \eqref{cosh-1y} that
\begin{align}\label{absorption-0-3}
\lim_{\textbf{c}\to 0}
\int_{-\infty}^y\f{\phi(y,\textbf{c})}{\phi(z,\textbf{c})^2}dz
=\int_{-\infty}^y\f{\phi(y,0)}{\phi(z,0)^2}dz=-\mathrm{sech}\,y .
\end{align}

To proceed it, we claim that
\begin{align}\label{textbfc-lim1}
\lim_{\pm\mathrm{Im}(\textbf{c})>0,\;\textbf{c}\to 0}\mathcal{T}(f)(\textbf{c})=\mathcal{T}(f)(0)\pm i\pi f(0),\quad\text{for}\; f\in H^1_{1/(u')^2}(\mathbb{R}).
\end{align}
Combining \eqref{absorption-0-1}, \eqref{absorption-0-3}, \eqref{textbfc-lim1} with Lemma \ref{lem:pa_cW_1}, we have 

\begin{align*}
\lim_{\pm\mathrm{Im}(\textbf{c})>0,\;\textbf{c}\to 0}\textbf{c}
\Phi(1,y,\textbf{c})&=
-\f{\pm\mathrm{Im}(\textbf{c})>0,\;\lim_{\textbf{c}\to 0}\mathcal{T}(f)(\textbf{c})}{
\lim_{\pm\mathrm{Im}(\textbf{c})>0,\;\textbf{c}\to 0}\f{W(\textbf{c},1)}{\textbf{c}}}
\cdot\lim_{\pm\mathrm{Im}(\textbf{c})>0,\;\textbf{c}\to 0}\int_{-\infty}^y\f{\phi(y,\textbf{c})}{\phi(z,\textbf{c})^2}dz\\
&=\f{\mathcal{T}(f)(0)\pm i\pi f(0)}{\pm 2\pi i}\mathrm{sech}\,y .
\end{align*}
Using \eqref{absorption-0-1}, \eqref{absorption-0-3}, \eqref{textbfc-lim1}, Lemma \ref{lem:pa_cW_1}  and \eqref{solution-Phi'-l} again, we have that for $y<0$, there exists $\delta_1(y)\in (0,\delta_0)$, such that for $\delta\in(0,\delta_1)$, $\textbf{c}\in B_{\delta}(0)/(-1,1)$, it holds that

\begin{align}
\Big|\textbf{c}\Phi(\al,y,\textbf{c})\Big|
\leq C|\textbf{c}|+\f{2\Big(|\mathcal{T}(f)(0)|+\pi|f(0)|\Big)}{\pi}\cdot Ce^{-2C^{-1}|y-y_c|}\leq C.
\end{align}

It remains to prove \eqref{textbfc-lim1}. We decompose \eqref{def:T(h)-2}  as

\begin{align}
\notag\mathcal{T}(f)(\textbf{c})
&=
\int_{\mathbb{R}}\int_{y_c}^z
\f{f(w)}{(u(z)-\textbf{c})^2}
\left(\f{\phi_1(w,\textbf{c})}
{\phi_1(z,\textbf{c})^2}
-\f{u'(z)}{1-c^2}\right)dwdz
+(1-c^2)^{-1}
\int_{\mathbb{R}}\f{u'(z)g(z,c)}{(u(z)-\textbf{c})^2}
dz\\
&\quad+(1-c^2)^{-2}f(y_c)
\int_{\mathbb{R}}\f{u'(z)(u(z)-c)}{(u(z)-\textbf{c})^2}dz
:=I_1(\textbf{c})+I_2(\textbf{c})+I_3(\textbf{c}),\label{dec:T(bfc)}
\end{align}
where $c=\mathrm{Re}(\textbf{c})$ and
$g(z,c):=\int_{y_c}^zf(w)dw-\f{(u(z)-c)f(y_c)}{1-c^2}$ with $g(y_c)=0$, $g'(y_c)=0$.

For  $I_1$, we decompose the integrated function  and use  \eqref{est:phi_1-ul-c}, \eqref{phi_1(z)leq phi_1(y)} and  \eqref{est:up-phi-1} to obtain for $y_c\leq w\leq z$ or
$z\leq w\leq y_c$ that
\begin{align}
&\quad\f{1}{|u(z)-\textbf{c}|^2}\notag
\cdot\Big|\f{\phi_1(w,\textbf{c})}
{\phi_1(z,\textbf{c})^2}-\f{u'(z)}{1-c^2}\Big|\\
&\leq \f{1}{|u(z)-\textbf{c}|^2}\cdot
\Big(\Big|\f{\phi_1(w,\textbf{c})-1}
{\phi_1(z,\textbf{c})^2}\Big|
+\Big|\f{1-\phi_1(z,\textbf{c})^2}{\phi_1(z,\textbf{c})^2}\Big|\Big)
\cdot \f{u'(z)}{1-c^2}
+\f{|1-\f{u'(z)}{1-c^2}|}{|u(z)-\textbf{c}|^2}\cdot \f{\phi_1(w,\textbf{c})}{\phi_1(z,\textbf{c})^2}\notag\\
&\lesssim \f{\min\{\al^2|z-y_c|^2,1\}u'(z)}{(1-c^2)|u(z)-c|^2}+
\f{e^{-C^{-1}\al|z-y_c|}}{(1-c^2)|u(z)-c|}.
\notag
\end{align}
Consequently,  using
\begin{align*}
\Big|\int_{y_c}^zf(w)dw\Big|\leq |u(z)-c|\cdot \|f/u'\|_{L^{\infty}(\mathbb{R})}\leq |u(z)-c|\cdot \|f\|_{H^1_{1/(u')^2}(\mathbb{R})},
\end{align*}
 we have
\begin{align}
&\quad \bigg|\int_{y_c}^z
\f{f(w)}{(u(z)-\textbf{c})^2}
\left(\f{\phi_1(w,\textbf{c})}
{\phi_1(z,\textbf{c})^2}
-\f{u'(z)}{1-c^2}\right)dw\bigg|\notag\\
&\leq
C(1-c^2)^{-1}\|f/u'\|_{L^{\infty}(\mathbb{R})}
\cdot \bigg(\f{\min\{\al^2|z-y_c|^2,1\}u'(z)}{|u(z)-c|}
+e^{-C^{-1}\al|z-y_c|}\bigg)\in L^1_z(\mathbb{R}),\label{est:I1(c)-1}
\end{align}
where the constant $C$ is independent of $\textbf{c}$ and $z$.
Thus, using  the dominated convergence theorem and the continuity of $\phi_1(y,\cdot)$, we have
  \begin{align}\label{lim:I1}
\lim_{\textbf{c}\to 0}
I_1(\textbf{c})= \int_{\mathbb{R}}
\int_{0}^z\f{f(w)}{(u(z)-\textbf{c})^2}
\left(\f{\phi_1(w,c)}{\phi_1(z,c)^2}-\f{u'(z)}{1-c^2}\right )\Big|_{\textbf{c}=0}dwdz=I_1(0).
\end{align}
For $I_2$, we notice that
 \begin{align}
 g(z,c)=(z-y_c)^2\iint_{[0,1]^2}tg''(st(z-y_c)+y_c)dsdt,\quad g''(z,c)=f'(z)-\f{f(y_c)}{1-c^2}u'(z).
 \end{align}
 Therefore, we have by Cauchy-Schwarz inequality that
  \begin{align}
\notag  |g(z,c)|&\leq |z-y_c|
 \Big|\iint_{[0,1]^2}t(z-y_c) (|g''|^2/u')(st(z-y_c)+y_c)dsdt\Big|^{\f12}
\\
\notag &\quad\times
  \Big|\iint_{[0,1]^2}t(z-y_c) u'(st(z-y_c)+y_c)dsdt\Big|^{\f12}\\
 \notag  &\leq |z-y_c|\Big(\int_{\mathbb{R}}(|g''|^2/u')(\tilde{s})d\tilde{s}\Big)^{\f12}
  \Big(\int_0^1|u(t(z-y_c)+y_c)-c|dt\Big)^{\f12}\\
  &\lesssim |z-y_c|\big(\|f'\|_{L^2_{1/u'}}+\|f/u'\|_{L^{\infty}}\big) |u(z)-c|^{\f12}\notag\\
  &\leq C|z-y_c|\cdot|u(z)-c|^{\f12}\|f\|_{H^1_{1/(u')^2}},\notag
 \end{align}
 which together with Lemma \ref{lem:simple-useful-inequ} gives for $\textbf{c}\in B_{\delta}(0)/(-1,1)$ and $c=\mathrm{Re}(\textbf{c})$,
  \begin{align}
\notag\Big|(1-c^2)^{-1}\f{u'(z)g(z)}{(u(z)-\textbf{c})^2}
\Big|&\leq (1-c^2)^{-1}\f{u'(z)|z-y_c|}{|u(z)-\textbf{c}|^{\f32}}\\
\notag&=(1-c^2)^{-1}\f{\cosh^{\f32}y_c\cdot |z-y_c|}{\sinh^{\f32}|z-y_c|\cosh ^{\f12}z}\\
\notag&\leq C(1-c^2)^{-\f32}\f{\cosh|z-y_c|\cdot |z-y_c|}{\sinh^{\f32}|z-y_c|}\\
&\leq C(1-c^2)^{-\f32}\max\{|z-y_c|^{-\f12},|z-y_c|\cdot e^{-|z-y_c|/2}\}\in L^1_z(\mathbb{R}),\label{est:I2(c)}
 \end{align}
 where the constant $C$ is independent of $\textbf{c}$, $y$.

Thus,  using  the dominated convergence theorem, we have

\begin{align}\label{lim:I2}
\lim_{\textbf{c}\to 0}I_2(\textbf{c})=(1-c^2)^{-1}
\int_{\mathbb{R}}\f{u'(z)g(z,c)}{(u(z)-\textbf{c})^2}
dz\Big|_{\textbf{c}=0}=I_2(0).
\end{align}
For $I_3$, we have
\begin{align}
I_3(\textbf{c})
\notag&=(1-c^2)^{-2}f(y_c)\Big(
\int_{\mathbb{R}}\f{u'(z)}{u(z)-\textbf{c}}dz
+i\mathrm{Im}(\textbf{c})
\int_{\mathbb{R}}\f{u'(z)}{(u(z)-\textbf{c})^2}dz\Big)\\
\notag&=(1-c^2)^{-2}f(y_c)\Big(
\ln\f{1-\textbf{c}}{-1-\textbf{c}}
-i\mathrm{Im}(\textbf{c})(\f{1}{1-\textbf{c}}
-\f{1}{-1-\textbf{c}})\Big)\\
&=(1-c^2)^{-2}f(y_c)\Big(
\ln\f{\textbf{c}-1}{\textbf{c}+1}
-\f{2\mathrm{Im}(\textbf{c})i}{1-\textbf{c}^2}\Big),\label{est:I3(c)}
\end{align}
which together with \eqref{ln-limit-1} gives
\begin{align}\label{lim:I3}
\notag\lim_{\pm\mathrm{Im}(\textbf{c})>0,\;\textbf{c}\to 0}I_3(\textbf{c})
&=(1-c^2)^{-2}f(y_c)\Big(
\ln\f{\textbf{c}-1}{\textbf{c}+1}
-\f{2\mathrm{Im}(\textbf{c})i}{1-\textbf{c}^2}\Big)\Big|_{\textbf{c}=0}
+
f(0)\cdot \mathrm{sgn}(\mathrm{Im}(\textbf{c}))\pi i\\
&=I_3(0)\pm i\pi f(0).
\end{align}
Then \eqref{textbfc-lim1} follows from \eqref{dec:T(bfc)},  \eqref{lim:I1},  \eqref{lim:I2} and \eqref{lim:I3}.
\end{proof}

\begin{Corollary}\label{Cor:c=delta}
 Let $\delta\in(0,\delta_1)$, 
 $\Phi(1,y,\textbf{c})$ be the inhomogeneous solution constructed in Proposition \ref{pro:solve-Phi} with $f=\hat{\omega}_0(1,y)$.
Then for    $y\neq 0$, it holds that
\begin{align}\label{def:projection}
\lim_{\delta\to 0+}\f{1}{2\pi i}\int_{\{|\textbf{c}|=\delta,\textbf{c}\neq \pm \delta\}}e^{-i\al \textbf{c}t}\Phi(1,y,\textbf{c})d\textbf{c}=\f12\hat{\omega}_0(1,0)\mathrm{sech}\,y .
\end{align}
\end{Corollary}

\begin{proof}
Let $y\neq 0$, $\delta_1(y)$ be chosen in Proposition \ref{lem:B-delta1}.
We write the contour in \eqref{def:projection} as
\begin{align}
\textbf{c}=\delta e^{i\theta},\quad\theta\in (-\pi,0)\cup(0,\pi), \; \delta\in(0,\delta_1).
\end{align}
Therefore, using \eqref{lim:T-theta+} and \eqref{bd:Phi} , we have by the dominated convergence theorem that
\begin{align*}
&\lim_{\delta\to 0+}\f{1}{2\pi i}\int_{\{|\textbf{c}|=\delta,\textbf{c}\neq \pm \delta\}}e^{-i\al \textbf{c}t}\Phi(1,y,\textbf{c})d\textbf{c}\\
&=\f{1}{2\pi }\lim_{\delta\to 0+}\int_{0}^{\pi}e^{-i\al \textbf{c}t}\textbf{c}\Phi(1,y,\textbf{c})d\theta+
\f{1}{2\pi }\lim_{\delta\to 0+}\int_{-\pi}^{0}e^{-i\al \textbf{c}t}\textbf{c}\Phi(1,y,\textbf{c})d\theta\\
&=\f{1}{2\pi }\int_{0}^{\pi}\lim_{\delta\to 0}e^{-i\al \textbf{c}t}\textbf{c}\Phi(1,y,\textbf{c})d\theta+\f{1}{2\pi }\int_{-\pi}^{0}
\lim_{\delta\to 0}e^{-i\al \textbf{c}t}\textbf{c}\Phi(1,y,\textbf{c})d\theta\\
&=\f{1}{2\pi }\int_{0}^{\pi}\f{\mathcal{T}(\omega_0)(0)+i\pi\omega_0(0)}{2\pi i}\mathrm{sech}\,y d\theta+\f{1}{2\pi }\int_{-\pi}^{0}
\f{\mathcal{T}(\omega_0)(0)-i\pi\omega_0(0)}{-2\pi i}\mathrm{sech}\,y d\theta\\
&=\f12\omega_0(0)\mathrm{sech}\,y .
\end{align*}

\end{proof}

\subsection{LAP without embedding eigenvalue   }
Let
\begin{align}
\begin{split}
\label{def:al2}
&(i) \;\al\geq 2, c\in(-1,1),\\
&(ii) \;\al=1,\; c\in (-1,1)/\{0\}.
\end{split}
\end{align}

We define
\begin{align}
\Phi_{\pm}(\al,y,c)=
\left\{
\begin{aligned}\label{Phi_pm'}
&\phi(y,c)\int_{-\infty}^y
\f{\int_{y_c}^z\phi_1(w,c)f(w)dw}{\phi(z,c)^2}dz
+\mu_{\pm}(c,\al)\Gamma(\al,y,c),\;y<y_c,\\
&\phi(y,c)\int_{+\infty}^y\f{\int_{y_c}^z\phi_1(w,c)f(w)dw}
{\phi(z,c)^2}dz
+\mu_{\pm}(c,\al)\Gamma(\al,y,c), \;y>y_c,
\end{aligned}
\right.
\end{align}
with $f(y)=\hat{\omega}_0(\al,y)$.
Here
\begin{align}\label{def:mupm}
\mu_{\pm}(c,\al)&
:=-\f{(1-c^2)^{2}\mathcal{T}(f)(c)\pm f(y_c)\pi i}{A(c,\al)\pm2c\pi i},
\end{align}
$A(c,\al)$ is defined in \eqref{def:A},  $\mathcal{T}(\cdot)(c)$ is defined in \eqref{def:T(h)-(-1,1)}.
For $\al\geq 1$, $c\in(-1,1)$, $y\neq y_c$, $\Gamma$
is defined as
\beq\label{def:Gamma}
\Gamma(\al,y,c):=
\left\{
\begin{aligned}
&\phi(y,c)\int_{-\infty}^y\f{1}{\phi(z,c)^2}dz\;y<y_c,\\
&\phi(y,c)\int_{+\infty}^y\f{1}{\phi(z,c)^2}dz\;y>y_c.
\end{aligned}
\right.
\eeq
 \begin{proposition}\label{lem:limitfunction}
 Let $f\in H^1_{1/(u')^2}(\mathbb{R})$ and $\Phi(\al,y,c_{\ep})$ be the inhomogeneous solution constructed in Proposition \ref{pro:solve-Phi} for $c_{\ep}\in O_{\ep_0}$.
Then  it holds for \eqref{def:al2} and $y\neq y_c$ that
\begin{align}\label{limitfunction}
\lim_{\ep\to 0\pm}\Phi(\al,y,c_{\ep})=\Phi_{\pm}(\al,y,c).
\end{align}
Let
\begin{align}
\begin{split}
\label{def:al3}
&(i) \;\al\geq 2, c_{\ep}\in O_{\ep_0},\\
&(ii) \;\al=1,\; c_{\ep}\in O_{\ep_0}\cap\{c_{\ep}||c|\geq \delta\},\quad \delta\in(0,\f12).
\end{split}
\end{align}
There exists constants $\ep_0$ small enough and $C$  independent of $y, c,\ep$ such that  for \eqref{def:al3} and $y\neq y_c$, it  holds that
\begin{align}\label{Uniformbd-2}
|\Phi(\al,y,c_{\ep})|\leq C(1-c^2)^{-\f12}e^{-C^{-1}|y-y_c|}\|f/u'\|_{H^1(\mathbb{R})}.
\end{align}
\end{proposition}
\begin{remark}
We establish the LAP for \eqref{def:al2}, without embedding eigenvalue. While for the case of embedding eigenvalue $\al=1,c=0$, the proof of \eqref{limitfunction} fails, since the limit of \eqref{lim:W*} is zero, by Lemma \ref{lem:embedd-critA}. This is the reason why we establish a new type of LAP in \eqref{lim:T-theta+}.
\end{remark}
\begin{proof}
It suffices to give the proof for
$y<y_c$, since the proof for  $y>y_c$ is  similar.

The proof of \eqref{limitfunction} is similar with the proof of \eqref{lim:T-theta+}. We will omit some details.
We write  by \eqref{solution-Phi-2}, \eqref{def:W-mu} and \eqref{def:T(h)} that
\begin{align}\label{solution-Phi''-l}
\Phi(\al,y,c_{\ep})&=\phi(y,c_{\ep})
\int_{-\infty}^y\f{\int_{y_c}^z\phi_1(w,c_{\ep})f(w)dw}
{\phi(z,c_{\ep})^2}dz
-\f{T(f)(c_{\ep})}{ W(c_{\ep},\al)}\int_{-\infty}^y\f{\phi(y,c_{\ep})}
{\phi(z,c_{\ep})^2}dz.
\end{align}

Thanks to \eqref{Oep-1} and \eqref{Oep-3}, we have for $f\in H^1_{1/(u')^2}$, $c_{\ep}\in O_{\ep}$ and $z\leq y<y_c$ that
\begin{align}\label{absorption-1-1}
\Big|
\f{\phi(y,c_{\ep})\int_{y_c}^z\phi_1(w,c_{\ep})f(w)dw}
{\phi(z,c_{\ep})^2}\Big|\leq C
e^{-C^{-1}\al|z-y|}\in L^1(-\infty,y),
\end{align}
and
\begin{align}\label{absorption-1-2}
\Big|\f{\phi(y,c_{\ep})}{\phi(z,c_{\ep})^2}
\Big|\lesssim \f{e^{-C^{-1}\al|z-y_c|}\cdot |u(z)-c_{\ep}|}{|u(y)-c_{\ep}|^2}\leq  \f{Ce^{-C^{-1}\al|z-y_c|}}{|u(y)-c|}\in L^1(-\infty,y),
\end{align}
where the constant $C$ is independent of $c,\ep,y$.
Thus, using  the dominated convergence theorem and the continuity of $\phi_1(y,\cdot)$, we have for $y<y_c$
\begin{align}\label{limit-absorption-1-1}
\lim_{\ep\to0\pm}\int_{-\infty}^y\f{\phi(y,c_{\ep})\int_{y_c}^z\phi_1(w,c_{\ep})f(w)dw}
{\phi(z,c_{\ep})^2}dz=
\int_{-\infty}^y\f{\phi(y,c)\int_{y_c}^z\phi_1(w,c)f(w)dw}
{\phi(z,c)^2}dz,
\end{align}
and
\begin{align}\label{limit-absorption-1-2}
\lim_{\ep\to0\pm}\int_{-\infty}^y\f{\phi(y,c_{\ep})}
{\phi(z,c_{\ep})^2}dz=
\int_{-\infty}^y\f{\phi(y,c)}
{\phi(z,c)^2}dz.
\end{align}
We recall by \eqref{lim:W} that
\begin{align}\label{lim:W*}
\lim_{\ep\to 0\pm}W(c_{\ep},\al)=(1-c^2)^{-2}\Big(A(c,\al)\pm2\pi ci\Big).
\end{align}
Here we note by Lemma \ref{lem:embedd-critA} that the limit is non-zero \textbf{if and only if  \eqref{def:al2} holds}.
Similarly as \eqref{textbfc-lim1}, we can prove
\begin{align}\label{textbfc-lim1-2}
\lim_{\ep\to 0\pm}\mathcal{T}(f)(c_{\ep})=\mathcal{T}(f)(c)\pm  \f{i\pi f(y_c)}{(1-c^2)^2} ,
\end{align}
 where we use the decomposition \eqref{dec:T(bfc)}($\textbf{c}$ is replaced by $c_{\ep}$).
Thus, it follows from \eqref{limit-absorption-1-1}-\eqref{textbfc-lim1-2} that
 \eqref{limitfunction} holds for \eqref{def:al2}.

To prove \eqref{Uniformbd-2}, we notice for \eqref{def:al3} that $1-c_{\ep}^2\sim 1-c^2$.
For $f\in H^1_{1/(u')^2}$ and \eqref{def:al3}, we have by \eqref{absorption-1-1} that
\begin{align}
 \label{est:phi-omega0-1} &\Big|\phi(y,c_{\ep})\int_{-\infty}^y\f{\int_{y_c}^z
 \phi_1(w,c_{\ep})f(w)dw}{\phi(z,c_{\ep})^2}dz\Big|
\leq C.
\end{align}
We recall by Lemma  \ref{lem:W1}  that
\begin{align*}
W(c_{\ep},\al)&=(1-c_{\ep}^2)^{-2}\Big(W_1(c_{\ep},\al)+2c_{\ep}\ln\f{c_{\ep}-1}{c_{\ep}+1}\Big),\\
W(\cdot,\al)\;&\text{is}\;\text{analytical}\;\text{in}\;\tilde{O}_{\ep_0},\quad
\lim_{c_{\ep\to \pm 1 }}\ln\f{c_{\ep}-1}{c_{\ep}+1}=\infty.
\end{align*}
Therefore, there exists $C$ large enough and $\delta>0$ small enough such that
\begin{align*}
(1-c_{\ep}^2)^2|W(c_{\ep},\al)|\geq C,\quad \text{for}\;c_{\ep}\in \{ O_{\ep_0}|c\leq -1+\delta\;\text{or}\;c\geq 1-\delta\},
\end{align*}
where $C,\delta$ are independent of $c_{\ep}$.
This together with \eqref{lim:W*} and Lemma \ref{lem:A} gives that, there exists $\ep_0$ small enough, such that for \eqref{def:al3},
\begin{align}\label{bd:W*}
(1-c_{\ep}^2)^2|W(c_{\ep},\al)|\geq C^{-1},
\end{align}
where $C$ is independent of $c,\ep$.
To proceed it, we   claim  for \eqref{def:al3} that
\begin{align}\label{est:phi/phi^2}
&\Big|\int_{-\infty}^y\f{\phi(y,c_{\ep})}{\phi(z,c_{\ep})^2}dz\Big|
 \leq C(1-c^2)^{-1}e^{-C^{-1}\al|y-y_c|},\quad y<y_c.
\end{align}
\begin{align}\label{textbfc-bd}
|\mathcal{T}(f)(c_{\ep})|\leq C(1-c^2)^{-\f32}\|f\|_{H^1_{1/(u')^2}}.
\end{align}
Here the constant $C$ is independent of $c,\ep$.
Thus, \eqref{Uniformbd-2} follows from \eqref{solution-Phi''-l}, \eqref{est:phi-omega0-1}-\eqref{textbfc-bd} and $1-c_{\ep}^2\sim 1-c^2$($c_{\ep}\in O_{\ep_0}$).

To prove \eqref{est:phi/phi^2},  we assume $y\in[y_c-1,y_c)$, since the case $y<y_c-1$ is similar and simpler. Thanks to \eqref{absorption-1-2} and Lemma \ref{lem:simple-useful-inequ}, we have for $c_{\ep}\in O_{\ep_0}$  and $y\in[y_c-1,y_c)$ that
\begin{align*}
\Big|\int_{-\infty}^y\f{\phi(y,c_{\ep})}{\phi(z,c_{\ep})^2}dz\Big|
&\lesssim \int_{-\infty}^{y_c-1}
\f{e^{-C^{-1}\al|z-y_c|}}{|u(z)-c|}dz
+\int^{y}_{y_c-1}
\f{e^{-C^{-1}\al|z-y_c|}|u(y)-c|}{|u(z)-c|^2}dz\\
&\lesssim \int_{-\infty}^{y_c-1}
\f{e^{-C^{-1}\al|z-y_c|}}{(c-u(y_c-1))}dz
+\int^{y}_{y_c-1}
\f{|u(y)-c|u'(z)}{(u(z)-c)^2}dz\\
&\leq C(1-c^2)^{-1}\leq C(1-c^2)^{-1}e^{-C^{-1}\al|y-y_c|}.
\end{align*}
To prove \eqref{textbfc-bd}, we use again the decomposition \eqref{dec:T(bfc)} ($\textbf{c}$ is replaced by $c_{\ep}$):
\begin{align*}
\mathcal{T}(f)(c_{\ep})
=I_1(c_{\ep})+I_2(c_{\ep})+I_3(c_{\ep}).
\end{align*}
By $\f{u'(z)}{|u(z)-c|}=\f{\cosh y_c}{\sinh|z-y_c|\cosh z}$, we have
\begin{align}
\int_{\mathbb{R}}\f{\min\{\al^2|z-y_c|^2,1\}u'(z)}{|u(z)-c|}dz
&\lesssim\notag
\int_{|z-y_c|\leq \al^{-1}}\al^2|z-y_c|dz+
\int_{|z-y_c|\geq \al^{-1}}\f{ \cosh y_c}{\cosh z\sinh |z-y_c|}dz\\
&\leq C+C\al\int_{|z-y_c|\geq \al^{-1}}e^{|y_c|}e^{-|z|}dz\leq C\al.\notag
\end{align}
Consider $f\in H^1_{1/(u')^2}(\mathbb{R})$.
Since \eqref{est:I1(c)-1} also holds for $\textbf{c}\in O_{\ep_0}$, we have $|I_1(c_{\ep})|\leq C(1-c^2)^{-1}$. Since \eqref{est:I2(c)} also holds for $\textbf{c}\in O_{\ep_0}$, we have $|I_2(c_{\ep})|\leq C(1-c^2)^{-\f32}$.
Since \eqref{est:I3(c)} also holds for $\textbf{c}\in O_{\ep_0}$, we have $|I_3(c_{\ep})|\leq C(1-c^2)^{-1}\ln(1-c^2)^{-1}$. Summing up, we have \eqref{textbfc-bd}.
\end{proof}

\section{Explicit expression of the stream function}
 The limiting absorption principle in last section allows us to derive the explicit expression for the stream function $\widehat{\psi}(t,\al,y)$ by starting from the  complex integral \eqref{eq:stream formula'}. In particular, when  $\al=1$,  the spectral density function $\Phi(\al,y,c)$ has
a singularity near the embedding eigenvalue $c=0$. To handle this singularity appropriately, a family of two-parameter contour $\Omega_{\ep,\delta}$ is carefully chosen. 

\subsection{Mode $\al=1$}
The solution of \eqref{eq:Euler-linearize-psi} $\widehat{\psi}(1,y,t)$  has the following expression.

\begin{proposition}\label{Prop:fm}
Let $\al=1$, $\hat{\omega}_0(1,y)\in H^1_{1/(u')^2}(\mathbb{R})$ . Then we have for $y\neq 0$,
\begin{align}
&\widehat{\psi}(1,y,t)
=p.v.\int_{-1}^1e^{-i ct}\mu(c,1)\Gamma(y,c)dc
+\f{1}{2}\omega_0(0)\mathrm{sech}\,y ,\label{fm:psi1}
\end{align}
where $\Gamma$ is defined in \eqref{def:Gamma} and for $\al\geq 1$, $c\in(-1,1)$, $\mu$ is defined as
\begin{align}\label{def:h(c,al)}
\mu(c,\al):=-\f{2 c(1-c^2)^{2}\mathcal{T}(f)(c)-A(c,\al)
f(y_c)}{A(c,\al)^2+4\pi^2c^2},
\end{align}
  with $A$  defined  in \eqref{def:A}, linear operator $\mathcal{T}$  defined in \eqref{def:T(h)-(-1,1)}.

\end{proposition}\label{Pro:yn}

\begin{proof}
Let $\delta\in(0,\ep_0)$, $\ep\in (0,\f{\delta}{\sqrt{2}})$. We define a family of functions on $(-1,1)$,
\ben\label{Phi_pm}
\gamma_{\ep,\delta}(c)=
\left\{
\begin{aligned}
 &C_o^{-1}(1-c^2)\quad\quad\quad\quad\quad\sqrt{1-C_o\ep}\leq |c|<1,\\
&\ep \quad\quad\quad\quad\quad\quad\quad\quad\quad\quad\sqrt{\delta^2-\ep^2}\leq |c|< \sqrt{1-C_o\ep},\\
&\sqrt{\delta^2-c^2}\quad\quad\quad\quad\quad\quad\quad\quad\quad\quad\quad |c|<\sqrt{\delta^2-\ep^2}.
\end{aligned}
\right.
\een
It holds that for $a.e.$ $c\in(-1,1)$,
\begin{align}
\begin{split}
\label{est:gamma'-ep-delta}&|\gamma_{\ep,\delta}'(c)|\leq
\f{C|c|}{\sqrt{\delta^2-c^2}}\textbf{1}_{(-\delta,\delta)}+ C\textbf{1}_{(-1,1)/[-\sqrt{1-C_0\ep}, \sqrt{1-C_0\ep} ]}\in L^1_c(-1,1),\\
&\lim_{\ep\to 0+}\gamma_{\ep,\delta}'(c)= \f{-c}{\sqrt{\delta^2-c^2}}\chi_{[-\delta,\delta]}(c),
\end{split}
\end{align}
where constant $C$ is independent of $\ep$, $\delta$.
\begin{figure}[!h]
\centering
\includegraphics[width=0.9\textwidth]{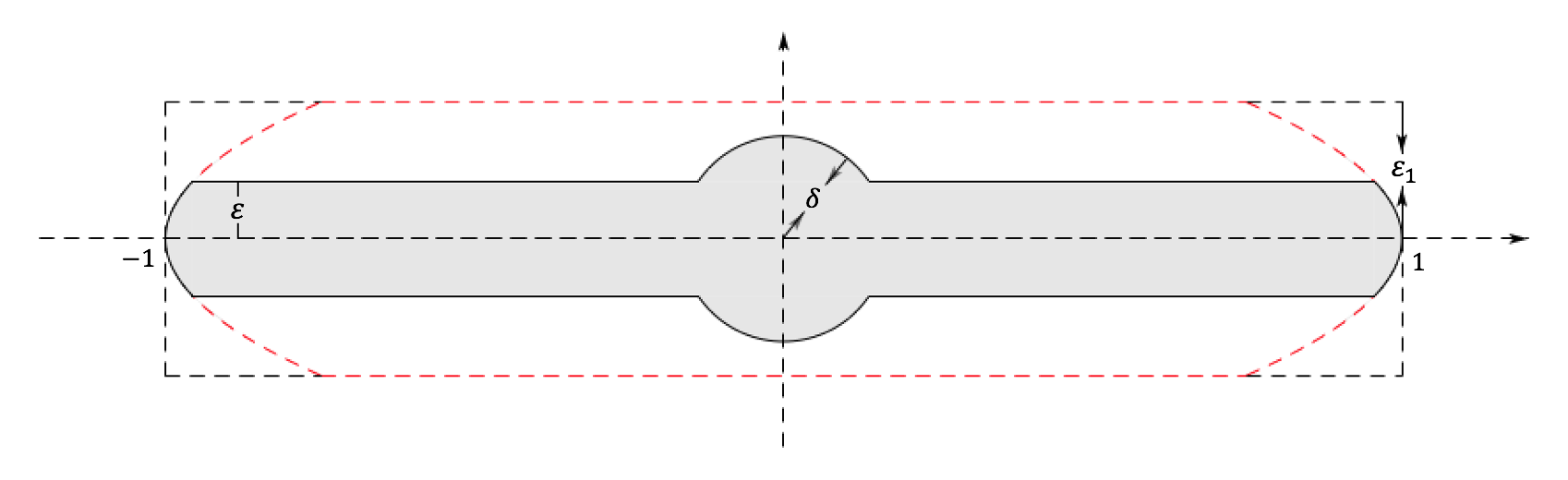}
\caption{$\Omega_{\ep, \delta}$  and \; $\Gamma_{\ep, \delta}^{\pm}=\{c\pm i\gamma_{\ep,\delta}(c)|c\in(-1,1)\}$ }
\label{fig4}
\end{figure}
Therefore, we define a family of parameterized domains and curves on $\mathbb{C}$,
\begin{align}
\begin{split}\label{def:Omega-ep-delta}
\Omega_{\ep,\delta}&:=\{c+i\eta|-\gamma_{\ep,\delta}(c)<\eta<\gamma_{\ep,\delta}(c),\;c\in(-1,1)\},\\
\Gamma_{\ep,\delta}^{\pm}&:=\{c\pm i\gamma_{\ep,\delta}(c)|c\in(-1,1)\}.
\end{split}
\end{align}

Taking $\Omega=\Omega_{\ep,\delta}$ in \eqref{eq:stream formula'}, we have
\begin{align}
\notag\hat{\psi}(t,\al,y)
&=\f{1}{2\pi i}\int_{\pa\Omega_{\ep,\delta}}e^{-i\al ct}\Phi(\al,y,\textbf{c})d\textbf{c}\\
&=\sum_{\tau\in\{+,-\}}\f{1}{2\pi i}\int_{\Gamma_{\ep,\delta}^{\tau}}e^{-i\al \textbf{c}t}\Phi(\al,y,\textbf{c})d\textbf{c}=I^+_{\ep,\delta}+I^-_{\ep,\delta},\label{def:psi}
\end{align}
where $\Phi$ is constructed in Proposition \ref{pro:solve-Phi}, i.e., $\Phi$ satisfies  inhomogeneous Rayleigh equation \eqref{eq:phi-inh} with $f=\hat{\omega}_0(\al,y)$.

We claim that as $\ep\to 0$,  it holds for $\al=1$ that
\begin{align}\label{Phi1:claim}
 &\quad\lim_{\ep\to 0+}I^{\pm}_{\ep,\delta} =I^{\pm}_{\delta}\\
\notag&=\mp\frac{1}{2\pi i}\int_{(-1,1)/(-\delta,\delta)}
e^{-i c t }\Phi_{\pm}(1,y,c)dc+\frac{1}{2\pi i}\int_{\{\pm\mathrm{Im}(\textbf{c})>0,\;|\textbf{c}|=\delta\}}
 e^{-i \textbf{c} t }\Phi(1,y,\textbf{c})d\textbf{c}\\
 &:=I_{\delta,1}^{\pm}+I_{\delta,2}^{\pm},\notag
\end{align}
see the contour after $\ep\to 0$ in Figure 3.
\begin{figure}[!h]
\centering
\includegraphics[width=0.9\textwidth]{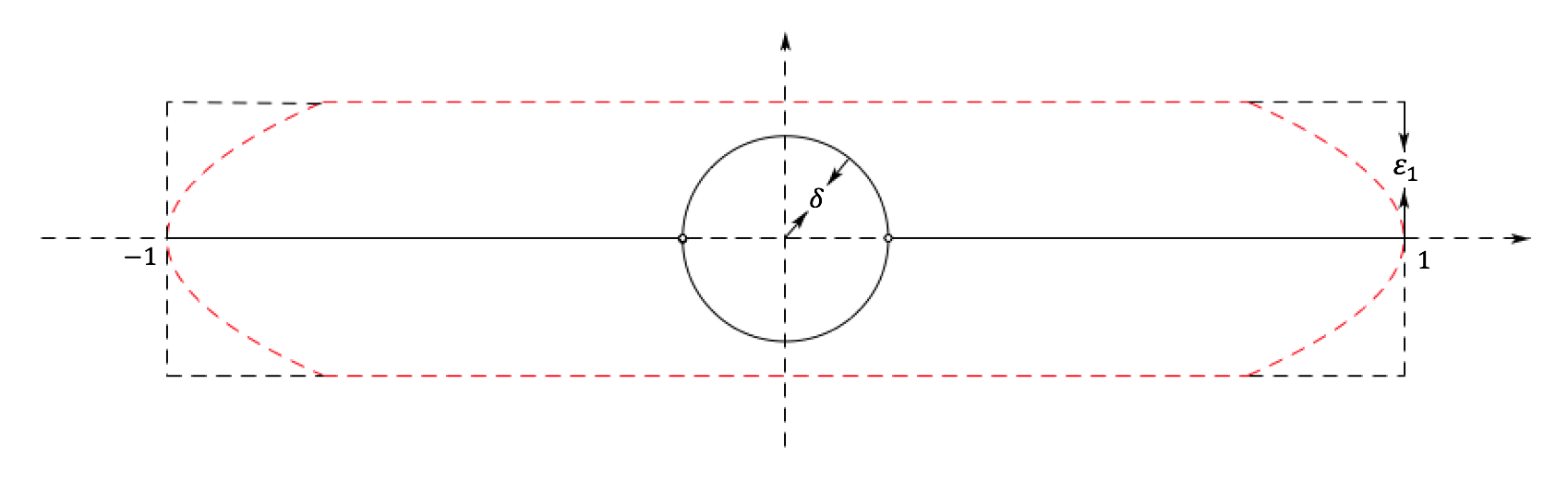}
\caption{After $\ep\to 0$}
\label{fig4}
\end{figure}

 Consequently, we first take $\ep\to 0+$ and then $\delta\to 0+$ in expression \eqref{def:psi}, thanks to \eqref{Phi1:claim} and Corollary \ref{Cor:c=delta},  it holds that
\begin{align*}
\hat{\psi}(t,\al,y)&=\lim_{\delta\to 0+}\lim_{\ep\to 0+}I_{\ep,\delta}^{+}
+\lim_{\delta\to 0+}\lim_{\ep\to 0+}I_{\ep,\delta}^{-}\\
&=\lim_{\delta\to 0+}I_{\delta}^++\lim_{\delta\to 0+}I_{\delta}^-\\
&=\lim_{\delta\to 0+}(I_{\delta,1}^{+}+I_{\delta,1}^{-})+\lim_{\delta\to 0+}(I_{\delta,2}^{+}+I_{\delta,2}^{-})\\
&=\lim_{\delta\to 0+}(I_{\delta,1}^{+}+I_{\delta,1}^{-})
+\f12\omega_0(0)\mathrm{sech}\,y .
\end{align*}
Thus, it follows that $\lim_{\delta\to 0+}(I_{\delta,1}^{+}+I_{\delta,1}^{-})$ must exist as a principle value, i.e.,
\begin{align*}p.v.
\frac{1}{2\pi i}\int_{-1}^1
e^{-i c t }\Big(\Phi_-(1,y,c)-\Phi_+(1,y,c)\Big)dc.
\end{align*}
Then the desired expression follows from definitions \eqref{Phi_pm'} and \eqref{def:mupm}.

It remains to prove the claim \eqref{Phi1:claim} for $\tau=+$, the proof  of the case $\tau=-$ is similar.
Using definition \eqref{def:Omega-ep-delta} and \eqref{Phi_pm}, noticing the integral along $\Gamma_{\ep,\delta}^+$ is counterclockwise, we have
\begin{align*}
I_{\ep,\delta}^+&
=-\frac{1}{2\pi i}\int_{-1}^1\textbf{1}_{J_{\ep,1}}(c)\cdot e^{-i (c+i\gamma_{\ep,\delta}(c)) t }\Phi(1,y,c+i\gamma_{\ep,\delta}(c))(1+i\gamma_{\ep,\delta}'(c))dc\\
&\quad-\frac{1}{2\pi i}\int_{-1}^1\textbf{1}_{J_{\ep,\delta,2}}(c)\cdot
e^{-i c_{\ep} t }\Phi(1,y,c_{\ep})dc+\frac{1}{2\pi i}\int_{\{\mathrm{Im}(\textbf{c})>0,\;|\textbf{c}|=\delta\}}
e^{-i \textbf{c} t }\Phi(1,y,\textbf{c})d\textbf{c}\\
&=-I_{\ep,\delta,1}^+
-I_{\ep,\delta,2}^++I_{\delta,3}^+,
\end{align*}
where $J_{\ep,1}=(-1,-\sqrt{1-C_0\ep})\cup(\sqrt{1-C_0\ep},1)$, $J_{\ep,\delta,2}=(-\sqrt{1-C_0\ep},-\sqrt{\delta^2-\ep^2})
\cup(\sqrt{\delta^2-\ep^2},\sqrt{1-C_0\ep})$.
It suffices to study $I_{\ep,\delta,1}^+$ and $I_{\ep,\delta,2}^+$. We recall the uniform bound \eqref{Uniformbd} in Proposition \ref{lem:limitfunction} that \begin{align*}
|\Phi(\al,y,c_{\ep})|\leq C(1-c^2)^{-\f12}e^{-C^{-1}|y-y_c|}.
\end{align*}
\label{Uniformbd}
Together with \eqref{est:gamma'-ep-delta}, we fix $y\neq 0$ and obtain for $c\in(-1,1)$, $\ep\in (0,\f{\delta}{\sqrt{2}})$ that
\begin{align}
\begin{split}
\label{Unibound:Phi}
\Big|\textbf{1}_{J_{1,\ep}}(c)\cdot\Phi(1,y,c+i\gamma_{\ep,\delta}(c))(1+i\gamma_{\ep,\delta}'(c))\Big|&\leq C(1-c^2)^{-\f12}\in L^1(-1,1),\\
\Big|\textbf{1}_{J_{2,\ep}}(c)\cdot\Phi(1,y,c_{\ep})\Big|&\leq C(1-c^2)^{-\f12}\in L^1(-1,1),
\end{split}
\end{align}
where the constants $C$ is independent of $c,\ep$.
Thus, using \eqref{limitfunction}, \eqref{Unibound:Phi}, we have by the dominated convergence theorem that
\begin{align*}
\lim_{\ep\to 0+}I_{\ep,\delta,1}^+=0,\quad \text{and}\quad
\lim_{\ep\to 0+}I_{\ep,\delta,2}^+=\int_{(-1,1)/(-\delta,\delta)}e^{-ict}\Phi_+(1,y,c)dc.
\end{align*}


\end{proof}

\begin{remark}
The kernel $\mu(c,\al)$ in the expression of the stream function, cf. \eqref{fm:psi1} or \eqref{fm:psi2} (or the kernel $\mathcal{K}(c)$ in Corollary \ref{lem:dual}) plays an important role in the analysis. According to Lemma \ref{lem:eigen-al=1}, we know that the embedding eigenvalue exists if and only if
$\al=1$ and $c=0$.  Furthermore,  we conclude the following:
\begin{itemize}
\item
$A(c)^2\neq 0, \quad \text{if}\;\;\al\geq 2.$

\item
$A(0)=0\;\text{and}\;A(c)\sim c^2,\;\text{as}\;\;c\to 0,\quad \text{if}\;\;\al=1.$

\end{itemize}
In the case when there is no embedding eigenvalue ($\al\geq 2$), the kernel $\mu(c)$(or $\mathcal{K}(c)$) has no singularity for $c\in(-1,1)$. However, when there is an embedding eigenvalue for $\al=1$, the kernel $\mu(c)$(or $\mathcal{K}(c)$) exhibits
a singularity at $c=0$ and it behaves as
\begin{align}\label{rmk:1/c}\mu(c)\sim \f{1}{c},\quad c\to 0.
\end{align}
\end{remark}

 \vspace{\baselineskip}
To study the singularity at $c=0$ in more details, let us introduce some new quantities.
\begin{align}
\label{def:tT-1}\widetilde{\mathcal{T}}(f)(c)&:=\f{\mathcal{T}(f)(c)-\mathcal{T}(f)(0)}{c}
=\int_0^1(\pa_c\mathcal{T})(f)(sc)ds,\quad c\in(-1,1).
\end{align}
 Since $A(0)=\pa_cA(0)=0$ according to Lemma \ref{lem:A}, we can define
\begin{align}
\begin{split}
\label{def:tA}\tilde{A}(c)&:=A(c)/c=\int_0^1\pa_cA(sc)ds,\quad \text{with}\quad \tilde{A}(0)=0,\\
\tilde{\tilde{A}}(c)&:=A(c)/c^2=\int_0^1\int_{0}^1s\pa_c^2A(stc)dsdt.
\end{split}
\end{align}
We define the linear operators
\begin{align}\label{def:calLambda1}
\Lambda_1(f)(c)&:=2c\mathcal{T}(f)(c)-(1-c^2)^{-2}Af(y_c).
\end{align}
\begin{align}
\begin{split}
\label{recall:Lambda}\tilde{\Lambda}_1(f)(c)&:
=2\widetilde{\mathcal{T}}(f)(c)-(1-c^2)^{-2}\tilde{\tilde{A}}(c)f(y_c),
\end{split}
\end{align}
We  introduce a smooth even function $\chi_{0}(c)$

\beq\label{def:chi0}
\chi_{0}(c)=
\left\{
\begin{aligned}
&1,\;(-\f14,\f14),\\
&\text{smooth},\;(-\f12,\f12)/(-\f14,\f14),\\
&0,\;(-1,1)/(-\f12,\f12).
\end{aligned}
\right.
\eeq
 Denote $\chi_{1}(c)= 1-\chi_{0}(c)$, which takes $0$ on $(-\f14,\f14)$ and $1$ on $(-1,1)/(-\f12,\f12)$. 
\begin{lemma}\label{lem:dec-psi}
Let $\hat{\psi}(t,1,y)$ be the explicit expression in Proposition \ref{Prop:fm}. It holds that
\begin{align}
\notag\hat{\psi}(t,1,y)&=-\int_{-1}^1e^{-ict}
\f{\chi_1(c)\Lambda_1(\omega_0)(c)}{A(c)^2+4\pi^2c^2}\Gamma(y,c)dc-\int_{-1}^{1}e^{-ict}
\f{\chi_0(c)(1-c^2)^2\tilde{\Lambda}_1(c)}{\tilde{A}(c)^2+4\pi^2}\Gamma(y,c)dc\\
\label{decom:psi}&\quad+\f{a_0}{2\pi^2}\int_{-1}^{1}e^{-ict}
\f{\chi_0(c)(1-c^2)^2c\tilde{\tilde{A}}(c)^2}{\tilde{A}(c)^2+4\pi^2}\Gamma(y,c)dc\\
&\quad+a_0f_1(t,y)+\f{a_0+i\pi b_0}{2\pi i}\mathrm{sech}\,y ,\notag
\end{align}
where $a_0$ is defined in \eqref{def:a0} and

\begin{align}\label{def:Psi1}
f_1(t,y):=\f{1}{2\pi^2}\Big(\Psi(t,y)+\big(S(t)+i\pi\big)\mathrm{sech}\,y \Big),
\end{align}
with
\begin{align}\label{def:Psi11}
\Psi(t,y)&:= -p.v.\int_{-1}^1e^{-ict}
\f{\chi_0(c)(1-c^2)^2}{c}\Big(\Gamma(y,c)-\Gamma(y,0)\Big)dc,\\
\label{def:S(t)}
S(t)&:=p.v.\int_{-1}^1e^{-ict}
\f{\chi_0(c)(1-c^2)^2}{c}dc.
\end{align}

\end{lemma}
\begin{proof}
We decompose $\mu$, cf. \eqref{def:h(c,al)} that
\begin{align*}
\mu(c)&=\f{\chi_1(c)2c(1-c^2)^2\mathcal{T}(c)-A\omega_0(y_c)} {A^2+4\pi^2c^2}
+\f{\chi_0(c)2c(1-c^2)^2\Big(\mathcal{T}(c)-
\mathcal{T}(0)\Big)-A\omega_0(y_c)} {A^2+4\pi^2c^2}\\
&\quad+\mathcal{T}(0)\chi_0(c)(1-c^2)^2\Big(\f{2c}{A^2+4\pi^2c^2}
-\f{1}{2\pi^2c}\Big)+\f{\mathcal{T}(0)\chi_0(c)(1-c^2)^2}{2\pi^2c}.
\end{align*}
Furthermore, we claim that
\begin{align}
\label{T(0)}
\mathcal{T}(\omega_0)(0)&=p.v.\int_{\mathbb{R}}\f{\hat{\omega}_0(1,y)}{\sinh y}dy=a_0.
\end{align}
Thus, the first three terms in \eqref{decom:psi} can be obtained from the definitions \eqref{def:tT-1}-\eqref{def:chi0}.
Thanks to Lemma \ref{lem:eigen-al=1}, the last two terms in \eqref{decom:psi} follow from
\begin{align*}
&\f{\mathcal{T}(0)\chi_0(c)(1-c^2)^2}{2\pi^2c}\Gamma(y,c)\\
&=a_0\f{\chi_0(c)(1-c^2)^2}{2\pi^2c}\Big(\Gamma(y,c)-
\Gamma(y,0)\Big)-a_0\f{\chi_0(c)(1-c^2)^2}{2\pi^2c}\mathrm{sech}\,y \\
&=a_0\f{\chi_0(c)(1-c^2)^2}{2\pi^2c}\Big(\Gamma(y,c)-
\Gamma(y,0)\Big)-a_0\Big(\f{\chi_0(c)(1-c^2)^2}{2\pi^2c}
+\f{i\pi}{2\pi^2}\Big)\mathrm{sech}\,y -\f{a_0}{2\pi i}\mathrm{sech}\,y .
\end{align*}
It remains to prove \eqref{T(0)}. We use the definition of $\Gamma$, cf. \eqref{def:Gamma} and Lemma \ref{lem:eigen-al=1} to obtain for $f\in H^1(\mathbb{R})$,
\begin{align}\notag
\mathcal{T}(f)(0)&=p.v.\int_{\mathbb{R}}\f{\int_{0}^y\phi_1(z,0)f(z)dz}{\phi(y,0)^2}dy
=p.v.\int_{\mathbb{R}}\left(\f{\Gamma(1,y,0)}{\phi(y,0)}\right)'\cdot \Big(\int_{0}^y\phi_1(z,0)f(z)dz\Big))dy\\
&=-p.v.\int_{\mathbb{R}}\Gamma(1,y,0)\cdot\f{f(y)}{u(y)}dy
=p.v.\int_{\mathbb{R}}\mathrm{sech}\,y \cdot\f{f(y)}{\tanh y}dy
=p.v.\int_{\mathbb{R}}\f{f(y)}{\sinh y}dy,\label{fm:T(0)}
\end{align}
which gives the desired identity.
\end{proof}
\begin{remark}\label{rmk:K}
 The lemma of decomposition is very important in studying the dynamics in the presence of embedding value $c=0$. The first three terms in \eqref{decom:psi} are regular. The emergence of the fourth term $a_0f_1$ is due to the $\f{1}{c}$ singularity caused by the embedding value. The fifth term falls in the eigenspace.
\end{remark}

 \subsection{Modes $\al\geq 2$}

 \begin{proposition}\label{Prop:fm'}
Let $\al\geq 2$, $\hat{\omega}_0(\al,\cdot)\in H^1_{1/(u')^2}(\mathbb{R})$.
 It holds that for $y\neq y_c$,
\begin{align}\label{fm:psi2}
&\widehat{\psi}(\al,y,t)=\int_{-1}^1e^{-i\al ct}\mu(c,\al)\Gamma(y,c)dc,
\end{align}
where $\Gamma$ is defined in \eqref{def:Gamma} and $\mu(c,\al)$ is defined in \eqref{def:h(c,al)}.
\end{proposition}

\begin{proof}
The proof of this proposition is simpler compared to the proof of Proposition \ref{Prop:fm}. We will provide a  sketch.

Let $\ep\in(0,\ep_0)$. We  define
\ben\label{Phi_pm}
\notag\gamma_{\ep}(c)=
\left\{
\begin{aligned}
\notag &C_o^{-1}(1-c^2),\quad\sqrt{1-C_o\ep}\leq |c|<1,\\
\notag&\ep, \quad\quad\quad\quad\quad\quad\quad\quad|c|\leq \sqrt{1-C_o\ep},
\end{aligned}
\right.
\een
It holds that  for $a.e.$ $c\in(-1,1)$,
\begin{align*}
\begin{split}
&|\gamma_{\ep}'(c)|\leq C\textbf{1}_{(-1,1)/[-\sqrt{1-C_0\ep}, \sqrt{1-C_0\ep} ]}\in L^1_c(-1,1),\quad\text{and}\quad\lim_{\ep\to 0+}\gamma_{\ep}'(c)= 0,
\end{split}
\end{align*}
where the constant $C$ is independent of $\ep$. We also define
\begin{align*}
\Gamma_{\ep}^{\pm}&:=\{c\pm i\gamma_{\ep}(c)|c\in(-1,1)\}.
\end{align*}
Taking $\Omega=\tilde{O}_{\ep}$ in \eqref{eq:stream formula'}, we have for $\al\geq 2$,
\begin{align*}
\notag\hat{\psi}(t,\al,y)
&=\f{1}{2\pi i}\int_{\pa\Omega_{\ep}}e^{-i\al \textbf{c}t}\Phi(\al,y,\textbf{c})d\textbf{c}=\sum_{\tau\in\{+,-\}}\f{1}{2\pi i}\int_{\Gamma_{\ep}^{\tau}}e^{-i\al \textbf{c}t}\Phi(\al,y,\textbf{c})d\textbf{c}=I^+_{\ep}+I^-_{\ep}.
\end{align*}
 Using \eqref{limitfunction}, \eqref{Uniformbd-2},  we obtain by the dominated convergence theorem that
\begin{align*}
 &\quad\lim_{\ep\to 0+}I^{\pm}_{\ep}=\mp\int_{-1}^1
e^{-i \al c t }\Phi_{\pm}(\al,y,c)dc\notag.
\end{align*}
Here one can refer to  similar process dealing with  $I_{\ep,\delta,1}^+$ and $I_{\ep,\delta,2}^+$, in the proof of Proposition \ref{Prop:fm}.
Finally, the desired expression follows from definitions \eqref{Phi_pm'} and \eqref{def:mupm}.
\end{proof}

\subsection{Dual argument }
\begin{lemma}\label{lem:dual0}
 Let $\Gamma$ be defined in \eqref{def:h(c,al)}, $f(c)$ be a function such that
 \begin{align*}
 p.v.\int_{-1}^1f(c)\Gamma(y,c)dc\quad \text{is}\; \text{well}\; \text{defined}\; \text{for}\;\;a.e. \; y\in\mathbb{R}.
  \end{align*}
 Let $g\in H^2(\mathbb{R})$, $\varphi=g''-\al^2g$. Then it holds that
\begin{align}\label{dual-fm}
&\quad \int_{\mathbb{R}}\Big(p.v.\int_{-1}^1f(c)\Gamma(y,c)dc\Big)\varphi(y)dy=-p.v.\int_{-1}^1f(c)\Lambda_2(g)(c)dc,
\end{align}
where the linear operator $\Lambda_2$ is defined as
\begin{align}\label{def:calLambda2}
\Lambda_2(g)(c):=\mathcal{T}(u''g)(c)+(1-c^2)^{-1}A(c)g(y_c),
\end{align}
 with $A$  defined in \eqref{def:A}, and $\mathcal{T}$  defined in \eqref{def:T(h)-(-1,1)}.

 Moreover, it holds that
 \begin{align}\label{dual-fm-g}
&\quad \int_{\mathbb{R}}
\Big(p.v.\int_{-1}^1f(c)\Big(\Gamma(y,c)
-\Gamma(y,0)\Big)dc\Big)\varphi(y)dy=-p.v.\int_{-1}^1cf(c)
\tilde{\Lambda}_2(g)(c)dc,
\end{align}
where the linear operator $\tilde{\Lambda}_2$ is defined as
\begin{align}
\label{recall:tLambda2}\tilde{\Lambda}_2(g)(c)&:
=\widetilde{\mathcal{T}}(u''g)(c)+(1-c^2)^{-1}
\tilde{A}(c)g(y_c),
\end{align}
with $\mathcal{T}$ defined in \eqref{def:tT-1}, and $\mathcal{A}$ defined in \eqref{def:tA}.
\end{lemma}

\begin{proof}
The proof of \eqref{dual-fm} is standard.
We exchange the order of integration 
to get
\begin{align}
 \notag&L.H.S\\
  \notag&=\int_{\mathbb{R}}\Big(p.v.\int_{u(y)}^1\int_{-\infty}^y
Fdzdc\Big)dy+
\int_{\mathbb{R}}\Big(p.v.\int_{-1}^{u(y)}\int_{+\infty}^y
Fdzdc\Big)dy\\
 \notag&=p.v.\int_{-1}^1dc \int_{-\infty}^{y_c}\int_{-\infty}^yFdzdy
-p.v.\int_{-1}^1dc \int_{y_c}^{+\infty}\int^{+\infty}_yFdzdy\\
 \notag&=p.v.\int_{-1}^1dc \int_{-\infty}^{y_c}\int_{z}^{y_c}Fdydz
-p.v.\int_{-1}^1dc \int^{+\infty}_{y_c}\int^{z}_{y_c}Fdydz\\
\notag&=-p.v.\int_{-1}^1f(c,t)\Big(\int_{\mathbb{R}}\f{\int_{y_c}^z\phi(y,c)\varphi(y)dy}{\phi(z,c)^2}dz\Big)dc,
\end{align}
where $F=\f{f(c,t)\varphi(y)\phi(y,c)}{\phi(z,c)^2}$.
Thus, for  $g\in H^2(\mathbb{R})$ and $\varphi=g''-\al^2g$, using \begin{align*}
\phi \varphi=(\phi g')'-(\phi' g)'+(\phi''-\al^2\phi)g,\; \phi''-\al^2\phi=\f{u''\phi}{u-c}, \; \f{g(\cdot)}{\phi(\cdot,c)}-\f{g(y_c)}{u(\cdot)-c}\in C(\mathbb{R}),
\end{align*}
definition \eqref{fm:W1} and \eqref{def:W-4}, we have
\begin{align*}
&\int_{\mathbb{R}}\f{\int_{y_c}^z\phi(y,c)\varphi(y)dy}{\phi(z,c)^2}dz=p.v.\int_{\mathbb{R}}\f{\int_{y_c}^z\f{u''\phi g}{u-c}dy}{\phi(z,c)^2}dz
+p.v.\int_{\mathbb{R}}\f{\phi g'-\phi' g+u'(y_c)g(y_c)}{\phi(z,c)^2}dz\\
&=p.v.\int_{\mathbb{R}}\f{\int_{y_c}^z\phi_1u'' gdy}{\phi(z,c)^2}dz
+\int_{\mathbb{R}}\big(\f{g(z)}{\phi(z,c)}-\f{g(y_c)}{u(z)-c}\big)'dz
+g(y_c)\int_{\mathbb{R}}\f{u'(z)}{(u(z)-c)^2}
\big(\f{1}{\phi_1^2(z,c)}-1\big)dz\\
&\quad+g(y_c)p.v.\int_{\mathbb{R}}\f{u'(y_c)-u'(z)}{\phi(z,c)^2}dz\\
&=\mathcal{T}(u''g)(c)+g(y_c)\Big(\f{-2}{1-c^2}+\int_{\mathbb{R}}\f{u'(z)}{(u(z)-c)^2}
\big(\f{1}{\phi_1^2(z,c)}-1\big)dz+p.v.\int_{\mathbb{R}}\f{u'(y_c)-u'(z)}{\phi(z,c)^2}dz\Big)\\
&=\mathcal{T}(u''g)(c)+\f{g(y_c)}{1-c^2}A(c,\al)=\Lambda_2(g)(c).
\end{align*}

To prove \eqref{dual-fm-g}, we use \eqref{fm:T(0)}, $u''=\f{-2\sinh y}{\cosh^3 y}$ and integration by parts to get
\begin{align*}
\mathcal{T}(u''g)(0)=p.v.\int_{\mathbb{R}}\f{u''(y)g(y)}{\sinh y}dy=\int_{\mathbb{R}}\f{-2g(y)}{\cosh^3 y}dy=
\int_{\mathbb{R}}\f{\varphi(y)}{\cosh y}dy.
\end{align*}
Combining  \eqref{dual-fm} and the definition of $\tilde{\Lambda}_2$,  we have the desired identity.
\end{proof}

\begin{Corollary}\label{lem:dual}
Let $f\in L^2(\mathbb{R})$, $g\in H^2(\mathbb{R})$, $\varphi=g''-\al^2g$. Let $\Gamma$ be defined in \eqref{def:h(c,al)}, $\mu$  be  defined in \eqref{def:h(c,al)}. It holds that for  $\al\geq 1$
\begin{align*}
&\quad \int_{\mathbb{R}}\Big(p.v.\int_{-1}^1e^{-i\al ct}\mu(c,\al)\Gamma(\al,y,c)dc\Big)\varphi(y)dy=p.v.\int_{-1}^1e^{-i \al ct}\mathcal{K}(f,g)(c,\al)dc,
\end{align*}
where the bilinear operator $\mathcal{K}(\cdot,\cdot)$ is defined as
\begin{align}\label{def:calK}
\mathcal{K}(f,g)(c,\al)&=\f{(1-c^2)^2\Lambda_1(f) (c)\Lambda_2(g)(c)}{A(c,\al)^2+4\pi^2c^2},\quad c\in (-1,1),\; \al\geq 1.
\end{align}
Here the linear operators $\Lambda_1$, $\Lambda_2$ are defined in \eqref{def:calLambda1} and \eqref{def:calLambda2}, $A$ is defined  in \eqref{def:A}, and $\mathcal{T}$ is defined in \eqref{def:T(h)-(-1,1)}.
\end{Corollary}

Thanks to Lemma \ref{lem:A}, we know that  the  principle value in the identity is a proper integral when $\al \geq 2$.
\begin{proof}
 Let $\mu$ be defined in \eqref{def:h(c,al)}, the identity follows from taking $f(c)=e^{-i\al ct}\mu(c,\al)$ in Lemma \ref{lem:dual0}.
\end{proof}

\section{
Weighted estimates for  the integral operators}
The goal of this section is to prove the boundedness of linear and bilinear operators that appear in Section 6 and Section 7. The estimates are weighted  to fit the degeneracy of $u'(y)$ and serve for the proof of the linear inviscid damping in the main theorem.

Let $I$ be an interval, $1\leq p\leq \infty$, $m(c,\al)$ be some simple function. We use the notation
 $f(c)=m(c,\al)\mathcal{L}^p(I)$ to represent a function $f(c)$ with $\|f/m\|_{L^p_c(I)}\leq C$, where $C$ is a universal constant. We write $f=m\mathcal{L}^p$ for simplicity, if $I=(-1,1)$. Sometimes, we omit the  $\al$ dependence of the operators.

\subsection{Estimates for  $\mathcal{T}(c)$. }
Recall the definition of $\mathcal{T}$, cf. \eqref{def:T(h)-(-1,1)},
\begin{align*}
\mathcal{T}(f)(c)=p.v.\int_{\mathbb{R}}\f{\int_{y_c}^y\phi_1(z,c)f(z)dz}{(u(y)-c)^2\phi_1(y,c)^2}dy,\quad \;c\in(-1,1). \end{align*}
We have the estimates of weighted Sobolev regularity for $\mathcal{T}$ as follows.
\begin{lemma}\label{lem:key-T}
Let $\al\geq 1$, $i=0,1,2$, $k=0,1,2,3$. If $f\in H^k_{1/(u')^{2i}}(\mathbb{R})$, then it holds that
\begin{align}\label{est:key-1/u'^2}
\pa_c^k\mathcal{T}(f)(c)=\al(1-c^2)^{-\f32-k+i}\mathcal{L}^2.
\end{align}

\end{lemma}

Before we start proving, we provide several basic estimates. Let $k\in \mathbb{N}$. It directly holds that
\begin{align}\label{lem:weighted-norm-derivative}
&\f{d^k}{dc^k}(f(y_c))=(1-c^2)^{-k}f^{(k)}(y_c),\quad
\big\|(1-c^2)^{-\f12}f(y_{c}) \big\|_{L^2_c(-1,1)}=\|f(y)\|_{L^2(\mathbb{R})}.
\end{align}

The maximal operator $\mathcal{M}$, Hilbert transform $\mathcal{H}$ and average operator $\mathcal{A}_{\tau}$ are  defined as follows
\begin{align*}
\mathcal{M}(f)(y)&:=\sup_{z\in\mathbb{R}}\f{\int_z^y|f(\tilde{y})|d\tilde{y}}{y-z},\;y\in\mathbb{R},\\
\mathcal{H}(f)(c)&:=\int_{-1}^1\f{f(\tilde{c})}{c-\tilde{c}}d\tilde{c},\;c\in(-1,1),\\
\mathcal{A}_{\tau}(f)(c)&:=\f{\int_{\tau}^cf(\tilde{c})d\tilde{c}}{c-\tau},\;c\in(-1,1),\;\tau\in[-1,1].
\end{align*}
Let $c\in(-1,1)$, $y_c=u^{-1}(c)$ and  $f\in L^2(\mathbb{R})$. By the classical theory of Harmonic analysis, $u'(y_c)=1-c^2$  and  \eqref{lem:weighted-norm-derivative}, we directly obtain the following
\begin{align}\label{est:maximalyc}
&\| (1-c^2)^{-\f12}\mathcal{M}(f)(y_c)\|_{L^2_c(-1,1)}\leq C
\|f\|_{L^2(\mathbb{R})},
\end{align}
\begin{align}
\begin{split}
\label{est:Hilbert}
&\Big\|\mathcal{H}\left((f/(u')^{\f12})\circ u^{-1}\right)(c)\Big\|_{L^2_c(-1,1)}\leq C\|f\|_{L^2(\mathbb{R})},\\ &\Big\|\mathcal{A}_{\tau}\left((f/u'^{\f12})\circ u^{-1}\right)(c)
\Big\|_{L^2_c(-1,1)}\leq C\|f\|_{L^2(\mathbb{R})}.
\end{split}
\end{align}
The estimates for the following  family of operators generalize the estimates in the proof of Lemma \ref{lem:key-T}.
\begin{lemma}\label{lemma:key-integral-estimate-1}
Let $\al\geq 1$, $m\in[0,1]$, $l\geq 0$.  We define
\begin{align}
\label{def:Tl}\mathcal{T}_{m,l}(f)(c)&:=\int_{\mathbb{R}}\f{\langle\al|y-y_c|\rangle^{l}\big|\int_{y_c}^y|f(z)|dz\big|}{|u(y)-c|^m\phi_1(y,c)}dy.
\end{align}
Then  there exists constant $C$ independent of $\al$ such that
\begin{align*}
\|(1-c^2)^{m-\f12}\mathcal{T}_{m,l}(f)(c)\|_{L^2_c(-1,1)}\leq C\al^{m-2}\|f\|_{L^2(\mathbb{R})}.
\end{align*}
\end{lemma}
\begin{proof}
Thanks to \eqref{u-key} and $u'(y_c)=1-c^2$, we have
\begin{align*}
\f{1}{|u(y)-c|^m}\leq \f{C\al^m(1-c^2)^{-m} }{\tanh^m \al|y-y_c|}
\leq \f{C\al^m(1-c^2)^{-m} }{\min\{ (\al|y-y_c|)^m,1\}} .
\end{align*}
This together with \eqref{est:phi_1-ul} give for $m\in[0,1]$, $l\geq 0$, $c\in(-1,1)$,
\begin{align}
&\notag\quad(1-c^2)^{m-\f12}\mathcal{T}_{m,l}(f)(c)\\
\notag&\leq C\al^{m-1}(1-c^2)^{-\f12}
\int_{\mathbb{R}}\f{\int_{y_c}^y|f(z)|dz}{y-y_c}\cdot \f{\al|y-y_c|}{\min\{(\al|y-y_c|)^m,1\}}\cdot \f{\langle\al|y-y_c|\rangle^l}{e^{C^{-1}\al|y-y_c|}}dy\\
\notag&\leq C\al^{m-1}(1-c^2)^{-\f12}
\int_{\mathbb{R}}\f{\int_{y_c}^y|f(z)|dz}{y-y_c}\cdot \f{\langle\al|y-y_c|\rangle^{l+1}}{e^{C^{-1}\al|y-y_c|}}dy\\
\notag&\leq C \al^{m-2}
(1-c^2)^{-\f12}\mathcal{M}(|f|)(y_c)\int_{\mathbb{R}}\langle \tilde{y}\rangle^{l+1}e^{-C^{-1}|\tilde{y}|}d\tilde{y}\\
\notag&\leq C \al^{m-2}
(1-c^2)^{-\f12}\mathcal{M}(|f|)(y_c).
\end{align}
Thus, the desired inequality follows from \eqref{est:maximalyc}.
\end{proof}
The following operator is constructed to ensure the optimal weighted estimate for $\mathcal{T}$.
\begin{lemma}\label{lemma:key-integral-estimate}
Let $\al\geq 1$. We define
\begin{align}
\label{def:L1.5}\mathcal{L}_{\f32}(f)(c):&=\int_{\mathbb{R}}\int_{y_c}^y\f{f(z)}{|u(y)-c|^2}\Big(
\f{\phi_1(z,c)}{\phi_1(y,c)^2}-\f{u'(y)u'(z)^{\f12}}{(1-c^2)^{\f32}}\Big)dzdy,\quad c\in(-1,1).
\end{align}
Then  there exists a constant $C$ independent of $\al$ such that
\begin{align}
\label{T3}&\|(1-c^2)^{\f32}\mathcal{L}_{\f32}(f)(c)\|_{L^2(-1,1)}\leq C\ln\al\|f\|_{L^2(\mathbb{R})}.
\end{align}
\end{lemma}

\begin{proof}
For simplicity of notation, we denote
\begin{align*}
P(y,z,c):=\f{(1-c^2)^{\f32}}{|u(y)-c|^2}\cdot\Big|\f{\phi_1(z,c)}
{\phi_1(y,c)^2}-\f{u'(y)u'(z)^{\f12}}{(1-c^2)^{\f32}}\Big|,\quad \text{for}\; y<z< y_c\;\text{or}\;y_c< z<y.
\end{align*}
We first give some basic estimates. Thanks to  $(1-c^2)-u'=(u-c)(u+c)$, we have
\begin{align*}
|(1-c^2)^{\f12}-(u')^{\f12}|&=\Big|\f{(u-c)(u+c)}{(1-c^2)^{\f12}
+u'^{\f12}}\Big|\leq 2(1-c^2)^{-\f12}|u-c|.
\end{align*}
Using \eqref{u(z)-c<u(y)-c}, we have
\begin{align}
\notag|(1-c^2)^{\f32}-u'(y)u'(z)^{\f12}|&\leq (1-c^2)|(1-c^2)^{\f12}-u'(z)^{\f12}|
+(1-c^2)^{\f12}|u'(z)-u'(y)|\\
&\leq C(1-c^2)^{\f12}|u(y)-c|\label{u'(u')12-(1-c^2)32}.
\end{align}
By \eqref{fm:u(y)-c}, \eqref{coshy1} and \eqref{sinhy}, we have
\begin{align}
\f{u'(y)u'(z)^{\f12}}{|u(y)-c|^2}
&=\f{\cosh^2 y\cosh^2y_c}{\sinh^2(y-y_c)\cosh^2y\cosh z}\leq \f{ C\cosh (z-y_c)\cosh z\cdot \cosh y_c}
{\sinh^2(y-y_c)\cosh z}\notag\\
&\leq \f{C(1-c^2)^{-\f12}}{e^{|y-y_c|}\min\{|y-y_c|^2,1\} }.
\label{u'(u')12/(u-c)^2}
\end{align}
By  \eqref{phi_1(z)leq phi_1(y)}, \eqref{phi_1-1leq}, \eqref{u'(u')12-(1-c^2)32} and \eqref{u'(u')12/(u-c)^2}, we have
\begin{align}
&\notag\quad
P(y,z,c)&\\
&\notag=\f{(1-c^2)^{\f32}}{|u(y)-c|^2}\cdot\Big|\f{\phi_1(z,c)}{\phi_1(y,c)^2}-\f{\phi_1(z,c)}{\phi_1(y,c)^2}\cdot \f{u'(y)u'(z)^{\f12}}{(1-c^2)^{\f32}}
+\f{\phi_1(z,c)}{\phi_1(y,c)^2}\cdot \f{u'(y)u'(z)^{\f12}}{(1-c^2)^{\f32}}
-\f{u'(y)u'(z)^{\f12}}{(1-c^2)^{\f32}}\Big|\\
\notag&\leq \f{|(1-c^2)^{\f32}-u'(y)u'(z)^{\f12}|}{|u(y)-c|^2}\cdot\f{\phi_1(z,c)}{\phi_1(y,c)^2}
+\f{u'(y)u'(z)^{\f12}}{|u(y)-c|^2}\cdot\Big(\Big|\f{\phi_1(z,c)-1}{\phi_1(y,c)^2}\Big|+\Big|\f{1-\phi_1(y,c)^2}{\phi_1(y,c)^2}\Big|\Big)\\
\notag&\lesssim  \f{(1-c^2)^{\f12}}{|u(y)-c|\phi_1(y,c)}+ \f{(1-c^2)^{-\f12}\min\{\al^2|y-y_c|^2,1\}}{\min\{|y-y_c|^2,1\} e^{|y-y_c|}}.
\end{align}
Now we are in a position to prove \eqref{T3}. Using the estimate of $P$, we have
 \begin{align*}
|(1-c^2)^{\f32}\mathcal{L}_{\f32}(c)|
&\leq
 \int_{\mathbb{R}}\f{(1-c^2)^{\f12}\big|\int_{y_c}^y|f(z)|dz\big|}{|u(y)-c|\phi_1(y,c)}dy\\
 &\quad+\int_{\mathbb{R}}\f{\int_{y_c}^y|f(z)|dz}{y-y_c}\cdot \f{(1-c^2)^{-\f12}\min\{\al^2|y-y_c|^2,1\}}{\min\{|y-y_c|^2,1\} e^{|y-y_c|}}\cdot |y-y_c|dy.
 \end{align*}
Therefore, using \eqref{est:phi_1-ul},  \eqref{est:maximalyc}, Lemma \ref{lemma:key-integral-estimate-1} ($m=1$, $l=0$),
  we have for $f\in L^2(\mathbb{R})$,
\begin{align*}
&|(1-c^2)^{\f32}\mathcal{L}_{\f32}(f)(c)|\\
&\leq C(1-c^2)^{\f12}\mathcal{T}_{1,0}(f)(c)
+C (1-c^2)^{-\f12}\mathcal{M}(f)(y_c)\int_{\mathbb{R}}
\f{\min\{\al^2|y-y_c|^2,1\}|y-y_c|}{\min\{|y-y_c|^2,1\} e^{|y-y_c|}}dy
\\
&\leq C(1-c^2)^{\f12}\mathcal{T}_{1,0}(f)(c)
+C (1-c^2)^{-\f12}\mathcal{M}(f)(y_c)\cdot \ln \al.
\end{align*}
Thus, the desired inequality follows from Lemma \ref{lemma:key-integral-estimate-1} and \eqref{est:maximalyc}.
\end{proof}

In what follows, we prove Lemma \ref{lem:key-T} by three steps.\smallskip




\textbf{Step 1}. $i=k=0$. It holds that
\begin{align}\label{lem:T-weighted-L2-key}
\|(1-c^2)^{\f32}\mathcal{T}(f)(c)\|_{L^2(-1,1)}\leq C\al\|f\|_{L^2(\mathbb{R})}.
\end{align}

\begin{proof}[Proof of Step 1.]
Direct calculation gives for $c\in(-1,1)$
\begin{align}
&\notag p.v.\int_{\mathbb{R}}\f{u'(y)\int_{y_c}^yF(z)dz}{(u(y)-c)^2}dy\overset{\tilde{c}=u(y)}{=}p.v.\int_{-1}^1\f{\int_{u^{-1}(c)}^{u^{-1}(\tilde{c})}F(z)dz}{(\tilde{c}-c)^2}d\tilde{c}\\
&\notag=p.v.\int_{-1}^1\f{\int_{c}^{\tilde{c}}(F/u')\circ u^{-1}(\tilde{\tilde{c}})d\tilde{\tilde{c}}}{(\tilde{c}-c)^2}d\tilde{c}\\
\notag&=-p.v.\int_{-1}^1\pa_{\tilde{c}}(\f{1}{\tilde{c}-c})\Big(\int_{c}^{\tilde{c}}(F/u')\circ u^{-1}(\tilde{\tilde{c}})d\tilde{\tilde{c}}\Big)d\tilde{c}\\
\notag&=p.v.\int_{-1}^1\f{(F/u')\circ u^{-1}(\tilde{c})}{\tilde{c}-c}d\tilde{c}-\f{\int_{c}^{\tilde{c}}(F/u')\circ u^{-1}(\tilde{\tilde{c}})d\tilde{\tilde{c}}}{\tilde{c}-c}\big|_{\tilde{c}=-1}^1\\
&=-\mathcal{H}\left((F/u')\circ u^{-1}\right)(c)-\mathcal{A}_{1}\left((F/u')\circ u^{-1}\right)(c)+\mathcal{A}_{-1}\left((F/u')\circ u^{-1}\right)(c).\label{fm:Hilbert}
\end{align}
Using  \eqref{fm:Hilbert}(with $F=fu'^{\f12}$) and \eqref{def:L1.5}, we are able to  decompose $\mathcal{T}$, cf. \eqref{def:T(h)-(-1,1)} as
\begin{align}
&\notag(1-c^2)^{\f32}\mathcal{T}(f)(c)\\
\notag&=\int_{\mathbb{R}}\f{u'(y)\int_{y_c}^y\big(fu'^{\f12}\big)(z)dz}{(u(y)-c)^2}dy
+(1-c^2)^{\f32}\int_{\mathbb{R}}\int_{y_c}^y\f{f(z)}{(u(y)-c)^2}\Big(\f{\phi_1(z,c)}{\phi_1(y,c)^2}-\f{u'(y)u'(z)^{\f12}}{(1-c^2)^{\f32}}\Big)dzdy\\
\notag 
&=-\mathcal{H}\Big((f/u'^{\f12})\circ u^{-1}\Big)(c)-\mathcal{A}_1\Big((f/u'^{\f12})\circ u^{-1}\Big)(c)+\mathcal{A}_{-1}\Big((f/u'^{\f12})\circ u^{-1}\Big)(c)\\
&\quad+(1-c^2)^{\f32}\mathcal{L}_{\f32}(f)(c)\notag.
\end{align}
Then \eqref{lem:T-weighted-L2-key} follows from \eqref{est:Hilbert} and \eqref{T3}.
\end{proof}

\begin{remark}
In the above decomposition of $\mathcal{T}(c)$,  we choose the part \eqref{def:L1.5} rather than a nature one, as seen in $I_1$ in decomposition \eqref{dec:T(bfc)}. The reason for this choice is that the estimates for $(1-c^2)^{\f32}\mathcal{L}_{\f32}(f)(c)$ and the remaining terms $-\mathcal{H}\Big((f/u'^{\f12})\circ u^{-1}\Big)(c)$, etc., are exactly the same. This results in the weighted estimate for $\mathcal{T}$ in \eqref{lem:T-weighted-L2-key} being optimal.
\end{remark}

\textbf{Step 2}. $k=0$. It holds that
\begin{align}\label{cor:T-weighted-L2-key}
\|(1-c^2)^{\f32-i}\mathcal{T}(f)(c)\|_{L^2(-1,1)}\leq C\al\|f\|_{L^2_{1/(u')^{2i}}(\mathbb{R})},\quad i=0,1,2.
\end{align}

\begin{proof}[Proof of Step 2.]
The case of $i=0$ follows from  \eqref{lem:T-weighted-L2-key}.
We observe by Lemma \ref{lem:simple-useful-equ} that
\begin{align}
\label{est:u(z)^2-c^2}& u'=(1-c^2)-(u^2-c^2),\\
&|u(z)^2-c^2|\leq 2|u(z)-c|\leq 2|u(y)-c|,\quad \text{for}\;y_c\leq z\leq y\;\text{or}\;y\leq z\leq y_c.\label{est:u(z)^2-c^2-2}
\end{align}
If $f/u'\in L^2$, then  we have by \eqref{est:u(z)^2-c^2} that
\begin{align}
\mathcal{T}(f)=\mathcal{T}\big(u'(f/u')\big)
\label{decompose:calTf-1}&
=(1-c^2)\mathcal{T}\big(f/u'\big)-\mathcal{T}\Big((u^2-c^2)f/u'\Big),
\end{align}
which together with \eqref{est:u(z)^2-c^2-2} gives
\begin{align*}
(1-c^2)^{\f12}|\mathcal{T}(f)(c)|&\leq
(1-c^2)^{\f32}|\mathcal{T}\big(f/u'\big)(c)|+2(1-c^2)^{\f12}\mathcal{T}_{1,0}(f/u')(c).
\end{align*}
Therefore, \eqref{cor:T-weighted-L2-key}($i=1$) follows from \eqref{lem:T-weighted-L2-key} and
Lemma \ref{lemma:key-integral-estimate-1}(with $m=1$,\;$l=0$).

Applying \eqref{est:u(z)^2-c^2} again on  \eqref{decompose:calTf-1}, we have
\begin{align}
\notag\mathcal{T}(f)
\notag&=(1-c^2)\mathcal{T}\big(u'f/u'^2\big)-\mathcal{T}\Big(u'(u^2-c^2)f/u'^2\Big)\\
&=(1-c^2)^{2}\mathcal{T}(f/u'^2)
-2(1-c^2)\mathcal{T}\Big((u^2-c^2)f/u'^2\Big)+\mathcal{T}\Big((u^2-c^2)^2f/u'^2\Big),
\end{align}
which together with \eqref{est:u(z)^2-c^2-2} gives
\begin{align*}
&\quad (1-c^2)^{-\f12}|\mathcal{T}(f)(c)|\\
&\leq
(1-c^2)^{\f32}|\mathcal{T}\big(f/(u')^2\big)(c)|
+4(1-c^2)^{\f12}\mathcal{T}_{1,0}(f/(u')^2)(c)
+4(1-c^2)^{-\f12}\mathcal{T}_{0,0}(f/(u')^2)(c).
\end{align*}
Therefore, \eqref{cor:T-weighted-L2-key}($i=2$) follows from
\eqref{lem:T-weighted-L2-key} and Lemma \ref{lemma:key-integral-estimate-1}(with $m=1$, $l=0$; $m=l=0$).
\end{proof}

Before  proceeding with \textit{Step 3},  let us review some basic facts that have been previously used. By \eqref{good:phi1}, \eqref{fm:paGintf}, \eqref{phi_1(z)leq phi_1(y)}, \eqref{est:phi_1-ul} and \eqref{est:u'/u-c1}, we have
\begin{align*}
&\pa_{G}^j\Big(\int_{y_c}^yf(z,c)dz\Big)=\int_{y_c}^y\pa_{G}^jf(z,c)dz,\quad j\in\mathbb{Z}^+,\\
&\pa_c\Big(p.v.\int_{\mathbb{R}}f(c,y)dy\Big)
=p.v.\int_{\mathbb{R}}\pa_Gf(c,y)dy,\quad\text{if}\;
\lim_{y\to\infty}f(c,y)=0,\\
&\lim_{y\to\infty}\f{(1-c^2)^2\int_{y_c}^y\phi_1(z,c)f(z)dz}{(u(y)-c)^2\phi_1(y,c)^2}=0,\quad f\in L^2(\mathbb{R}),\\
&\bigg|\pa_{G}^j\left(\left(\f{1-c^2}{u(y)-c}\right)^2\right)\bigg|\leq
C\Big(\f{|u(y)-c|}{1-c^2}\Big)^{j-2},\quad j\in\mathbb{Z}^+,\\ &\Big|\f{\pa_G^j\phi_1(y,c)}{\phi_1(y,c)}\Big|
\leq C\Big(\f{|u(y)-c|}{1-c^2}\Big)^{j}\langle\al|y-y_c|\rangle^{j},\quad j=1,2,3.
\end{align*}

\textbf{Step 3}. $k=1,2,3$ and $i=0, 1,2$.

\begin{proof}[Proof of Lemma \ref{lem:key-T}]
Taking $\pa_c$ on $(1-c^2)^2\mathcal{T}(c)$, we  get for $f\in H^1(\mathbb{R})$,
\begin{align}
\notag&\pa_c\left(p.v.\int_{\mathbb{R}}\f{(1-c^2)^2\int_{y_c}^y\phi_1(z,c)f(z)dz}{(u(y)-c)^2\phi_1(y,c)^2}dy\right)\\
\notag&=p.v.\int_{\mathbb{R}}\pa_{G}\Big(\f{(1-c^2)^2\int_{y_c}^y\phi_1(z,c)f(z)dz}{(u(y)-c)^2\phi_1(y,c)^2}\Big)dy\\
\notag&=\int_{\mathbb{R}}\Big(\int_{y_c}^y\pa_{G}\phi_1(z,c)\cdot f(z)dz\Big)\cdot\f{(1-c^2)^2}{(u(y)-c)^2}\cdot\phi_1(y,c)^{-2}dy\\
\notag&\quad+\int_{\mathbb{R}}\Big(\int_{y_c}^y\phi_1(z,c) f(z)dz\Big)\cdot\pa_G\Big(\f{(1-c^2)^2}{(u(y)-c)^2}\Big)\cdot\phi_1(y,c)^{-2}dy\\
\notag&\quad+\int_{\mathbb{R}}\Big(\int_{y_c}^y\pa_{G}\phi_1(z,c)\cdot f(z)dz\Big)\cdot\f{(1-c^2)^2}{(u(y)-c)^2}\cdot\pa_G(\phi_1(y,c)^{-2})dy\\
\label{GoodT1}&\quad+ (1-c^2)^{-1}\int_{\mathbb{R}}\Big(\int_{y_c}^y\phi_1(z,c)\cdot f'(z)dz\Big)\cdot\f{(1-c^2)^2}{(u(y)-c)^2}\cdot\phi_1(y,c)^{-2}dy.
\end{align}
In general, for  $k=1,2,3$ and assuming $f\in H^k(\mathbb{R})$,  we have the derivative formulas as follows
\begin{align}
\notag&\pa_c^k\Big((1-c^2)^2\mathcal{T}(f)(c)\Big)\\
&\notag=p.v.\int_{\mathbb{R}}\pa_{G}^k\Big((\int_{y_c}^y\phi_1(z,c)f(z)dz)\cdot\f{(1-c^2)^2}{(u(y)-c)^2}\cdot\phi_1(y,c)^{-2}\Big)dy\\
\notag&=\small\sum_{j_1=0}^{k-1}\sum_{j_2=0}^{k-j_1} \sum_{j_3=0}^{k-j_1-j_2} C_k^{j_1}C_{k-j_1}^{j_2}C_{k-j_1-j_2}^{j_3}\times\\
\notag&\qquad\int_{\mathbb{R}} \int_{y_c}^y\f{f^{(j_1)}(z)}{(1-c^2)^{j_1}}dz\cdot\pa_{G}^{j_2}\phi_1(z,c)
\cdot\pa_{G}^{j_3}\Big(\f{(1-c^2)^2}{(u(y)-c)^2}\Big)\cdot\pa_G^{k-j_1-j_2-j_3}\big(\phi_1(y,c)^{-2}\big)dy\\
&\quad\label{fm:pa_kT}+ (1-c^2)^{-k}p.v.\int_{\mathbb{R}}\f{\int_{y_c}^y\phi_1(z,c)f^{(k)}(z)dz}{\phi(y,c)^2}dy.
\end{align}

Now we  prove the lemma  by the induction on $k$. The case of $k=0$ follows from \eqref{cor:T-weighted-L2-key}.\smallskip

Consider $k=1$ and $f/(u')^i\in H^1(\mathbb{R})$, $i=0,1,2$.  Thanks to \eqref{GoodT1} and  \eqref{phi_1(z)leq phi_1(y)}, we have
 \begin{align}\label{est:pac(1-c2)2T}
\notag \Big|\pa_c\big((1-c^2)^2\mathcal{T}(f)(c)\big)-(1-c^2)\mathcal{T}(f^{'})(c)\Big|
 &\lesssim \int_{\mathbb{R}}\f{1-c^2}{|u(z)-c|}
 \cdot\f{\langle\al|y-y_c|\rangle}{\phi_1(y,c)}\cdot
 \Big|\int_{y_c}^y|f(z)|dz\Big|dy\\
 &= C(1-c^2) \mathcal{T}_{1,1}(f)(c).
 \end{align}
Decomposing $\mathcal{T}_{1,1}(f)$($i=1,2$) similarly as \eqref{decompose:calTf-1}, we obtain by Lemma \ref{lemma:key-integral-estimate-1} that
\beq\notag
\mathcal{T}_{1,1}(f)(c)
\left\{
\begin{aligned}\notag
&=\al^{-1}(1-c^2)^{-\f12}\mathcal{L}^2,\;i=0,\\
&\leq 2\mathcal{T}_{0,1}(f/u')(c)+(1-c^2)
 \mathcal{T}_{1,1}(f/u')(c)=\al^{-1}(1-c^2)^{\f12}\mathcal{L}^2,\;i=1,2,
\end{aligned}
\right.
\eeq
which gives
\begin{align}\label{est:T11}
\mathcal{T}_{1,1}(f)(c)
=\al^{-1}(1-c^2)^{-\f12+\min\{i,1\}}\mathcal{L}^2.
\end{align}
 Thanks to  \eqref{est:pac(1-c2)2T}, Lemma \ref{lem:key-T} for  $k=0$ and  \eqref{est:T11}, we have
 \begin{align*}
 |\pa_c\mathcal{T}(f)(c)|
 &\leq 4c(1-c^2)^{-1}|\mathcal{T}(f)(c)|+(1-c^2)^{-1}|\mathcal{T}(f')(c)|+C
 (1-c^2)^{-1}|\mathcal{T}_{1,1}(f)(c)|\\
 &=\al(1-c^2)^{-\f52+i}\mathcal{L}^2+\al^{-1}(1-c^2)^{-\f32+\min\{i,1\}}
 \mathcal{L}^2=\al(1-c^2)^{-\f52+i}\mathcal{L}^2,
 \end{align*}
which gives Lemma \ref{lem:key-T} for $k=1$, $i=0,1,2$.

Consider $k=2$ and $f/(u')^i\in H^2(\mathbb{R})$, $i=0,1,2$. Analogous to \eqref{est:pac(1-c2)2T} and  \eqref{est:T11}, using \eqref{fm:pa_kT}($k=2$),  \eqref{phi_1(z)leq phi_1(y)} and  Lemma \ref{lemma:key-integral-estimate-1}, 
 we have
\begin{align*}
\Big|\pa_c^2\big((1-c^2)^2\mathcal{T}(f)(c)\big)-\mathcal{T}(f^{''})(c)\Big|
 &\lesssim \mathcal{T}_{0,2}(f)(c)+\mathcal{T}_{1,1}(f')(c)\\
&=\al^{-2}(1-c^2)^{\f12}
+\al^{-1}(1-c^2)^{-\f12+\min\{i,1\}}\mathcal{L}^2\\
&=\al^{-1}(1-c^2)^{-\f12+\min\{i,1\}}\mathcal{L}^2.
 \end{align*}
 Therefore, using the estimates for $k=0,1$, we obtain
 \begin{align*}
 |\pa_c^2\mathcal{T}(f)(c)|
 &\lesssim(1-c^2)^{-1}|\pa_c\mathcal{T}(f)(c)|+
 (1-c^2)^{-2}|\mathcal{T}(f)(c)|+(1-c^2)^{-2}|\mathcal{T}(f'')(c)|\\
 &\quad+\al^{-1}(1-c^2)^{-\f52+\min\{i,1\}}\mathcal{L}^2\\
 &=\al(1-c^2)^{-\f72+i}\mathcal{L}^2+\al^{-1}(1-c^2)^{-\f52+\min\{i,1\}}
 \mathcal{L}^2=\al(1-c^2)^{-\f72+i}\mathcal{L}^2,
 \end{align*}
which gives Lemma \ref{lem:key-T} for $k=2$, $i=0,1,2$.

Consider $k=3$ and $f/(u')^i\in H^2(\mathbb{R})$, $i=0,1,2$. Analogous to \eqref{est:pac(1-c2)2T} and  \eqref{est:T11}, using \eqref{fm:pa_kT}($k=3$),  \eqref{phi_1(z)leq phi_1(y)} and  Lemma \ref{lemma:key-integral-estimate-1}, 
 we have
\begin{align*}
 &\quad\Big|\pa_c^3\big((1-c^2)^2\mathcal{T}(f)(c)\big)-(1-c^2)^{-1}\mathcal{T}(f^{'''})(c)\Big|\\
 &\lesssim
 (1-c^2)^{-1}\mathcal{T}_{1,1}(f'')(c)+(1-c^2)^{-1}\mathcal{T}_{0,2}(f')(c)
 +(1-c^2)^{-1}\mathcal{T}_{0,3}(f)(c)\\
 &=\al^{-1}(1-c^2)^{-\f32+\min\{i,1\}}\mathcal{L}^2.
 \end{align*}
  Therefore, using Lemma \ref{lem:key-T} for $k=0,1,2$, we obtain
 \begin{align*}
 |\pa_c^3\mathcal{T}(f)(c)|
 &\lesssim(1-c^2)^{-1}|\pa_c^2\mathcal{T}(f)(c)|+(1-c^2)^{-2}|\pa_c\mathcal{T}(f)(c)|
 +(1-c^2)^{-2}|\mathcal{T}(f)(c)|\\
 &\quad+(1-c^2)^{-3}|\mathcal{T}(f'')(c)|+\al^{-1}(1-c^2)^{-\f72+\min\{i,1\}}\\
 &=\al(1-c^2)^{-\f92+i}\mathcal{L}^2+\al^{-1}(1-c^2)^{-\f72+\min\{i,1\}}\mathcal{L}^2
 \mathcal{L}^2=\al(1-c^2)^{-\f92+i}\mathcal{L}^2,
 \end{align*}
 which gives Lemma \ref{lem:key-T} for $k=3$, $i=0,1,2$.
\end{proof}

\subsection{Estimates of the kernel when $\al \geq 2$.} 
Recall the kernel $\mathcal{K}$ introduced in 
Corollary \ref{lem:dual}
\begin{align*}
\mathcal{K}(f,g)(c,\al)&=\f{(1-c^2)^2\Lambda_1(f)(c,\al)\Lambda_2(g)(c,\al)}{A(c,\al)^2+4\pi^2c^2},\quad c\in(-1,1),\;\al\geq 1,
\end{align*}
with
\begin{align*}
\Lambda_1(f)&=2c\mathcal{T}(f)-(1-c^2)^{-2}A(c,\al)f(y_c),\quad
\Lambda_2(g)=\mathcal{T}(u''g)+(1-c^2)^{-1}A(c,\al)g(y_c), \end{align*}
and $A$ defined  in \eqref{def:A} and $\mathcal{T}$  defined in \eqref{def:T(h)-(-1,1)}.

From Remark \ref{rmk:K},  we know that $\mathcal{K}(c)$ has no singularities on $(-1,1)$ when $\al\geq 2$. Consequently, we can establish the following weighted bilinear estimates.
\begin{proposition}\label{Prop:K-1.}
Let $\al\geq 2$, $k=0,1,2$. 
If $f\in H^k_{1/(u')^{2k}}(\mathbb{R})$, $g\in H^1(\mathbb{R})$, then
\begin{align}
\label{2K1}&\|(\pa_c^k\mathcal{K})(f,g)(c)\|_{L^1_c}\leq C\|f\|_{H^k_{1/(u')^{2k}}}\|g\|_{H^k},
\end{align}
where  $C$ is independent of $\al$. Moreover, if $f\in H^1 _{1/(u')^2}(\mathbb{R})$, $g\in H^1(\mathbb{R})$, then
\begin{align}
\label{BoundaryK2}&
\lim_{c\to\pm 1}\mathcal{K}(f,g)(c,\al)=0.
\end{align}

\end{proposition}
Before  proceeding with the proof, we  summarize the estimates we have obtained in following two corollarys. We conclude by Lemma \ref{lem:A} and Lemma \ref{lem:pa_cA} as follows.
\begin{Corollary}\label{cor:1A}
Let  $\al\geq 2$, $c\in(-1,1)$,  $k=0,1,2$. It holds that
\begin{align}
\notag
&A(c,\al)=\al\mathcal{L}^{\infty},\;\;
\pa_cA(c,\al)=\mathcal{L}^{\infty},\;\;
\pa_c^2A(c,\al)=\al^{-1}(1-c^2)^{-1}\mathcal{L}^{\infty},\\
\label{1K2''}&\pa_c^k\left(\f{(1-c^2)^2}{A(c,\al)^2+4\pi^2c^2}\right)=\al^{-2}(1-c^2)^{2-k}\mathcal{L}^{\infty}.
\end{align}
\end{Corollary}
Thus, together with Lemma \ref{lem:key-T} and Lemma \ref{lem:weighted-norm-derivative},  we have

\begin{Corollary}\label{cor:Lambda}
Let $\al\geq 1$, $c\in(-1,1)$, $k=0,1,2,3$, $i=0,1,2$, $s=0,1,2$. If  $f\in H^k_{1/(u')^{2i}}(\mathbb{R})$, $g\in H^s(\mathbb{R})$, then it holds that
\begin{align}
\label{2K2}&\pa_c^k\Lambda_1(f)(c)=\al(1-c^2)^{-\f32-k+i}\mathcal{L}^2,\\
\label{3K2}&\pa_c^s\Lambda_2(g)(c)=\al(1-c^2)^{-\f12-s}\mathcal{L}^2.
\end{align}

\end{Corollary}

\begin{proof}[Proof of Proposition \ref{Prop:K-1.}.]
We first prove \eqref{2K1}.
Consider $k=0$,  $f,g\in L^2(\mathbb{R})$.
Thanks to \eqref{1K2''}($k=0$), \eqref{2K2}($k=i=0$) and \eqref{3K2}($s=0$), we obtain
\begin{align*}
\mathcal{K}(c)=\f{(1-c^2)^2\Lambda_1\Lambda_2}{A^2+4\pi^2c^2}
=\al^{-2}(1-c^2)^{2}\mathcal{L}^{\infty}\cdot \al(1-c^2)^{-\f32}\mathcal{L}^2\cdot \al(1-c^2)^{-\f12}\mathcal{L}^2
=\mathcal{L}^1,
\end{align*}
 which gives \eqref{2K1} for $k=0$.
Consider $k=1$,  $f\in H^1_{1/(u')^2}(\mathbb{R}), g\in H^1(\mathbb{R})$.
 Thanks to \eqref{1K2''}($k=0,1$), \eqref{2K2}($k=0,1,i=0,1$) and  \eqref{3K2}($s=0,1$),   we obtain
\begin{align}\label{fm:pa_cK}
\pa_c\mathcal{K}(c)& =\pa_c\left(\f{(1-c^2)^2}{A^2+4\pi^2c^2}\right)\Lambda_1(f)\Lambda_2(g)
+\f{(1-c^2)^2}{A^2+4\pi^2c^2}\pa_c\Lambda_1(f)\Lambda_2(g)
=\mathcal{L}^1,
\end{align}
 which gives \eqref{2K1} for $k=1$.
Similarly, we also have
\begin{align}\label{est:1-c^2L1}
\mathcal{K}(c)
=(1-c^2)\mathcal{L}^1.
\end{align}
Consider $k=2$,  $f\in H^2_{1/(u')^4}(\mathbb{R}), g\in H^2(\mathbb{R})$.
Thanks to \eqref{1K2''}($k=0,1,2$), \eqref{2K2}($k=0,1,2,i=0,1,2$) and  \eqref{3K2}($s=0,1,2$),   we obtain
\begin{align}
\notag\pa_c^2\mathcal{K}(c)&=\pa_c^2\big(\f{(1-c^2)^2}{A^2+4\pi^2c^2}\big)\Lambda_1(h)\Lambda_2(g)
+2\pa_c\big(\f{(1-c^2)^2}{A^2+4\pi^2c^2}\big)\pa_c\Lambda_1(h)\Lambda_2(g)\\
\label{fm:pa_c^2K}&\quad+2\pa_c\big(\f{(1-c^2)^2}{A^2+4\pi^2c^2}\big)\pa_c\Lambda_2(g)\Lambda_1(h)+
2\f{(1-c^2)^2\pa_c\Lambda_1(h)\pa_c\Lambda_2(g)}{A^2+4\pi^2c^2}\\
\notag&\quad+\f{(1-c^2)^2\pa_c^2\Lambda_1(h)\Lambda_2(g)}{A^2+4\pi^2c^2}
+\f{\Lambda_1(h)(1-c^2)^2\pa_c^2\Lambda_2(g)}{A^2+4\pi^2c^2}=\mathcal{L}^1,
\end{align}
 which gives \eqref{2K1} for $k=2$.

We obtain the continuity of $\mathcal{K}$ by \eqref{fm:pa_cK}. Then \eqref{BoundaryK2} follows from \eqref{est:1-c^2L1}.
\end{proof}

\subsection{Estimates of the kernels in presence of embedding eigenvalue.}
From Remark \ref{rmk:K}, we are aware that $\mathcal{K}(c)$ exhibits a singularity of $\f{1}{c}$ at the embedding eigenvalue $c=0$ when $\al=1$. The goal of this section is to provide estimates for the kernels near and away from  $c=0$.

For $\tilde{A}$ and $\tilde{\tilde{A}}$ defined in \eqref{def:tA}, we conclude by Lemma \ref{lem:A} and Lemma  \ref{lem:pa_cA} that

\begin{Corollary}\label{Cor:1A'}
Let  $\al=1$, $k=0,1,2$. It holds for $c\in(-1,1)$ that
\begin{align*}
\begin{split}
&\tilde{A}(c)^2+4\pi^2\geq 1,\\
|\pa_c^k\tilde{A}(c)|\leq C(1-&c^2)^{-k},\quad\;|\pa_c^k\tilde{\tilde{A}}(c)|\leq C(1-c^2)^{-k-1}.
\end{split}
\end{align*}
Moreover, it holds that
\begin{align*}
\pa_c^k\left(
\f{\chi_0(c)(1-c^2)^2}{\tilde{A}(c)^2+4\pi^2}\right)=\mathcal{L}^{\infty}
\left(-\f12,\f12\right).
\end{align*}
\end{Corollary}
It follows from Lemma \ref{lem:key-T}($i=0$) that 
\begin{Corollary}\label{cor:pa_c^3T}
Let $\al\geq 1$,  $c\in [-\f12,\f12]$, $k=0,1,2,3$. If  $f\in H^k(\mathbb{R})$,  then it holds that
\begin{align*}
&\|\pa_c^k\mathcal{T}(f)(c)\|_{L^2(-\f12,\f12)}\leq C\al\|f\|_{H^k(\mathbb{R})},
\end{align*}
where $C$ is independent of $\al$.
\end{Corollary}
For $\tilde{\mathcal{T}}$  defined in \eqref{def:tT-1}, we have the following.
\begin{Corollary}\label{lem:1T'}
Let   $\al\geq 1$,  $k=0,1,2$, $c\in(-\f12,\f12)$. If  $f,g\in H^{k+1}(\mathbb{R})$,  then it holds that
\begin{align}
\label{2K2'''}&\|\pa_c^k\widetilde{\mathcal{T}}(f)(c)\|_{L^2_c(-\f12,\f12)}\leq C \al\|f\|_{ H^{k+1}(\mathbb{R})},\quad \|\pa_c^k\tilde{\mathcal{T}}(u''g)(c)\|_{L^2_c(-\f12,\f12)}\leq C \al\|g\|_{ H^{k+1}(\mathbb{R})}.
\end{align}
\end{Corollary}
\begin{proof}
 The second inequality follows from the first one with $f=u''g$.  To prove the first inequality, we write by the definition \eqref{def:tT-1} that
\begin{align}\label{fm:packtT}
\pa_c^k\widetilde{\mathcal{T}}(f)(c)=\int_0^1s^k(\pa_c^{k+1}\mathcal{T})(f)(sc)ds.
\end{align}
Therefore, thanks to Corollary \ref{cor:pa_c^3T} and Minkowoski inequality, we have by \eqref{fm:packtT} that
\begin{align*}
\|\pa_c^k\widetilde{\mathcal{T}}(f)(c)\|_{L^2_c(-\f12,\f12)}&\leq\int_0^1s^k\|(\pa_c^{k+1}\mathcal{T})(f)(sc)\|_{L^2_c(-\f12,\f12)}ds\\
&=\int_0^1s^{k-\f12}\|(\pa_c^{k+1}\mathcal{T})(f)(\tilde{c})\|_{L^2_{\tilde{c}}(-\f12s,\f12s)}ds
\leq C\al\|f\|_{H^{k+1}(\mathbb{R})}.
\end{align*}\end{proof}
\begin{definition}
Let $\al=1$.
We define the kernels near $c=0$  as
\begin{align}
\mathcal{K}_0(f,g)(c)
\label{fm:K-0}&:=\f{\chi_0(c)(1-c^2)^2\tilde{\Lambda}_1(f)(c)\Lambda_2(g)(c)}{\tilde{A}(c)^2+4\pi^2},
\quad\mathrm{supp}(\mathcal{K}_0)\subset(-\f12,\f12),
\end{align}
 and
\begin{align}\label{fm:tK-0}
\widetilde{\mathcal{K}}_{0}(g)(c)
&:=\f{\chi_0(c)(1-c^2)^2c\tilde{\tilde{A}}(c)^2\Lambda_2(g)(c)}{\tilde{A}(c)^2+4\pi^2},
\quad\mathrm{supp}(\widetilde{\mathcal{K}}_0)\subset(-\f12,\f12),
\end{align}
with $\tilde{\Lambda}_1$  defined in \eqref{recall:Lambda}. We define the kernel away from $c=0$  as
\begin{align}
\label{fm:K-11}&\mathcal{K}_1(f,g)(c)=\f{\chi_1(c)(1-c^2)^2\Lambda_1(f)(c)\Lambda_2(g)(c)}{A(c)^2+4\pi^2c^2},\quad \mathrm{supp}(\mathcal{K}_1)\subset(-1,1)/(-\f14,\f14),
\end{align}
with $\Lambda_1$, $\Lambda_2$ defined in \eqref{def:calLambda1}, \eqref{def:calLambda2}.
\end{definition}

\begin{proposition}\label{lem:K-1}
Let $\al=1$,  $k=0,1,2$. If $f\in H^{k+1}(\mathbb{R})$, $g\in H^{k}(\mathbb{R})$, then it holds that
\begin{align}
\label{est:K0}&\|(\pa_c^k\mathcal{K}_0)(f,g)(c)\|_{L^1(-\f12,\f12)}\leq C\|f\|_{H^{k+1}(\mathbb{R})}\|g\|_{H^k(\mathbb{R})},
\end{align}
and
\begin{align}
\label{est:tK0}&\|(\pa_c^k
\widetilde{\mathcal{K}}_0)(g)(c)\|_{L^1(-\f12,\f12)}\leq C\|g\|_{H^k(\mathbb{R})}.
\end{align}
If $f\in H^{k}_{1/(u')^{2k}}(\mathbb{R})$, $g\in H^{k}(\mathbb{R})$, then it holds that
\begin{align}\label{est:K1}
\|(\pa_c^k\mathcal{K}_1)(f,g)(c)\|_{L^1(-1,1)}\leq C\|f\|_{H^k_{1/(u')^{2k}}(\mathbb{R})}\|g\|_{H^k(\mathbb{R})}.
\end{align}
In particular, if $f\in H^{1}_{1/(u')^{2}}(\mathbb{R})$, $g\in H^{1}(\mathbb{R})$, then we have
\begin{align}\label{lim:K1}
\lim_{c\to \pm 1}\mathcal{K}_1(c)=0.
\end{align}
Let $\Lambda_2$ be defined in \eqref{def:calLambda2}, and $\tilde{\Lambda}_2$ be defined in \eqref{recall:tLambda2}. It holds that
\begin{align}\label{est:tLambda2}
\begin{split}
&\|(\pa_c^k\Lambda_2)(g)(c)\|_{L^2(-\f12,\f12)}
\leq C\|g\|_{H^{k}(\mathbb{R})},\\
&\|(\pa_c^k\tilde{\Lambda}_2)(g)(c)\|_{L^2(-\f12,\f12)}
\leq C\|g\|_{H^{k+1}(\mathbb{R})}.
\end{split}
\end{align}

\end{proposition}

\begin{proof}
We first prove \eqref{est:tLambda2}.
The first estimate  directly follows  from \eqref{3K2}. The second estimate follows from  Corollarys \ref{Cor:1A'} and \ref{lem:1T'}. In the same manner, we also have for $k=0,1,2$,
\begin{align}\label{est:tLambda1}
\|(\pa_c^k\tilde{\Lambda}_1)(f)(c)\|_{L^2(-\f12,\f12)}
\leq C\|f\|_{H^{k+1}(\mathbb{R})}.
\end{align}
\eqref{est:K0} follows from \eqref{est:tLambda2}, \eqref{est:tLambda1} and Corollary \ref{Cor:1A'}. \eqref{est:tK0} follows from \eqref{est:tLambda2} and Corollary \ref{Cor:1A'}. Since we notice by Lemma \ref{lem:A} that $\pa_c^k\big(\f{\chi_1(c)(1-c^2)^2}{A(c,1)^2+4\pi^2c^2}\big)=\mathcal{L}^{\infty}$,
the proof of  \eqref{est:K1} and \eqref{lim:K1} is  similar as  the proof of Proposition \ref{Prop:K-1.}.
\end{proof}

\section{Proof of main results}
The first two subsections are devoted  to proving Theorem \ref{Thm:main}. The last section is devoted to proving Theorem \ref{Thm:main0}.

\subsection{Linear inviscid damping for modes $\al\geq 2$. }
\begin{proof}[Proof of \eqref{Thm:main-al-geq2} and \eqref{Thm:main-al-geq22}.] The proof uses integration by parts and Proposition \ref{Prop:K-1.}.
Thanks to $\omega=-\psi''+\psi$ and  Corollary \ref{lem:dual}, we have by \eqref{fm:psi2} that
\begin{align*}
&\al^2\|\psi(t)\|_{L^2}^2
+\|\pa_y\psi(t)\|_{L^2}^2\\
&=\Big|\int_{\mathbb{R}}\psi\bar{\omega} dy\Big|=\bigg|\int_{\mathbb{R}}\int_{-1}^1e^{-i\al c t}\mu(c)\Gamma(y,c)\omega(y)dcdy\bigg|\\
&=
\bigg|\int_{-1}^1e^{-i \al c t}\mathcal{K}\big(\omega_0,\psi\big)(c,\al)dc\bigg|.
\end{align*}
Hence, by Proposition \ref{Prop:K-1.}($k=0$), we have
\begin{align*}
\al^2\|\psi(t)\|_{L^2}^2
+\|\pa_y\psi(t)\|_{L^2}^2&\leq
\|\mathcal{K}\big(\omega_0,\psi\big)(c)\|_{L^1(-1,1)}
\leq C\|\omega_0\|_{L^2}\|\psi\|_{L^2}.
\end{align*}
If $\omega_0\in H^1_{1/(u')^2}$, then  we integrate by parts in $c$ and use Proposition \ref{Prop:K-1.}($k=1$) to get
\begin{align*}
\al^2\|\psi(t)\|_{L^2}^2
+\|\pa_y\psi(t)\|_{L^2}^2&\leq (\al t)^{-1}
\|(\pa_c\mathcal{K})\big(\omega_0,\psi\big)(c)\|_{L^1(-1,1)}\\
&\leq C\al^{-1}t^{-1}\|\omega_0\|_{H^1_{1/(u')^2},\al}\|\psi\|_{H^1}.
\end{align*}
Thus, \eqref{Thm:main-al-geq2} follows.

Let $\varphi\in L^2$, $g''-g=\varphi$. Thanks to  \eqref{fm:psi2} and Corollary \ref{lem:dual}, we have
\begin{align*}
&\quad\|\psi(t)\|_{L^2}
=\sup_{\|\varphi\|_{L^2}=1}=
\Big|\int_{\mathbb{R}}\psi \varphi dy\Big|=
\sup_{\|\varphi\|_{L^2}=1}\Big|\int_{-1}^1e^{-i \al c t}\mathcal{K}\big(\omega_0,g\big)(c,\al)dc\Big|.
\end{align*}
Hence, by Proposition \ref{Prop:K-1.}($k=0$), we have
\begin{align*}
\|\psi(t)\|_{L^2}&\leq \sup_{\|\varphi\|_{L^2}=1}
\|\mathcal{K}(\omega_0,g)(c)\|_{L^1(-1,1)}\leq
C\sup_{\|\varphi\|_{L^2}=1}
\|\omega\|_{L^2}\|g\|_{L^2}\leq C
\|\omega\|_{L^2}.
\end{align*}
Notice that if $\lim_{c\pm 1}\mathcal{K}(c)=0$, then it holds
\begin{align*}
\|\pa_c\mathcal{K}\|_{L^{\infty}_c(-1,1)}\leq C
\|\pa_c^2\mathcal{K}\|_{L^{1}_c(-1,1)}.
\end{align*}
If $\omega_0\in H^2_{1/(u')^4}$, then  we integrate by parts twice and use Proposition \ref{Prop:K-1.}($k=2$) to get
\begin{align*}
\|\psi(t)\|_{L^2}&\leq (\al t)^{-2}\sup_{\|\varphi\|_{L^2}=1}
\|(\pa_c^2\mathcal{K})(\omega_0,g)(c)\|_{L^1(-1,1)}\\
&\leq
C(\al t)^{-2}\sup_{\|\varphi\|_{L^2}=1}
\|\omega\|_{H^2_{1/(u')^4}}\|g\|_{H^2}\leq C(\al t)^{-2}\|\omega\|_{H^2_{1/(u')^4}}.
\end{align*}
Thus, \eqref{Thm:main-al-geq22} follows.
\end{proof}

\subsection{Dynamics for mode $\al=1$}

We recall by Lemma \ref{lem:dec-psi} that
\begin{align}
\notag\hat{\psi}(t,1,y)&=-\int_{-1}^1e^{-ict}
\f{\chi_1(c)\Lambda_1(\omega_0)(c)}{A(c)^2+4\pi^2c^2}\Gamma(y,c)dc-\int_{-1}^{1}e^{-ict}
\f{\chi_0(c)(1-c^2)^2\tilde{\Lambda}_1(c)}{\tilde{A}(c)^2+4\pi^2}\Gamma(y,c)dc\\
\notag&\quad+\f{a_0}{2\pi^2}\int_{-1}^{1}e^{-ict}
\f{\chi_0(c)(1-c^2)^2c\tilde{\tilde{A}}(c)^2}{\tilde{A}(c)^2+4\pi^2}\Gamma(y,c)dc\\
&\quad+a_0f_1(t,y)+\f{a_0+i\pi b_0}{2\pi i}\mathrm{sech}\,y ,\label{decom:psi'}
\end{align}
where $a_0$ is defined in \eqref{def:a0} and $f_1(t,y)$ is defined in \eqref{def:Psi1}.

\begin{proof}[Proof of \eqref{Thm:main-al-1} and \eqref{Thm:main-al-12}.]

Let $\hat{\omega}_0(1,\cdot)\in H^1(\mathbb{R})$, $a_0$ be defined in \eqref{def:a0}. We have by Hardy's inequality that
\begin{align}\label{est:a_0}
|a_0|=\Big|\int_{\mathbb{R}}\f{\hat{\omega}_0(1,y)-
\hat{\omega}_0(1,-y)}{\sinh y}dy\Big|\leq \|\hat{\omega}_0(1,\cdot)\|_{H^1}.
\end{align}
We take $f_1(t,y)$ as \eqref{def:Psi1}:
\begin{align*}
f_1(t,y)&=\f{1}{2\pi^2}\Big(\Psi(t,y)+S(t)+i\pi\Big)
\end{align*}
with
\begin{align*}
\Psi(t,y)&:= -p.v.\int_{-1}^1e^{-ict}
\f{\chi_0(c)(1-c^2)^2}{c}\Big(\Gamma(y,c)-\Gamma(y,0)\Big)dc,\\
S(t)&:=-i\int_{-1}^1\sin ct\cdot
\f{\chi_0(c)(1-c^2)^2}{c}dc.
\end{align*}
We denote
 \begin{align*}
\tilde{\psi}(t,y):=\psi(t,y)-a_0f_1(t,y)-\f{a_0+i\pi b_0}{2\pi i}\mathrm{sech}\,y ,\;\text{and}\;
\tilde{\omega}=-\pa_y^2\tilde{\psi}+\tilde{\psi}.
\end{align*}

To prove \eqref{Thm:main-al-1}, we consider $\omega_0\in H^2_{1/(u')^2}(\mathbb{R})$. Thanks to \eqref{decom:psi'} and Corollary \ref{lem:dual0}, we have
\begin{align*}
&\|\tilde{\psi}(t,y)\|_{H^1}^2=
\Big|\int_{\mathbb{R}}\tilde{\psi}(t,y)\bar{\tilde{\omega}}(t,y)dy\Big|\\
&=\bigg|
\int_{-1}^1e^{-i c t}\mathcal{K}_1(\omega_0,\bar{\tilde{\psi}})(c)dc
+\int_{-1}^1e^{-i c t}\mathcal{K}_0(\omega_0,\bar{\tilde{\psi}})(c)dc+
\f{a_0}{2\pi^2}\int_{-1}^1e^{-i c t}\widetilde{\mathcal{K}}_0(\tilde{\psi})(c)dc\bigg|.
\end{align*}
Therefore, by Proposition \ref{lem:K-1}($k=0$), we have
\begin{align*}
\|\tilde{\psi}(t,y)\|_{H^1}^2\lesssim
\|\omega_0\|_{L^2}\|\tilde{\psi}\|_{L^2}+
\|\omega_0\|_{H^1}\|\tilde{\psi}\|_{L^2}
+\|\omega_0\|_{H^1}\|\tilde{\psi}\|_{L^2}\leq C\|\omega_0\|_{H^1}\|\tilde{\psi}\|_{L^2}.
\end{align*}
By Proposition \ref{lem:K-1}($k=1$), we integrate by parts to obtain for $t>0$,
\begin{align*}
\|\tilde{\psi}(t,y)\|_{H^1}^2&\lesssim t^{-1}
\|\omega_0\|_{H^1_{1/(u')^2}}\|\tilde{\psi}\|_{H^1}
+t^{-1}
\|\omega_0\|_{H^2}\|\tilde{\psi}\|_{H^1}
+ t^{-1}
\|\omega_0\|_{H^1}\|\tilde{\psi}\|_{H^1}\\
&\leq Ct^{-1}\|\omega_0\|_{H^2_{1/(u')^2}}\|\tilde{\psi}\|_{H^1}.
\end{align*}
Combining the above estimates and \eqref{est:a_0}, we obtain \eqref{Thm:main-al-1}.

To prove \eqref{Thm:main-al-12}, we consider $\omega_0\in H^3_{1/(u')^4}(\mathbb{R})$, $L^2(\mathbb{R})\ni\varphi=g''-g$.
 Thanks to \eqref{decom:psi'} and Corollary \ref{lem:dual0}, we have
\begin{align*}
&\|\tilde{\psi}(t,y)\|_{L^2}=\sup_{\|\varphi\|_{L^2}=1}
\Big|\int_{\mathbb{R}}\tilde{\psi}(t,y)\varphi(y)dy\Big|\\
&=\sup_{\|\varphi\|_{L^2}=1}\bigg|
\int_{-1}^1e^{-i c t}\mathcal{K}_1(\omega_0,g)(c)dc
+\int_{-1}^1e^{-i c t}\mathcal{K}_0(\omega_0,g)(c)dc
+\f{a_0}{2\pi^2}\int_{-1}^1e^{-i c t}\widetilde{\mathcal{K}}_0(g)(c)dc\bigg|.
\end{align*}
Therefore, by Proposition \ref{lem:K-1}($k=0$), we have
\begin{align*}
\|\tilde{\psi}(t,y)\|_{L^2}\lesssim
\sup_{\|\varphi\|_{L^2}=1}\Big(\|\omega_0\|_{L^2}\|g\|_{L^2}+
\|\omega_0\|_{H^1}\|g\|_{L^2}
+\|\omega_0\|_{H^1}\|\tilde{\psi}\|_{L^2}\Big)\leq C\|\omega_0\|_{H^1}.
\end{align*}
By Proposition \ref{lem:K-1}($k=2$), we integrate by parts twice to obtain for $t>0$,
\begin{align*}
\|\tilde{\psi}(t,y)\|_{L^2}&\lesssim t^{-2}\sup_{\|\varphi\|_{L^2}=1}\Big(
\|\omega_0\|_{H^2_{1/(u')^4}}\|g\|_{H^2}
+t^{-1}
\|\omega_0\|_{H^3}\|g\|_{H^2}
+ t^{-1}
\|\omega_0\|_{H^1}\|g\|_{H^2}\Big)\\
&\leq Ct^{-1}\|\omega_0\|_{H^3_{1/(u')^4}}.
\end{align*}
Combining the above estimates and \eqref{est:a_0}, we obtain \eqref{Thm:main-al-12}.\smallskip

It remains to prove \eqref{lim:lambda1}. It suffices to prove the following
\begin{align}\label{est:S(t)}
&S(t)+i\pi=O(t^{-2}),\quad t\to\infty,\\
\label{est:Psi1}
&\|f_1(t,\cdot)\|_{H^1}=o(1),\quad t\to\infty,\\
&\|f_1(t,\cdot)\|_{L^2}=o(t^{-1}),\quad t\to\infty.\;\label{est:Psi0}
\end{align}
To prove \eqref{est:S(t)}, we deduce by the definition that
\begin{align*}
iS(t)=\int_{-1}^1\f{\sin ct}{c}dc-\int_{-1}^1c(2-c^2)\sin ct dc-\int_{-1}^1
\f{\chi_1(c)(1-c^2)^2\sin ct}{c}dc=S_1+S_2+S_3.
\end{align*}
Due to integration by parts, we have
\begin{align*}
S_1(t)&=\pi-2\int_{t}^{+\infty}\f{\sin x}{x}dx=\pi-\f{2\cos t}{t}-\f{2\sin t}{t^2}+O(t^{-3}),\\
S_2(t)&=\f{2\cos t}{t}+\f{2\sin t}{t^2}+\f{12\cos t}{t^3}-\f{12\sin t}{t^4},\\
 S_3(t)&=O(t^{-3}).
\end{align*}
Thus, \eqref{est:S(t)} follows. To prove \eqref{est:Psi1} and \eqref{est:Psi0}, we first give a quantitative Riemann-Lebesgue lemma. For $f(c)\in L^1(a,b)$, it holds that
\begin{align}\label{quant:RL}
\int_{a}^be^{-ict}f(c)dc=o(1)\|f\|_{L^1(a,b)},\quad t\to \infty.
\end{align}
For \eqref{est:Psi1}, similar as the treatment for $\tilde{\psi}(t,y)$, we obtain by the definition \eqref{def:Psi11} and  \eqref{dual-fm-g}(with $f=e^{-ict}\chi_0(1-c^2)^2$, $g=f_1$) that
\begin{align*}
&\quad\|f_1(t,\cdot)\|_{H^1}^2=\bigg|
\int_{-1}^1e^{-i c t}\chi_0(c)(1-c^2)^2\widetilde{\Lambda}_2(f_{1})(c)dc\bigg|.
\end{align*}
Thanks to \eqref{est:tLambda2}, we have
\begin{align*}
&\quad\|\chi_0(c)(1-c^2)^2\widetilde{\Lambda}_2(f_{1})(c)\|_{L^1(-1,1)}
\leq \|f_{1}(t,\cdot)\|_{H^1}.
\end{align*}
Thus, the desired estimate follows from \eqref{quant:RL}.
For \eqref{est:Psi0}, similar as the treatment for $\tilde{\psi}(t,y)$, we obtain by the definition \eqref{def:Psi11} and \eqref{dual-fm-g}(with $f=e^{-ict}\chi_0(1-c^2)^2$, $g''-g=\varphi$) that
\begin{align*}
\|f_1(t,\cdot)\|_{L^2}
&=\sup_{\|\varphi\|_{L^2}=1}\bigg|
\int_{-1}^1e^{-i c t}\chi_0(c)(1-c^2)^2\widetilde{\Lambda}_2(g)(c)dc\bigg|\\
&=t^{-1}\sup_{\|\varphi\|_{L^2}=1}\bigg|
\int_{-1}^1e^{-i c t}\pa_c\left(\chi_0(c)(1-c^2)^2\widetilde{\Lambda}_2(g)(c)
\right)dc\bigg|.
\end{align*}
Thanks to \eqref{est:tLambda2}, we have
\begin{align*}
&\quad\Big\|\pa_c\left(\chi_0(c)(1-c^2)^2\widetilde{\Lambda}_2(g)(c)\Big)\right\|_{L^1(-1,1)}
\leq C\|g\|_{H^2}\leq C\|\varphi\|_{L^2}.
\end{align*}
Thus, the desired estimate follows from \eqref{quant:RL}.

\subsection{Proof of Theorem \ref{Thm:main0}.}
  Using condition (1), we have $\hat{\omega}_0(t,0,y)=\hat{\psi}_0(t,0,y)\equiv0$. Since $\textbf{V}=(\pa_y\hat{\psi}, -i\al\hat{\psi})$, $\hat{\psi}(t,-\al,y)=\bar{\hat{\psi}}(t,\al,y)$,
$\hat{\omega}_0(-\al,y)=\bar{\hat{\omega}}_0(\al,y)$, we have
by  Plancherel's  formula that
\begin{align*}
\|\textbf{V}(t)\|_{L^2}^2&=\sum_{\al\in \mathbb{Z}}\al^2\|\hat{\psi}(t,\al,\cdot)\|^2_{L^2_y}
+\|\pa_y\hat{\psi}(t,\al,\cdot)\|^2_{L^2_y}\\
&=2\sum_{\al\in \mathbb{Z}^+}\al^2\|\hat{\psi}(t,\al,\cdot)\|^2_{L^2_y}
+\|\pa_y\hat{\psi}(t,\al,\cdot)\|^2_{L^2_y},\\
\|V^2(t)\|_{L^2}^2&=\sum_{\al\in \mathbb{Z}}\al^2\|\hat{\psi}(t,\al,\cdot)\|^2_{L^2_y}
=2\sum_{\al\in \mathbb{Z}^+}\al^2\|\hat{\psi}(t,\al,\cdot)\|^2_{L^2_y}.
\end{align*}
If (1), (2), (3) hold, then Theorem \ref{Thm:main} gives
\begin{align*}
&\|\hat{\psi}(t,1,y)\|_{H^1_y}\leq C \langle t\rangle^{-1}\|\hat{\omega}_0(1,y)\|_{H^2_{y,1/(u')^2}}
,\\
&\|\hat{\psi}(t,1,y)\|_{L^2_y}\leq C \langle t\rangle^{-2}\|\hat{\omega}_0(1,y)\|_{H^3_{y,1/(u')^4}}
,\\
&\|\pa_y\hat{\psi}(t,\al,y)\|_{L^2_y}
+\al\|\hat{\psi}(t,\al,y)\|_{L^2_y}\leq C\al^{-1} \langle t\rangle^{-1}\|\hat{\omega}_0(\al,y)\|_{H^1_{y,1/(u')^2}},\;\al\geq 2,\\
&\|\hat{\psi}(t,\al,y)\|_{L^2_y}\leq C\al^{-2} \langle t\rangle^{-2}\|\hat{\omega}_0(\al,y)\|_{H^2_{y,1/(u')^4}},\;\al\geq 2.
\end{align*}
Summing up, we have
\begin{align*}
\|\textbf{V}(t)\|_{L^2}^2&\leq
C\langle t\rangle ^{-2}\sum_{\al\in \mathbb{Z}^+}\al^{-2}\|\hat{\omega}_0(\al,\cdot)\|^2_{H^2_{y,1/(u')^2}}
\leq C\langle t\rangle ^{-2}\|\omega_0\|^2_{H^{-1}_{x}H^2_{y,1/(u')^2}},\\
\|V^2(t)\|_{L^2}^2&
\leq
C\langle t\rangle ^{-4}\sum_{\al\in \mathbb{Z}^+}\al^{-2}\|\hat{\omega}_0(\al,\cdot)\|^2_{H^3_{y,1/(u')^4}}
\leq C\langle t\rangle ^{-4}\|\omega_0\|^2_{H^{-1}_{x}H^3_{y,1/(u')^4}}.
\end{align*}
The above estimates give \eqref{Thm:main-gernal}.
\end{proof}

\section*{Acknowledgements}
This work was done when the first author
was visiting the School of Mathematical Science, Peking University. The first author appreciates the warm hospitality.
The authors acknowledge the valuable assistance of Y. Xie in creating the graphs for this paper.
S. Ren is partially supported by NSF of China under Grant 12171010.  Z. Zhang is partially supported by NSF of China under Grants 12171010 and 12288101.


\end{document}